\documentclass[a4paper, 11pt, twoside, leqno, openright, slashbox]{amsart}

\usepackage[top=0.9in, bottom=0.9in, left=0.95in, right=0.95in]{geometry}

\usepackage{amssymb}
\usepackage{amsthm}
\usepackage{amscd}
\usepackage{mathrsfs}
\usepackage{amsopn}
\usepackage{manfnt}

\usepackage{latexsym}
\usepackage{amsmath}
\usepackage{epic,eepic}
\usepackage{pifont}
\usepackage{lmodern}
\usepackage{array}	
\usepackage{lscape}

\usepackage{graphicx, color}
\usepackage{tikz}
\usepackage{tikz-cd}
\usepackage{pgfplots}
\usepackage {diagbox}

\usetikzlibrary{intersections}
\usetikzlibrary{positioning}
\usetikzlibrary{arrows}
\usetikzlibrary{patterns}
\usetikzlibrary{decorations.pathreplacing}

\usepackage{xcolor}
\usetikzlibrary{arrows.meta, calc, backgrounds, fit}

\usepackage{url}
\usepackage[all]{xy}
\usepackage{paralist}

\usepackage{shuffle}
\usepackage{here}

\usepackage{mathtools}

\definecolor{e-mail}{rgb}{0,.40,.80}
\definecolor{reference}{rgb}{.20,.60,.22}
\definecolor{mrnumber}{rgb}{.80,.40,0}
\definecolor{citation}{rgb}{0,.40,.80}
\usepackage[breaklinks=true,colorlinks=true,linkcolor=reference, citecolor=citation, urlcolor=e-mail]{hyperref}

\newcommand{\la}{\langle}
\newcommand{\ra}{\rangle}

\def\bG{{\mathbb G}}
\def\bP{{\mathbb P}}

\def\bZ{{\mathbb Z}}
\def\bQ{{\mathbb Q}}

\def\bC{{\mathbb C}}

\def\an{{\mathrm{an}}}

\def\log{{\mathrm{log}}}

\def\CalO{{\mathcal O}}

\def\Conf{{\mbox{Conf}}}

\def\Aut{{\mbox{Aut}}}

\def\Gal{{\mbox{Gal}}}

\def\Spec{{\mbox{Spec}}}

\def\ellet{{\ell\text{-\'et}}}

\def\Coeff{\mathtt{Coeff}}

\def\C{{\mathbb C}}
\def\Z{{\mathbb Z}}

\def\Q{{\mathbb Q}}
\def\Gal{{\mathrm{Gal}}}

\def\ab{\mathrm{ab}}

\def\la{\langle}
\def\ra{\rangle}

\def\Spec{\mathrm{Spec}\,}

\def\ovec#1{\overrightarrow{#1}}

\newtheorem{dfn}{Definition}[section]

\newtheorem{thm}[dfn]{Theorem}

\newtheorem{cor}[dfn]{Corollary}
\newtheorem{rem}[dfn]{Remark}

\theoremstyle{remark}
\theoremstyle{definition}

\newtheorem*{theorem*}{Theorem}

\newtheorem*{definition*}{Definition}

\newtheorem*{Remark*}{Remark}

\usepackage{fancyhdr} 
\definecolor{eqblueborder}{RGB}{40, 90, 185}
\definecolor{eqbluefill}{RGB}{218, 230, 255}
\definecolor{geoborder}{RGB}{170, 50, 28}
\definecolor{geofill}{RGB}{255, 241, 230}
\definecolor{geotext}{RGB}{145, 38, 18}
\definecolor{bgcard}{RGB}{240, 244, 252}
\definecolor{arrowblue}{RGB}{18, 52, 140}
\definecolor{arrowred}{RGB}{158, 28, 28}
\definecolor{rulecol}{RGB}{170, 175, 190}

\definecolor{arrowgold}{RGB}{220, 170, 0}
\definecolor{eqgoldfill}{RGB}{255, 250, 230}

\definecolor{cBlue}{RGB}{52, 100, 182}
\definecolor{cBlueFill}{RGB}{218, 232, 252}
\definecolor{cDeepBlue}{RGB}{26, 58, 126}
\definecolor{cHubFill}{RGB}{195, 214, 246}
\definecolor{cGeoText}{RGB}{68, 76, 106}
\definecolor{cSep}{RGB}{140, 177, 226}
\definecolor{cGold}{RGB}{180, 130, 0}
\definecolor{cRed}{RGB}{168, 40, 40}
\definecolor{cLeg}{RGB}{185, 185, 185}

\tikzset{
     injarrow/.style={{Hooks[right, line cap=round]}-{Stealth[length=5pt]}},
     surjarrow/.style={-{Stealth[length=4pt, sep=-0.5pt]Stealth[length=5pt]}},
     morph/.style={-{Stealth[length=5pt]}},
     dasharrow/.style={-{Stealth[length=4.5pt]}, dashed},
   }

\makeatletter

\@addtoreset{equation}{section}
\makeatother

\makeatletter
\newcommand*\snakerightarrow[2][]{%
\mathrel{%
\settowidth\@tempdima{\footnotesize#1}%
\settowidth\@tempdimb{\footnotesize#2}%
\ifdim\@tempdimb>\@tempdima\@tempdima=\@tempdimb\fi
\advance\@tempdima15pt
\tikz[font=\footnotesize,baseline=-2pt]
\draw[decorate,decoration={snake, pre length=3pt, post length=3pt, segment length=6pt, amplitude=2pt},->]
(0,0) -- node[above] {#2} node[below] {#1} (\@tempdima,0);%
}%
}
\makeatother

\begin{document}

\title{
Spence--Kummer's trilogarithm functional equation\\
and its underlying geometry:\\{\footnotesize a $\pi_1$-theoretic refinement and its $\ell$-adic Galois analogue}}

\author[D.~Shiraishi]{Densuke Shiraishi}
\address{Department of General Education,
National Institute of Technology, Kagawa College,
355 Chokushi-cho, Takamatsu, Kagawa 761-8058 Japan
}
\email{densuke.shiraishi@gmail.com,
shiraishi-d@t.kagawa-nct.ac.jp}
\date{}

\maketitle


\begin{abstract}
In this paper, we investigate the underlying geometry of the Spence--Kummer functional equation for the trilogarithm.
Our geometry naturally determines a certain path system on the projective line minus three points, connecting the standard tangential base point to the nine variables of the $Li_{3}$ terms in the equation,
reflecting the symmetry of the non-Fano arrangement.
Consequently,
we derive a precise form of the Spence--Kummer equation together with its $\ell$-adic Galois analogue by using algebraic relations between polylogarithm generating series arising from the path system.
We apply the tensor
and homotopy 
criteria for functional equations of complex and $\ell$-adic iterated integrals developed by Zagier and Nakamura--Wojtkowiak.
To compute the lower-degree terms of the functional equation in both the complex and the $\ell$-adic Galois cases,
we also focus on 
a diagram of three geometric objects:
the moduli space $M_{0,5}$,
the complement to the Coxeter arrangement of type ${\rm B_3}$, and the complement to the non-Fano arrangement.
\end{abstract}

\section*{Introduction and main results}

In the first half of the 19th century,
Spence \cite{Sp1809} and Kummer \cite{K1840} independently derived the famous $9$-term functional equation of the complex trilogarithm $Li_3$:
\begin{description}
   \item[(d-$\bC$)~Spence--Kummer equation]\label{d-C}
\begin{align}
&Li_3\left( \dfrac{x(1-y)^2}{y(1-x)^2};\gamma_1\right)+Li_3\left( xy ;\gamma_2\right)+Li_3\left( \dfrac{x}{y};\gamma_3\right)-2Li_3\left( \dfrac{x(1-y)}{y(1-x)};\gamma_4\right) \notag \\
&-2Li_3\left( \dfrac{x(1-y)}{x-1};\gamma_5\right)-2Li_3\left( \dfrac{1-y}{1-x};\gamma_6\right)-2Li_3\left( \dfrac{1-y}{y(x-1)};\gamma_7\right)-2Li_3\left(x;\gamma_8\right) \notag  \\
&-2Li_3\left(y;\gamma_9\right)+2\zeta(3)=\log^2(y;\gamma_9)\log\left(\dfrac{1-y}{1-x};\gamma_6\right)-\dfrac{\pi^2}{3}\log(y;\gamma_9)-\dfrac{1}{3}\log^3(y;\gamma_9), \notag
\end{align}
\end{description}
which embodies a beautiful harmony among $Li_3$ terms.
This functional equation relies on a subtle geometric balance of
the path system 
$\{\gamma_i\}_{i=1,\ldots,9}$ on $\mathbb{P}^1(\bC) \backslash \{0,1,\infty\}$
from the standard tangential base point $\overrightarrow{01}$ to the nine variables of $Li_3$ on the left-hand side.
However,
neither Spence nor Kummer specified the proper domain of $(x,y)$,
nor did they provide a precise definition of the path system.
The purpose of this paper is to 
address these points
by investigating the underlying geometry of the functional equation,
which enables us 
to derive not only the complex case but also 
its $\ell$-adic Galois analogue.
Our derivation is based on a delicate $\pi_1$-theoretic analysis,
building upon the framework established in Nakamura--Wojtkowiak \cite{NW12}.

In the present paper, 
we regard
the complex polylogarithm $Li_k(z)$
as a map
\begin{align}
Li_k(z): \pi_1^{\rm top}\left(\mathbb{P}^1(\mathbb{C}) \backslash \{0,1,\infty\}; \overrightarrow{01}, z\right) \to \bC,
\quad \gamma \mapsto Li_k(z;\gamma),
\end{align}
which represents 
a certain iterated integral along the topological paths from $\overrightarrow{01}$ to $z$,
where $Li_k(z;\gamma)$ denotes the image of $\gamma$ under this map.
See \S \ref{comp-poly} for precise definitions of $Li_k(z;\gamma)$ and $\log(z;\gamma)$.

Let $K$ be a subfield of $\bC$ with algebraic closure $\overline{K}$.
The geometric object for describing $\{\gamma_i\}_{i=1,\ldots,9}$ is the complement to the non-Fano arrangement over $\overline{K}$
\begin{align}\label{V-fano}
V_{\rm non\text{-}Fano} := \Spec\left({\overline{K}}\left[s_1,s_2,\frac{1}{s_1 s_2(1-s_1)(1-s_2)(s_1-s_2)(1-s_1 s_2)}\right]\right)
\end{align}
together with the family of nine morphisms
\vspace{-0.3cm}
\begin{figure}[h]
\begin{tikzpicture}
\node at (-0.27,0.1) {$\{f_i\}_{i=1,\ldots,9}:$};
\node at (1.7,0.1) {$V_{\rm non\text{-}Fano}$};
\node at (5.5,0.1) {${\bP}^{1}_{\overline{K}} \backslash \{0,1,\infty\}$};
\draw[->] (2.5,0.45) to [out=50,in=130] (4.2,0.45);
\draw[->] (2.5,0.25) to [out=30,in=150] (4.2,0.25);
\draw[->] (2.5,-0.25) to [out=-50,in=-130] (4.2,-0.25);
\node at (3.35,0.09) {$\vdots$};
\node at (3.35,1) {\tiny $f_1$};
\node at (3.35,0.65) {\tiny $f_2$};
\node at (3.35,-0.8) {\tiny $f_9$};
\end{tikzpicture}
\end{figure}
\vspace{-0.3cm}

\noindent
defined as follows:
\begin{align}\label{fi}
f_1(s_1,s_2)&:=\dfrac{s_1(1-s_2)^2}{s_2(1-s_1)^2},&
f_2(s_1,s_2)&:=s_1 s_2,&
f_3(s_1,s_2)&:=\dfrac{s_1}{s_2},\notag \\
f_4(s_1,s_2)&:=\dfrac{s_1(1-s_2)}{s_2(1-s_1)},& f_5(s_1,s_2)&:=\dfrac{s_1(1-s_2)}{s_1-1},&
f_6(s_1,s_2)&:=\dfrac{1-s_2}{1-s_1},\\
f_7(s_1,s_2)&:=\dfrac{1-s_2}{s_2(s_1-1)},&
f_8(s_1,s_2)&:=s_1,&
f_9(s_1,s_2)&:={s_2}. \notag
\end{align}
It should be noted that the exclusion of the divisors $s_1=s_2$ and $s_1s_2=1$ in the definition of $V_{\rm non\text{-}Fano}$ ensures that 
all nine morphisms $f_1,\ldots,f_9$ are well-defined.
Thus,
$V_{\rm non\text{-}Fano}$ is the most natural affine algebraic variety
equipped with a family of nine well-defined morphisms to  
${\bP}^{1}_{\overline{K}} \backslash \{0,1,\infty\}$ described above.

The arrangement
$s_1 s_2(1-s_1)(1-s_2)(s_1-s_2)(1-s_1 s_2)$ is shown in FIGURE \ref{fig:non-Fano}.
Hereafter, we refer to this arrangement as the non-Fano arrangement, named after the well-known notion of the non-Fano matroid (see Remark \ref{rem-non-Fano} below).
The affine variety $V_{\rm non\text{-}Fano}$ is isomorphic to Goncharov's moduli space of certain $7$-point configurations on $\mathbb{P}^2_{{\overline{K}}}$ (see Remark \ref{gon} for details).

We consider
the $K$-rational tangential base point
\begin{align}
\vec{v}:\Spec\Big({K}((t))\Big) \to V_{\rm non\text{-}Fano}
\end{align}
over the $K(t)$-rational point $(t^2,t)$.

\vspace{-0.1cm}
\begin{figure}[hbtp]
\centering
\caption{The non-Fano arrangement}
\label{fig:non-Fano}
\begin{tikzpicture}[scale=1]
\draw[-] (-5,1) to (-9,1);
\draw[-] (-5,0) to (-9,0);
\draw[-] (-5,-1) to (-9,-1);
\draw[-] (-8,-2) to (-8,2);
\draw[-] (-7,-2) to (-7,2);
\draw[-] (-6,-2) to (-6,2);
\draw[-] (-5,2) to (-9,-2);
\draw[-] (-5,-2) to (-9,2);
\node at (-7,2.3) {~};
\node (a0) at (-8.2, -2.2) {\tiny $s_1=0$};
\node (a1) at (-7, -2.2) {\tiny $s_1=1$};
\node (ainf) at (-5.8, -2.2) {\tiny $s_1=\infty$};
\node (aa0) at (-9.6, -1) {\tiny $s_2=0$};
\node (aa1) at (-9.6, 0) {\tiny $s_2=1$};
\node (aainf) at (-9.6, 1) {\tiny $s_2=\infty$};
\node (cL1) at (-9.55,-2) {\tiny $s_1=s_2$};
\node (cL1) at (-4.4,-2) {\tiny $1=s_1 s_2$};
\end{tikzpicture}
\end{figure}

Our main result for the complex case is stated as follows.
The following theorem refines the Spence--Kummer functional equation (d-$\bC$), by precisely specifying the proper domain of $(x,y)$ and the family $\{\gamma_i\}_{i=1,\ldots,9}$ defining the nine principal polylogarithm terms.
Let $V_{\rm non\text{-}Fano}^\an$ be the complex analytic space associated to $V_{\rm non\text{-}Fano}$ via the inclusion $\overline{K} \hookrightarrow \mathbb{C}$.
We denote by $f^{\mathrm{an}}_i: V_{\rm non\text{-}Fano}^{\mathrm{an}} \to {\bP}^{1}(\bC) \backslash \{0,1,\infty\}~(i=1,...,9)$ the morphisms between complex analytic spaces associated to $f_i
 ~(i=1,...,9)$ in (\ref{fi}).

\begin{thm}[The $9$-term functional equation for the complex trilogarithm]\label{main1}
Given a point
$
(x,y) \in V_{\rm non\text{-}Fano}^\an
$
and a path $\gamma_0 \in \pi_1^{\rm top}\left(V_{\rm non\text{-}Fano}^\an;{\vec{v}},(x,y)\right)$,
define the path system $\{\gamma_i\}_{i=1,\ldots,9}$ associated with $\gamma_0$ by
\begin{align}\label{gammai}
\gamma_i:= \delta_i \cdot f_i^{\an}\left(\gamma_0\right) \in \pi_1^{\rm top}\left(\mathbb{P}^1(\mathbb{C}) \backslash \{0,1,\infty\}; \overrightarrow{01}, f^{\an}_{i}(x,y)\right),
\end{align}
where $\delta_i \in \pi_1^{\rm top}\left({\bP}^{1}(\bC) \backslash \{0,1,\infty\};\overrightarrow{01},f_i^{\an}\left(\vec{v}\right)\right)$ is as shown in TABLE \ref{table_delta} and the paths are composed from left to right.
Then the Spence--Kummer functional equation (d-$\bC$) above holds.
\end{thm}

\renewcommand{\arraystretch}{1.6}
\tabcolsep = 0.3cm
\begin{table}[hbtp]
  \caption{$\delta_1, \ldots ,\delta_9$}
  \label{table_delta}
  \centering
  \begin{tabular}{|c|c|c|c|c|}
    \hline
    $i$ & $f^{}_{i}(x,y)$ & $f_i(t^2,t)$ & $f_i\left(\vec{v}\right)$ & $\delta_i \in \pi_1^{\rm top}\left({\bP}^{1}(\bC) \backslash \{0,1,\infty\};\overrightarrow{01},f_i^{\an}\left(\vec{v}\right)\right)$  \\
    \hline \hline

    $1$  & $\frac{x(1-y)^2}{y(1-x)^2}$ & $\frac{t}{(1+t)^2}$  & $\overrightarrow{01} \approx f_1\left(\vec{v}\right)$ & $\delta_1:=1
~(=\rm{trivial~path})$  \\ \hline

    $2$  & $xy$ & $t^3$  & $\overrightarrow{01} \approx f_2\left(\vec{v}\right)$ & $\delta_2:=1$  \\ \hline

    $3$  & $\frac{x}{y}$ & $t$  & $\overrightarrow{01} = f_3\left(\vec{v}\right)$ & $\delta_3:=1$  \\ \hline

    $4$  & $\frac{x(1-y)}{y(1-x)}$ & $\frac{t}{1+t}$  &  $\overrightarrow{01} \approx f_4\left(\vec{v}\right)$ & $\delta_4:=1$  \\
    \hline

    $5$  & $\frac{x(1-y)}{x-1}$ & $\frac{-t^2}{1+t}$  &  $\overrightarrow{0\infty} \approx f_5\left(\vec{v}\right)$ & $\delta_5:=\delta_{\overrightarrow{0\infty}}\left(=\text{as in FIGURE~\ref{fig:path}}\right)$  \\
    \hline

    $6$  & $\frac{1-y}{1-x}$ & $\frac{1}{1+t}$  &  $\overrightarrow{10} \approx f_6\left(\vec{v}\right)$ & $\delta_6:=\delta_{\overrightarrow{10}}\left(=\text{as in FIGURE~\ref{fig:path}}\right)$  \\
    \hline

    $7$  & $\frac{1-y}{y(x-1)}$ & $\frac{-1}{t(1+t)}$  &  $\overrightarrow{\infty0} \approx f_7\left(\vec{v}\right)$ & $\delta_7:=\delta_{\overrightarrow{\infty0}}\left(=\text{as in FIGURE~\ref{fig:path}}\right)$  \\
    \hline

    $8$  & $x$ &  $t^2$  &  $\overrightarrow{01} \approx f_8\left(\vec{v}\right)$ & $\delta_8:=1$  \\
    \hline

    $9$  & $y$ &  $t$ &  $\overrightarrow{01} = f_9\left(\vec{v}\right)$ & $\delta_9:=1$  \\
    \hline
  \end{tabular}
\end{table}

In TABLE \ref{table_delta},
we identify the images $f_i\left(\vec{v}\right)$ with standard $K$-rational tangential base points of ${\bP}^{1}_{\overline{K}} \backslash \{0,1,\infty\}$ under Galois equivalence (denoted $\approx$ in the sense of \cite[\S5.9]{N02}).
In FIGURE \ref{fig:path},
the dashed line represents ${\mathbb P}^1({\mathbb R})\backslash 
\{0,1,\infty\}$, and
the upper half-plane is located above this line.

\begin{figure}[hbtp]
\centering
\caption{Topological paths on ${\bP}^{1}(\bC) \backslash \{0,1,\infty\}$}
\label{fig:path}
\begin{tikzpicture}
\draw (2,0) -- (6,0);
\draw (4,0) -- (3.9,0.1);
\draw (4,0) -- (3.9,-0.1);



\draw (2,1) -- (2.1,1.1);
\draw (2,1) -- (2.1,0.9);

\draw (2,3) -- (2.1,3.1);
\draw (2,3) -- (2.1,2.9);



\draw (2,0) to [out=0,in=270] (2.5,0.5);
\draw (2.5,0.5) to [out=90,in=0] (2,1);
\draw (2,1) to [out=180,in=90] (1.5,0.5);
\draw (1.5,0.5) to [out=270,in=180] (2,0);


\draw (2,0) to [out=0,in=270] (3.5,1.5);
\draw (3.5,1.5) to [out=90,in=0] (2,3);
\draw (2,3) to [out=180,in=0] (-2,0);
\draw[dotted](-3.5,0)--(7.5,0);

\fill[white](2,0)circle(0.07);
\fill[white](6,0)circle(0.07);
\fill[white](-2,0)circle(0.07);

\node at (0,3.3) {~};
\node at (2,-0.4) {$0$};
\node at (4,-0.4) {$\delta_{\overrightarrow{10}}$};
\node at (1.1,0.5) {$\delta_{\overrightarrow{0\infty}}$};
\node at (3.9,1.7) {$\delta_{\overrightarrow{\infty0}}$};
\node at (6,-0.4) {$1$};
\node at (2,0) {${\circ}$};
\node at (6,0) {${\circ}$};
\node at (-2,0) {${\circ}$};
\node at (-2,-0.4) {$\infty$};
\end{tikzpicture}
\end{figure}

We next discuss the $\ell$-adic Galois case for any fixed prime number $\ell$.
Let
\begin{align}
G_K=\Gal(\overline{K}/K)
\end{align}
be the absolute Galois group of a subfield $K \subset \mathbb{C}$.
Suppose that $z$ is a $K$-rational point of $\mathbb{P}^1 \backslash \{0,1,\infty\}$, and
consider each topological path $\gamma \in \pi_1^{\rm top}\left(\mathbb{P}^1(\mathbb{C}) \backslash \{0,1,\infty\}; \overrightarrow{01}, z\right)$ to be a pro-$\ell$ \'etale path $\gamma \in \pi_1^{\ell \text{-\'et}}\left(\mathbb{P}^1_{\overline{K}} \backslash \{0,1,\infty\}; \overrightarrow{01}, z\right)$ through the comparison map induced by the inclusion $\overline{K} \hookrightarrow \mathbb{C}$.
For $\sigma \in G_K$,
the $\ell$-adic Galois polylogarithm $Li_k^\ell(z;\gamma,\sigma)$ was introduced by Wojtkowiak as the $\ell$-adic iterated integral along $\gamma$.
See \S \ref{l-poly} for a precise definition of $Li_k^\ell(z;\gamma,\sigma)$.
We understand
$Li_k^\ell(z)$ to be a map
\begin{align}
Li_k^\ell(z): \pi_1^{\rm top}\left(\mathbb{P}^1(\mathbb{C}) \backslash \{0,1,\infty\}; \overrightarrow{01}, z\right) \times G_K \to \bQ_\ell, \quad
(\gamma,\sigma) \mapsto Li_k^\ell(z;\gamma,\sigma),
\end{align}
with values in the $\ell$-adic number field.
This object
is closely related to the generalized $\ell$-adic Soul\'e character
\[
\tilde{\chi}_{k}
^{z,\gamma}: G_K \to {\mathbb Z}_{\ell}
\]
formulated by Nakamura-Wojtkowiak \cite[DEFINITION~3]{NW99}
via the explicit formula \cite[COROLLARY]{NW99},~\cite[Proposition 4.2]{NS25}.
In particular,
the $\ell$-adic Galois zeta value (i.e., $\ell$-adic Soul\'e element) 
${\boldsymbol \zeta}_{{k}}^\ell(\sigma)$ is defined by the special value ${Li}^\ell_{{k}}\left(\overrightarrow{10};\delta_{\overrightarrow{10}},\sigma\right)$.

Our main result for the  $\ell$-adic Galois case is as follows.
Notably, the following functional equation involves nontrivial lower-weight terms including two $Li_2^\ell$ terms in contrast to the complex Spence-Kummer equation (d-$\mathbb{C}$).
These lower-weight terms, 
which arise specifically in the $\ell$-adic Galois setting, 
are referred to in \cite[Subsection 4.3]{NW12} as ``$\ell$-adic error terms'';
see Remark \ref{Zl-integral} and Remark \ref{arithm} for their geometric and arithmetic meanings.

\begin{thm}[The $9$-term functional equation for the $\ell$-adic Galois trilogarithm]\label{main2}
Given a $K$-rational point
$(x,y) \in V_{\rm non\text{-}Fano}(K)$ and a path $\gamma_0 \in \pi_1^{\rm top}\left(V_{\rm non\text{-}Fano}^\an;{\vec{v}},(x,y)\right)$,
define the path system $\{\gamma_i\}_{i=1,\ldots,9}$ associated with $\gamma_0$ as in (\ref{gammai}).
For any $\sigma \in G_K$,
the following holds:
\begin{description}
   \item[(d-$\ell$)~$\ell$-adic Spence--Kummer equation]
\begin{align}
&Li_3^{\ell}\left( \dfrac{x(1-y)^2}{y(1-x)^2};\gamma_1,\sigma\right)+Li_3^{\ell}\left( xy ;\gamma_2,\sigma\right)+Li_3^{\ell}\left( \dfrac{x}{y};\gamma_3,\sigma\right)-2Li_3^{\ell}\left( \dfrac{x(1-y)}{y(1-x)};\gamma_4,\sigma\right) \notag \\
&-2Li_3^{\ell}\left( \dfrac{x(1-y)}{x-1};\gamma_5,\sigma\right)-2Li_3^{\ell}\left( \dfrac{1-y}{1-x};\gamma_6,\sigma\right)-2Li_3^{\ell}\left( \dfrac{1-y}{y(x-1)};\gamma_7,\sigma\right)-2Li_3^{\ell}\left(x;\gamma_8,\sigma\right) \notag \\
&-2Li_3^{\ell}\left(y;\gamma_9,\sigma\right)+2{\boldsymbol \zeta}^{\ell}_3(\sigma)=-\rho_{y,\gamma_9}(\sigma)^2\rho_{\frac{1-y}{1-x},\gamma_6}(\sigma)+2{\boldsymbol \zeta}^{\ell}_2(\sigma)\rho_{y,\gamma_9}(\sigma)+\dfrac{1}{3}\rho_{y,\gamma_9}(\sigma)^3 \notag \\
&
-Li_2^\ell\left(\dfrac{x(1-y)}{x-1};\gamma_5,\sigma\right)-Li_2^\ell\left(\dfrac{1-y}{y(x-1)};\gamma_7,\sigma\right)+\dfrac{1}{2}\rho_{\frac{1-xy}{1-x},\gamma'_5}(\sigma)
-\dfrac{1}{3}\rho_{y,\gamma_9}(\sigma), \notag
\end{align}
\end{description}
together with the following equation for the generalized $\ell$-adic Soul\'e character $\tilde{\chi}_{k}
^{z,\gamma}: G_K \to {\mathbb Z}_{\ell}$:
\begin{description}
\item[(d'-$\ell$)~Integral $\ell$-adic Spence-Kummer equation]
\begin{align} \notag
&\tilde{\chi}_{3}^{{\frac{x(1-y)^2}{y(1-x)^2}},\gamma_1}(\sigma)
+\tilde{\chi}_{3}^{{xy},\gamma_2}(\sigma)
+\tilde{\chi}_{3}^{{\frac{x}{y}},\gamma_3}(\sigma)
-2\tilde{\chi}_{3}^{\frac{x(1-y)}{y(1-x)},\gamma_4}(\sigma) \notag \\
&-2\tilde{\chi}_{3}^{\frac{x(1-y)}{x-1},\gamma_5}(\sigma)
-2\tilde{\chi}_{3}^{\frac{1-y}{1-x},\gamma_6}(\sigma)-2\tilde{\chi}_{3}^{\frac{1-y}{y(x-1)},\gamma_7}(\sigma)
-2\tilde{\chi}_{3}^{{x},\gamma_8}(\sigma) \notag \\
&-2\tilde{\chi}_{3}^{{y},\gamma_9}(\sigma)+2\tilde{\chi}_{3}^{{\overrightarrow{10}},\delta_{\overrightarrow{10}}}(\sigma)
=-2\rho_{y,\gamma_9}(\sigma)^2 \rho_{\frac{1-y}{1-x},\gamma_6}(\sigma)-12\tilde{\chi}_{2}^{{\overrightarrow{10}},\delta_{\overrightarrow{10}}}(\sigma)\rho_{y,\gamma_9}(\sigma) \notag \\
&+2\chi(\sigma)\left(\tilde{\chi}_{2}^{\frac{x(1-y)}{x-1},\gamma_5}(\sigma)+\tilde{\chi}_{2}^{\frac{1-y}{y(x-1)},\gamma_7}(\sigma)\right)
+
\chi(\sigma)^2\rho_{\frac{1-xy}{1-x},\gamma'_5}(\sigma)
\notag \\
&
-\frac{2}{3}\rho_{y,\gamma_9}(\sigma)\left(1-\rho_{y,\gamma_9}(\sigma)\right)\left(1+\rho_{y,\gamma_9}(\sigma)\right).\notag
\end{align}
\end{description}
Here,
$\chi: G_K \to {\mathbb Z}_{\ell}^\times$ is the $\ell$-adic cyclotomic character,
$
\rho_{z,\gamma}: G_K \to \mathbb{Z}_{\ell}
$
is the Kummer $1$-cocycle
defined by $\sigma(z^{1/{\ell^k}})=\zeta_{\ell^k}^{\rho_{z,\gamma}(\sigma)}z^{1/{\ell^k}}$
with respect to $\{z^{1/{\ell^k}}\}_{k \in \mathbb{N}}$ along $\gamma \in \pi_1^{\rm top}\left(\mathbb{P}^1(\mathbb{C}) \backslash \{0,1,\infty\}; \overrightarrow{01}, z\right)$,
and $\gamma'_5$ is a certain path associated
with $\gamma_5$ (see (\ref{gamma'}) for the precise definition).
\end{thm}

A primary step of the present work is to capture
$\pi_1^{\rm top}\left(V_{\rm non\text{-}Fano}^{\rm an}\right)$ as a subquotient of the well-known fundamental group $\pi_1^{\rm top}\left(M_{0,5}^{\rm an}\right)$ of the moduli space of the projective line with five ordered points
through an intermediate geometric object ${V}_{\rm B_3}$, that is,
the complement to the Coxeter arrangement of type ${\rm B_3}$.
Using the diagram FIGURE \ref{kdiag} and the Galois theory for $M_{0,5}^{\rm an}$,
we explicitly construct topological loops on $V_{\rm non\text{-}Fano}^{\rm an}$ that serve as generators of $\pi_1^{\rm top}\left(V_{\rm non\text{-}Fano}^{\rm an}\right)$.

\begin{figure}[h]
\centering
\caption{Key diagram to capture $\pi_1^{\rm top}\left(V_{\rm non\text{-}Fano}^{\rm an}\right)$}
\label{kdiag}
\resizebox{\textwidth}{!}{%
\begin{tikzpicture}[scale=1,
  line width=0.75pt,
  sp/.style={
    draw=blue!55!black, fill=blue!7, rounded corners=5pt,
    inner xsep=9pt, inner ysep=5pt, font=\small
  },
  gp/.style={
    draw=orange!75!black, fill=orange!7, rounded corners=5pt,
    inner xsep=9pt, inner ysep=5pt, font=\small
  },
  tp/.style={
    draw=teal!65!black, fill=teal!7, rounded corners=5pt,
    inner xsep=9pt, inner ysep=5pt, font=\small
  },
  lbl/.style={font=\scriptsize},
  tlbl/.style={font=\tiny},
]


\node[sp] (VB3)  at (0,    0)    {$V_{{\rm B}_3}^{\rm an}$};
\node[sp] (VnF)  at (3.8,  0)    {$V_{\rm non\text{-}Fano}^{\rm an}$};
\node[sp] (M05)  at (0,   -2.6)  {$M_{0,5}^{\rm an}$};

\draw[injarrow, blue!55!black] (VB3) -- (VnF)
    node[midway, above, tlbl] {open immersion};

\draw[surjarrow, blue!55!black] (VB3) -- (M05);
\node[lbl, left=2pt]  at (-0.1, -1.3) {fin.\ Gal.};
\node[lbl, right=2pt] at ( 0.1, -1.3) {$f_{\rm cov}^{\rm an}$};

\node[tlbl, blue!45!black] at (2.5, -2.8)
  {\itshape geometric objects};


\draw[dotted, gray!45, line width=0.5pt] (5.8, 0.9) -- (5.8, -3.1);
\node[rotate=360, gray!55!black, tlbl] at (5.8, 0.5)
  {$\pi_1^{\rm top}(\,\cdot\,)$};

\draw[-{Stealth[length=4pt]}, gray!45, line width=0.5pt,
      shorten >=3pt, shorten <=2pt]
  (5.2,0) to (6.4,0);


\node[gp] (pVB3) at (7.7,   0)    {$\pi_1^{\rm top}(V_{{\rm B}_3}^{\rm an})$};
\node[gp] (pVnF) at (12.1,  0)    {$\pi_1^{\rm top}(V_{\rm non\text{-}Fano}^{\rm an})$};
\node[gp] (pM05) at (7.7,  -2.6)  {$\pi_1^{\rm top}(M_{0,5}^{\rm an})$};

\draw[surjarrow, orange!65!black] (pVB3) -- (pVnF);

\draw[injarrow, orange!65!black] (pVB3) -- (pM05)
    node[midway, right=1pt, lbl] {$f_{{\rm cov}*}^{\rm an}$};


\node[tlbl, gray!65!black, draw=gray!30, fill=gray!5,
      rounded corners=4pt, inner sep=5pt, align=center]
  (galnote) at (11.9, -1.55)
  {generators of $\pi_1^{\rm top}(M_{0,5}^{\rm an})$\\
   $\leadsto$ generators of $\pi_1^{\rm top}(V_{\rm non\text{-}Fano}^{\rm an})$\\[1ex] (\text{via Galois theory of} $M_{0,5}^{\rm an}$)};

\draw[-{Stealth[length=4pt]}, gray!45, line width=0.5pt,
      shorten >=3pt, shorten <=2pt]
  (galnote.north) to[out=90, in=270] (11.9, -0.38);

\draw[-{Stealth[length=4pt]}, gray!45, line width=0.5pt,
      shorten >=3pt, shorten <=2pt]
  (pM05) to (galnote);

\node[tlbl, gray!65!black, draw=gray!30, fill=gray!5,
      rounded corners=4pt, inner sep=5pt, align=center]
  (galnote2) at (3.8, -1.55)
  { loops on $M_{0,5}^{\rm an}$\\
   $\leadsto$ loops on $V_{\rm non\text{-}Fano}^{\rm an}$\\[1ex] (via taking $\left(f_{{\rm cov}}^{\rm an}\right)^{-1}$)};

\draw[-{Stealth[length=4pt]}, gray!45, line width=0.5pt,
      shorten >=3pt, shorten <=2pt]
  (galnote2.north) to[out=90, in=270] (3.8, -0.38);

\draw[-{Stealth[length=4pt]}, gray!45, line width=0.5pt,
      shorten >=3pt, shorten <=2pt]
  (M05) to (galnote2);

\node[tlbl, orange!55!black] at (10.9, -2.8)
  {\itshape fundamental groups};

\end{tikzpicture}
}
\end{figure}


This enables us to precisely compute the lower-degree terms of the Spence--Kummer equation in both the complex case and the $\ell$-adic Galois case,
where we follow the computational method devised by Nakamura and Wojtkowiak \cite[Proposition 5.11]{NW12}, incorporating Zagier's tensor criterion into the language of fundamental groups \cite[Theorem~5.7]{NW12}.
During our procedure, the dilogarithm functional equations of Schaeffer, Kummer, and Hill types, denoted by (a,b,c-$\bC$) and (a,b,c-$\ell$) in TABLE \ref{funcs}, are also obtained as byproducts.
These equations are used to reduce the ${Li}_2$ terms to derive Spence--Kummer equations (d-$\bC$) and (d-$\ell$).

{\renewcommand{\arraystretch}{1.8}
\tabcolsep = 0.5cm
\small
\begin{table}[hbtp]
\centering
\caption{Functional equations to be proved}
\label{funcs}
  \begin{tabular}{|c||c|}  \hline
    $\ell$-adic Galois side & complex side  \\ \hline \hline

   \tiny$Li_k^\ell(z;\gamma,\sigma) \in \bQ_\ell$ & \tiny $Li_k(z;\gamma) \in \bC$ \\

   \tiny where $z$ is a $K$-rational base point of $\mathbb{P}^1 \backslash \{0,1,\infty\}$ & \tiny where $z$ is a $\bC$-rational base point of $\mathbb{P}^1 \backslash \{0,1,\infty\}$ \\

   \tiny and $(\gamma,\sigma) \in \pi_1^{\rm top}\left(\mathbb{P}^1(\mathbb{C}) \backslash \{0,1,\infty\}; \overrightarrow{01}, z\right) \times G_K$. & \tiny and $\gamma \in \pi_1^{\rm top}\left(\mathbb{P}^1(\mathbb{C}) \backslash \{0,1,\infty\}; \overrightarrow{01}, z\right)$. \\ \hline
    
    \tiny $\rho_{z,\gamma}(\sigma)$,\quad$\rho_{1-z,\gamma'}(\sigma).$ & \tiny $-\log(z;\gamma)$,\quad$-\log(1-z;\gamma').$ \\ \hline \hline

    \tiny$\displaystyle Li_3^{\ell}\left( \dfrac{x(1-y)^2}{y(1-x)^2};\gamma_1,\sigma\right)+Li_3^{\ell}\left( xy ;\gamma_2,\sigma\right)+Li_3^{\ell}\left( \dfrac{x}{y};\gamma_3,\sigma\right)
		$ & \tiny$\displaystyle Li_3\left( \dfrac{x(1-y)^2}{y(1-x)^2};\gamma_1\right)+Li_3\left( xy ;\gamma_2\right)+Li_3\left( \dfrac{x}{y};\gamma_3\right)$ \\ 

     \tiny $-2Li_3^{\ell}\left( \dfrac{x(1-y)}{y(1-x)};\gamma_4,\sigma\right)-2Li_3^{\ell}\left( \dfrac{x(1-y)}{x-1};\gamma_5,\sigma\right)$ & \tiny $-2Li_3\left( \dfrac{x(1-y)}{y(1-x)};\gamma_4\right)-2Li_3\left( \dfrac{x(1-y)}{x-1};\gamma_5\right)$  \\

\tiny $-2Li_3^{\ell}\left( \dfrac{1-y}{1-x};\gamma_6,\sigma\right)-2Li_3^{\ell}\left( \dfrac{1-y}{y(x-1)};\gamma_7,\sigma\right)$ & \tiny $-2Li_3\left( \dfrac{1-y}{1-x};\gamma_6\right)-2Li_3\left( \dfrac{1-y}{y(x-1)};\gamma_7\right)$  \\

\tiny $-2Li_3^{\ell}\left(x;\gamma_8,\sigma\right)-2Li_3^{\ell}\left(y;\gamma_9,\sigma\right)+2{\boldsymbol \zeta}^{\ell}_3(\sigma)$ & \tiny $-2Li_3\left(x;\gamma_8\right)-2Li_3\left(y;\gamma_9\right)+2\zeta(3)$ \\

\tiny $=-\left(\rho_{y,\gamma_9}(\sigma)\right)^2\rho_{\frac{1-y}{1-x},\gamma_6}(\sigma)$ & \tiny $=\log^2(y;\gamma_9)\log\left(\dfrac{1-y}{1-x};\gamma_6\right)$  \\ 

\tiny $+2{\boldsymbol \zeta}^{\ell}_2(\sigma)\rho_{y,\gamma_9}(\sigma)+\dfrac{1}{3}\left(\rho_{y,\gamma_9}(\sigma)\right)^3$ & \tiny $-\dfrac{\pi^2}{3}\log(y;\gamma_9)-\dfrac{1}{3}\log^3(y;\gamma_9).$  \\ 

\tiny $-Li_2^\ell\left(\dfrac{x(1-y)}{x-1};\gamma_5,\sigma\right)-Li_2^\ell\left(\dfrac{1-y}{y(x-1)};\gamma_7,\sigma\right)$ & Theorem \ref{main1}~~(d-$\bC$),  \\

\tiny $+\dfrac{1}{2}\rho_{\frac{1-xy}{1-x},\gamma'_5}(\sigma)
-\dfrac{1}{3}\rho_{y,\gamma_9}(\sigma).$ & Spence~\cite{Sp1809},  \\ 

Theorem \ref{main2}~~(d-$\ell$) & Kummer~\cite{K1840} \\ \hline \hline

   \tiny $Li_{2}^{\ell}\left({\frac{x(1-y)}{y(1-x)};\gamma_4,\sigma}\right)
-Li_{2}^{\ell}\left(y;\gamma_9,\sigma\right)
+Li_{2}^{\ell}\left(x;\gamma_8,\sigma\right)$  & \tiny $Li_{2}\left({\frac{x(1-y)}{y(1-x)};\gamma_4}\right)
-Li_{2}\left(y;\gamma_9\right)
+Li_{2}\left(x;\gamma_8\right)$ \\ 

   \tiny $-Li_{2}^{\ell}\left(\frac{x}{y};\gamma_3,\sigma\right)-Li_{2}^{\ell}\left({\frac{1-y}{1-x};\gamma_6,\sigma}\right)$  & \tiny $-Li_{2}\left(\frac{x}{y};\gamma_3\right)-Li_{2}\left({\frac{1-y}{1-x};\gamma_6}\right).$ \\
   \tiny $=\rho_{y,\gamma_9}(\sigma)\rho_{\frac{1-y}{1-x},\gamma_6}(\sigma)-{\boldsymbol \zeta}_2^{\ell}(\sigma).$  & \tiny $=\log(y;\gamma_9)\log\left(\dfrac{1-y}{1-x};\gamma_6\right)-\dfrac{\pi^2}{6}.$ \\

Theorem \ref{l-dilog}~(a-$\ell$) & Theorem \ref{c-dilog}~(a-$\bC$),~Schaeffer~\cite{Sc1846} \\ \hline

   \tiny $Li_{2}^{\ell}\left({\frac{x(1-y)^2}{y(1-x)^2};\gamma_1,\sigma}\right)
-Li_{2}^{\ell}\left({\frac{x(1-y)}{x-1};\gamma_5,\sigma}\right)$  & \tiny $Li_{2}\left({\frac{x(1-y)^2}{y(1-x)^2};\gamma_1}\right)
-Li_{2}\left({\frac{x(1-y)}{x-1};\gamma_5}\right)$ \\
 
   \tiny $-Li_{2}^{\ell}\left({\frac{1-y}{y(x-1)};\gamma_7,\sigma}\right)-Li_{2}^{\ell}\left({\frac{x(1-y)}{y(1-x)};\gamma_4,\sigma}\right)$  & \tiny $-Li_{2}\left({\frac{1-y}{y(x-1)};\gamma_7}\right)-Li_{2}\left({\frac{x(1-y)}{y(1-x)};\gamma_4}\right)$ \\ 

   \tiny $-Li_{2}^{\ell}\left({\frac{1-y}{1-x};\gamma_6,\sigma}\right)=\frac{1}{2}\left(\rho_{y,\gamma_9}(\sigma)\right)^2$ & \tiny $-Li_{2}\left({\frac{1-y}{1-x};\gamma_6}\right)=\frac{1}{2}\log^2(y;\gamma_9).$ \\ 

   \tiny $+\frac{1}{2}\rho_{y,\gamma_9}(\sigma)+\rho_{1-x,\gamma'_8}(\sigma)-\rho_{1-xy,\gamma'_2}(\sigma).$ & Theorem \ref{c-dilog}~~(b-$\bC$), \\ 

   Theorem \ref{l-dilog}~~(b-$\ell$) & Kummer~\cite{K1840} \\ \hline

   \tiny $Li_{2}^{\ell}\left({\frac{1-y}{y(x-1)};\gamma_7,\sigma}\right)
+Li_{2}^{\ell}\left({xy;\gamma_2,\sigma}\right)
-Li_{2}^{\ell}\left({x;\gamma_8,\sigma}\right)$  & \tiny $Li_{2}\left({\frac{1-y}{y(x-1)};\gamma_7}\right)
+Li_{2}\left({xy;\gamma_2}\right)
-Li_{2}\left({x;\gamma_8}\right)$ \\ 

   \tiny $-Li_{2}^{\ell}\left({y;\gamma_9,\sigma}\right)
-Li_{2}^{\ell}\left({\frac{x(1-y)}{x-1};\gamma_5,\sigma}\right)$  & \tiny $-Li_{2}\left({y;\gamma_9}\right)
-Li_{2}\left({\frac{x(1-y)}{x-1};\gamma_5}\right)$ \\ 

   \tiny $=-{\boldsymbol \zeta}_2^{\ell}(\sigma)
+\rho_{y,\gamma_9}(\sigma)\rho_{\frac{1-y}{1-x},\gamma_6}(\sigma)-\frac{1}{2}\left(\rho_{y,\gamma_9}(\sigma)\right)^2$  & \tiny $=-\dfrac{\pi^2}{6}
+\log(y;\gamma_9)\log\left(\frac{1-y}{1-x};\gamma_6\right)-\frac{1}{2}\log^2(y;\gamma_9).$ \\ 

   \tiny $-\frac{1}{2}\rho_{y,\gamma_9}(\sigma).$ & Theorem \ref{c-dilog}~~(c-$\bC$), \\ 

 Theorem \ref{l-dilog}~~(c-$\ell$) & Hill~\cite{H1830} \\ \hline
  \end{tabular}
\end{table}}

\begin{rem}\label{Zl-integral}
At first glance,
the presence of these lower-weight terms in the $\ell$-adic Spence-Kummer equation (Theorem \ref{main2} (d-$\ell$)) might appear unnatural.
However,
these terms are not artefacts resulting from our proof technique;
rather, they are a natural and geometric consequence arising from the essential difference in the geometric relations of non-commutative variables $X, Y$ and $Z$ between the $\ell$-adic Galois and complex settings:

In the $\ell$-adic Galois setting,
within the generating function of Wojtkowiak's $\ell$-adic iterated integrals,
the non-commutative variables $X := {\bf log}(l_0), Y := {\bf log}(l_1),$ and $Z := {\bf log}(l_\infty)$ (cf. FIGURE \ref{fig:path2}, (\ref{lG-variable}), (\ref{ladicmagnus})) are governed by the non-abelian relation $l_0 l_1 l_\infty = 1$ in the topological fundamental group
\[
\pi_1^{\rm top}\left(\mathbb{P}^1(\mathbb{C}) \backslash \{0,1,\infty\}, \overrightarrow{01}\right)=\left<l_0,
l_1,
l_{\infty} \mid l_0 l_1 l_{\infty}=1 \right>.
\]
Therefore,
$Z$ forms the Baker-Campbell-Hausdorff sum
\begin{align*}
Z = {\bf log}\left({\bf exp}\left(-Y\right){\bf exp}\left(-X\right)\right)
=-Y-X+\text{$($higher-order terms$)$},
\end{align*}
and such a non-linear relation inevitably appears.
The lower-weight terms ($\ell$-adic error terms) appearing in the $\ell$-adic Spence-Kummer equation (Theorem \ref{main2} (d-$\ell$)) are precisely the manifestations of the higher-order terms in this BCH sum (see Remark \ref{imprem1},
Remark \ref{imprem2}
and Remark \ref{imprem3}),
reflecting the non-abelian nature of the pro-$\ell$ \'etale fundamental group
\[
\pi_1^\ellet\left(\mathbb{P}^1_{\overline{K}} \backslash \{0,1,\infty\},\overrightarrow{01}\right).
\]

In contrast,
in the complex setting, let 
$
X:=\left(\dfrac{dz}{z}\right)^{\ast},~
Y:=\left(\dfrac{dz}{z-1}\right)^{\ast}~
\in \Omega^1_{\rm log}\left( \mathbb{P}^1(\mathbb{C}) \backslash \{0,1,\infty\} \right)^{\ast}
$
be the duals of the canonical logarithmic differential forms.
By the natural isomorphism 
\[
\Omega^1_{\rm log}\left( \mathbb{P}^1(\mathbb{C}) \backslash \{0,1,\infty\} \right)^{\ast} \simeq \pi_1^{\rm top}\left(\mathbb{P}^1(\mathbb{C}) \backslash \{0,1,\infty\}, \overrightarrow{01}\right)^{\rm ab} \otimes \mathbb{C},
\]
we have the identifications $X=\frac{\bar{l}_0}{2\pi \sqrt{-1}}$, $Y=\frac{\bar{l}_1}{2\pi \sqrt{-1}}$.
Let $Z \in \Omega^1_{\rm log}\left( \mathbb{P}^1(\mathbb{C}) \backslash \{0,1,\infty\} \right)^{\ast}$ be the element corresponding to $\frac{\bar{l}_\infty}{2\pi \sqrt{-1}}$.
Then the linear relation
\[
Z = -Y - X
\]
holds.
Because of this linearity, the weights of the terms are preserved in functional equations of the complex polylogarithm $Li_k(z)$.

Consequently,
due to these fundamental differences,
the complex setting yields pure-weight functional equations,
whereas the $\ell$-adic Galois setting gives rise to mixed-weight functional equations (see FIGURE \ref{fig:remark0.3}).
\end{rem}

\begin{figure}[htbp]
\caption{Fundamental structural differences between the two sides}
\label{fig:remark0.3}
\centering
\vspace{0.2cm}
\begin{tikzpicture}[
    box/.style={draw=black, thick, rounded corners=4pt, text width=6.6cm, minimum height=1.5cm, align=center},
    arrow/.style={-Stealth, thick},
    darrow/.style={<->, dashed, thick},
    label/.style={font=\footnotesize\itshape, fill=white, inner sep=2pt} 
]

\node[box] (ladic_origin) {Geometric origin:\\[1ex] The action of $G_K$ on\\[1ex] $\pi_1^{\ellet}\left(\mathbb{P}^1_{\overline{K}} \setminus \{0,1,\infty\}; \overrightarrow{01}, z\right)$};

\node[box, below=1.5cm of ladic_origin] (ladic_var) {Non-commutative variables\\[1ex] $X, Y$\\[1ex] correspond to generators\\[1ex] $l_0, l_1$ of the fundamental group\\[1ex]$\pi_1\left(\mathbb{P}^1 \setminus \{0,1,\infty\}\right)$};

\node[box, below=1.3cm of ladic_var] (ladic_rel) {Non-linear relation\\[1ex] ${\bf exp}(X){\bf exp}(Y){\bf exp}(Z)=1$\\[1ex] reflecting the Betti nature i.e.,\\[1ex]the topological relation $l_0 l_1 l_{\infty} = 1$\\[1ex]in ~$\pi_1\left(\mathbb{P}^1 \setminus \{0,1,\infty\}\right)$};

\node[box, below=1.28cm of ladic_rel] (ladic_eq) {Mixed weight \\ functional equations};

\node[box, right=1.5cm of ladic_origin] (comp_origin) {Geometric origin:\\[1ex] The formal KZ equation\\[1ex] $\dfrac{d}{d z}G(z)=\left( \dfrac{X}{z}+\dfrac{Y}{z-1} \right)G(z)$};

\node[box, below=1cm of comp_origin] (comp_var) {Non-commutative variables\\[1ex] $X, Y$\\[1ex] correspond to duals of the \\ differential forms\\[1ex] $\left(\dfrac{dz}{z}\right)^{\ast}, \left(\dfrac{dz}{z-1}\right)^{\ast}$\\[1ex] {\small or generators $\overline{l}_0, \overline{l}_1$ of\\[1ex]$\pi_1^\ab\left(\mathbb{P}^1 \setminus \{0,1,\infty\}\right)$}};

\node[box, below=1cm of comp_var] (comp_rel) {Linear relation\\[1ex] $X+Y+Z=0$\\[1ex] reflecting the de Rham nature\\[1ex] i.e., the residue theorem};

\node[box, below=1.2cm of comp_rel] (comp_eq) {Pure weight \\ functional equations};

\node[font=\large\bfseries, above=0.3cm of ladic_origin] (ladic) {$\ell$-adic Galois (\'Etale) Side};
\node[font=\large\bfseries, above=0.2cm of comp_origin] (complex) {Complex (Hodge) Side};

\draw[arrow] (ladic_origin) -- (ladic_var);
\draw[arrow] (ladic_var) -- (ladic_rel);
\draw[arrow] (ladic_rel) -- node[right, font=\small] {yields} (ladic_eq);

\draw[arrow] (comp_origin) -- (comp_var);
\draw[arrow] (comp_var) -- (comp_rel);
\draw[arrow] (comp_rel) -- node[right, font=\small] {yields} (comp_eq);


\end{tikzpicture}
\end{figure}

\begin{rem}\label{arithm}
Furthermore,
it should be emphasized that these lower-weight terms in the $\ell$-adic Spence-Kummer equation (Theorem \ref{main2} (d-$\ell$)) are not merely theoretical inevitabilities, but also possess essential arithmetic significance.
In Remark \ref{Zl-test},
by taking
\[-\frac{2}{3}\rho_{y,\gamma_9}(\sigma) + \frac{2}{3}\rho_{y,\gamma_9}(\sigma)^3\]
appearing on the right-hand side of the integral $\ell$-adic Spence-Kummer equation (d'-$\ell$) in this paper as an example,
we demonstrate that the inseparable intertwining of this lower-weight term and the matching full-weight term is indispensable for guaranteeing the $\Z_\ell$-integrality of the functional equation.
The presence of these lower-weight terms does not indicate any unnatural feature of the definition of the $\ell$-adic Galois polylogarithm $Li_k^\ell(z)$, but rather it is evidence that the desirable $\Z_\ell$-integral arithmetic information specific to the $\ell$-adic Galois setting is fully captured.
\end{rem}

\begin{rem}\label{rem-non-Fano} 
The name ``non-Fano'' is used for the arrangement
$s_1 s_2(1-s_1)(1-s_2)(s_1-s_2)(1-s_1 s_2)$
because of the non-Fano matroid, named after Gino Fano.
Through the change of variables $s'_1=\frac{1}{1-s_1}$ and $s'_2=\frac{1}{1-s_2}$,
the variety $V_{\rm non\text{-}Fano}$ is isomorphic to the complement of the affine line arrangement
$s'_1 s'_2 (1-s'_1)(1-s'_2)(s'_1-s'_2)(s'_1+s'_2-1)$.
This affine line arrangement is the decone
of the realization of the non-Fano matroid $F_7^{-}$ (cf. \cite[Figure 1.15(a)]{O11}, \cite[Example 10.5]{Su01}).
\end{rem}

\begin{rem}\label{gon}
The algebraic variety $V_{\rm non\text{-}Fano}$ has the following moduli interpretation.
In \cite{Go91}, \cite{Go95}, \cite{Go00},
Goncharov considered the space
\begin{align}
\mathscr{M}_{G}:=\text{the moduli space of certain ordered $7$-point configurations on $\mathbb{P}^2_{{\overline{K}}}$ depicted in FIGURE \ref{confgon}}
\end{align}
to derive the Spence--Kummer functional equation for the real-valued trilogarithm.
In fact,
$\mathscr{M}_{G}$ is isomorphic to $V_{\rm non\text{-}Fano}$, as follows:
Since no three of the four points $x_1, x_2, x_3, z$ are collinear,
there is a unique projective transformation sending the standard frame, namely, $[1 : 0 : 0]$, $[0 : 1 : 0]$, $[0 : 0 : 1]$ and $[1 : 1 : 1]$, to $x_1$, $x_2$, $x_3$ and $z$, respectively.
By this projective transformation,
$y_2$ is transferred to $[0 : 1 : 1]$,
and the images of $y_1$ and $y_3$ are denoted as
$[1 : u_1 : 0]$ and $[1 : 0 : u_2]$, respectively.
Then,
$\mathscr{M}_{G}$ is identified with the affine variety
\begin{align}
\Spec\left({\overline{K}}\left[u_1,u_2,\frac{1}{u_1 u_2(1-u_1)(1-u_2)(u_1+u_2)(u_1 u_2-u_1-u_2)}\right]\right)
\end{align}
by sending
$[(x_1,x_2,x_3,y_1,y_2,y_3,z)] \mapsto (u_1,u_2)$.
This affine variety is isomorphic to $V_{\rm non\text{-}Fano}$ by
\begin{align}
u_1=s_1, \quad u_2=f_5(s_1,s_2)\left(=\frac{s_1(1-s_2)}{s_1-1}\right).
\end{align}
\end{rem}

\begin{figure}[hbtp]
\centering
\caption{Goncharov's $7$-point configuration on $\mathbb{P}^2$}
\label{confgon}
\begin{tikzpicture}[scale=1]
\node (emp) at (0, 4.5) {};
\draw[-] (-4,0) to (3,0);
\draw[-] (1,4) to (-4,-1);
\draw[-] (-1,4) to (3,-1);
\draw[-] (0,4) to (-0.5,-1);
\draw[dotted] (-4,-0.4) to (2,2);
\draw[dotted] (3,-0.36) to (-2,1.8);
\node (b1) at (-0.4, 0) {$\bullet$};
\node (b2) at (-0.1, 2.9) {$\bullet$};
\node (b3) at (2.2, 0) {$\bullet$};
\node (b4) at (-3, 0) {$\bullet$};
\node (b6) at (-0.3, 1.07) {$\bullet$};
\node (b7) at (1.5, 0.85) {$\bullet$};
\node (b5) at (-1, 2) {$\bullet$};
\node (pb1) at (-0.7, -0.3) {$y_2$};
\node (pb2) at (0.4, 2.9) {$x_1$};
\node (pb3) at (1.9, -0.3) {$x_3$};
\node (pb4) at (-3.3, 0.3) {$x_2$};
\node (pb6) at (0.1, 1.1) {$z$};
\node (pb7) at (1.8, 1.2) {$y_3$};
\node (pb5) at (-1.3, 2.2) {$y_1$};
\end{tikzpicture}
\end{figure}

\begin{rem}\label{speci}
The trilogarithm satisfies a 22-term functional equation in three variables discovered by Goncharov \cite{Go91}, which is conjectured to generate all functional equations for the trilogarithm. Suitably specializing one of these variables reduces the Goncharov 22-term functional equation to the Spence--Kummer 9-term functional equation,
which is expected to be a strictly weaker identity.
Nevertheless, the algebraic variety $V_{\rm non\text{-}Fano}$ exhibits rich symmetries that yield numerous functional equations for trilogarithms arising as specializations of the Spence--Kummer equation. For example, combining the $x=y$ boundary specialization of this equation with the inversion formula for polylogarithms recovers a particular case of the distribution formula \cite{NW20a}. Furthermore, applying the change of variables $x=YZ$ and $y=-Y/Z$ produces the Newman equation \cite{Ne1892}, \cite[(6.131)]{Lew81}. 
Moreover, combining the Spence-Kummer equation specialized to the boundary $x=0$ with the inversion formula for polylogarithms leads to the derivation of the Landen equation for trilogarithms (refer to Remark \ref{landen} for a detailed discussion). 
\end{rem}

\begin{figure}[htbp]
\centering
\caption{Relationships among functional equations and their underlying geometry}
\label{fig:trilogeq}
\vspace{0.3cm}
\resizebox{\linewidth}{!}{%
\begin{tikzpicture}[
  eqbox/.style={
    draw=eqblueborder, fill=eqbluefill, rounded corners=7pt,
    line width=1.2pt, text width=#1, align=center, inner sep=6pt, font=\small
  }, eqbox/.default=4.4cm,
  geobox/.style={
    draw=geoborder, fill=geofill, rounded corners=5pt,
    line width=0.9pt, text width=#1, align=center, inner sep=5pt, font=\footnotesize
  }, geobox/.default=4.4cm,
  darr/.style={->, >=Stealth, line width=1.3pt, color=arrowblue},
  carr/.style={->, >=Stealth, line width=1.1pt, dashed, color=arrowred,
    dash pattern=on 5pt off 3pt},
  garr/.style={->, >=Stealth, line width=2.5pt, color=arrowgold}
]


\node[eqbox=6.8cm, fill=eqbluefill!60!white, line width=1.8pt, inner sep=5pt] (G) at (-3, 0.5) {%
  \textbf{Goncharov's 22-term}\\
  \textbf{trilog. functional equation}\\[1pt]
  {\footnotesize (3 variables) \cite{Go91}}%
};
\node[geobox=6.8cm, anchor=north, line width=1.8pt, inner sep=5pt] (Ggeo) at (G.south) {%
  {\color{geotext}\itshape underlying geometry:}\\[2pt]
  the moduli space of special ordered 7-point configurations on $\mathbb{P}^2$ (cf. \cite[FIGURE 1]{Go91})%
};

\node[eqbox=6cm, fill=eqbluefill!60!white, line width=1.8pt, inner sep=5pt] (SC) at (5.0, 0.5) {%
  \textbf{Schaeffer's 5-term}\\
  \textbf{dilog. functional equation}\\[1pt]
  {\footnotesize (2 variables) \cite{Sc1846}, Theorem \ref{c-dilog}~(a-$\bC$), Theorem \ref{l-dilog}~(a-$\ell$)}%
};
\node[geobox=6cm, anchor=north, line width=1.8pt, inner sep=5pt] (SCgeo) at (SC.south) {%
  {\color{geotext}\itshape underlying geometry:}\\[2pt]
  The complement of $s_1 s_2(1-s_1)(1-s_2)(s_1-s_2)=0$,
i.e., the moduli space $M_{0,5}=\Conf_2\left({\bP}^{1} \backslash \{0,1,\infty\}\right)$%
};

\node[eqbox=5.5cm] (KK) at (7.5, -4.3) {%
  \textbf{Kummer's 5-term}\\
  \textbf{dilog. functional equation}\\[1pt]
  {\footnotesize (2 variables) \cite{K1840}, Theorem \ref{c-dilog}~(b-$\bC$), Theorem \ref{l-dilog}~(b-$\ell$)}%
};
\node[geobox=5.5cm, anchor=north] (KKgeo) at (KK.south) {%
  {\color{geotext}\itshape underlying geometry:}\\[2pt]
  $V_{\mathrm{non\text{-}Fano}}$%
};

\node[eqbox=5.5cm] (H) at (7.5, -8.3) {%
  \textbf{Hill's 5-term}\\
  \textbf{dilog. functional equation}\\[1pt]
  {\footnotesize (2 variables) \cite{H1830}, Theorem \ref{c-dilog}~(c-$\bC$), Theorem \ref{l-dilog}~(c-$\ell$)}%
};
\node[geobox=5.5cm, anchor=north] (Hgeo) at (H.south) {%
  {\color{geotext}\itshape underlying geometry:}\\[2pt]
  The complement of $s_1 s_2(1-s_1)(1-s_2)(1-s_1 s_2)=0$%
};

\node[eqbox=6.8cm, fill=eqbluefill!60!white, line width=1.8pt, inner sep=10pt]
  (SK) at (-3, -4.5) {%
  \Large\textbf{Spence--Kummer's 9-term}\\
  \Large\textbf{trilog. functional equation}\\[3pt]
  {\normalsize (2 variables) \cite{Sp1809}, \cite{K1840}, Theorem \ref{main1}~(d-$\bC$), Theorem \ref{main2}~(d-$\ell$)}%
};
\node[geobox=6.8cm, anchor=north, line width=1.2pt]
  (SKgeo) at (-3, -6.5) {%
  {\color{geotext}\itshape underlying geometry:}\\[3pt]
  {\Large\textbf{$V_{\mathrm{non\text{-}Fano}}$}}\\[3pt]
  {\large\color{geotext}\bfseries
    Rich symmetries that lead to many formulas}%
};

\draw[garr] (SKgeo.north) -- (SK.south);

\node[eqbox] (N) at (-4.4, -10.9) {%
  \textbf{Newman's equation}\\
  \textbf{for trilogarithms}\\[1pt]
  {\footnotesize (3 variables~$+$~1 relation) \cite{Ne1892}}%
};
\node[geobox, anchor=north] (Ngeo) at (N.south) {%
  {\color{geotext}\itshape underlying geometry:}\\[2pt]
  $X+Y+Z=XYZ$%
};

\node[eqbox] (D) at (-2.3, -14.7) {%
  \textbf{Distribution formula}\\
  \textbf{for trilogarithms}\\[1pt]
  {\footnotesize (1 variable) \cite{NW20a}}%
};
\node[geobox, anchor=north] (Dgeo) at (D.south) {%
  {\color{geotext}\itshape underlying geometry:}\\[2pt]
  $\mathbb{P}^1\!\setminus\!\{0,1,-1,\infty\}$\\[4pt]
  {\scriptsize covering space:\enspace$z\mapsto z^2$}%
};

\node[eqbox] (L) at (8, -12.9) {%
  \textbf{Landen's equation}\\
  \textbf{for trilogarithms}\\[1pt]
  {\footnotesize (1 variable) \cite{L1780}, \cite{NS25}, Remark \ref{landen}}%
};
\node[geobox, anchor=north] (Lgeo) at (L.south) {%
  {\color{geotext}\itshape underlying geometry:}\\[2pt]
  $\mathbb{P}^1\!\setminus\!\{0,1,\infty\}$\\[4pt]
  \tikz[baseline=-0.3ex, scale=0.6]{
    \node[inner sep=1pt] (z) at (-1.6,0)   {\scriptsize $0$};
    \node[inner sep=1pt] (o) at (0,0) {\scriptsize $1$};
    \draw[<->, >=Stealth, blue!80, semithick] (z) to[bend left=40] (o);
    \node[inner sep=1pt] (o2) at (1.6,0) {\scriptsize $\infty$};
    \draw[<->, >=Stealth, red!80, semithick] (o) to[bend left=40] (o2);
  }\\[2pt]
  {\scriptsize sym:\enspace$z\mapsto 1-z$ and $z\mapsto\dfrac{z}{z-1}$}%
};

\node[eqbox=3.8cm] (I) at (3, -14.2) {%
  \textbf{Inversion formula}\\
  \textbf{for trilogarithms}\\[1pt]
  {\footnotesize (1 variable) \cite{NW12}}%
};
\node[geobox=3.8cm, anchor=north] (Igeo) at (I.south) {%
  {\color{geotext}\itshape underlying geometry:}\\[2pt]
  $\mathbb{P}^1\!\setminus\!\{0,1,\infty\}$\\[4pt]
  \tikz[baseline=-0.3ex, scale=0.6]{
    \node[inner sep=1pt] (z) at (0.4,0)   {\scriptsize $0$};
    \node[inner sep=1pt] (o) at (2,0) {\scriptsize $\infty$};
    \draw[<->, >=Stealth, purple!80, semithick] (z) to[bend left=40] (o);
  }\\[2pt]
  {\scriptsize sym:\enspace$z\mapsto \frac{1}{z}$}%
};

\begin{scope}[on background layer]
  \node[fill=bgcard, rounded corners=10pt, inner sep=4pt, fit=(G)(Ggeo)] {};
  \node[fill=bgcard, rounded corners=10pt, inner sep=4pt, fit=(SK)(SKgeo)] {};
  \node[fill=bgcard, rounded corners=10pt, inner sep=4pt, fit=(N)(Ngeo)] {};
  \node[fill=bgcard, rounded corners=10pt, inner sep=4pt, fit=(L)(Lgeo)] {};
  \node[fill=bgcard, rounded corners=10pt, inner sep=4pt, fit=(I)(Igeo)] {};
  \node[fill=bgcard, rounded corners=10pt, inner sep=4pt, fit=(D)(Dgeo)] {};
  \node[fill=bgcard, rounded corners=10pt, inner sep=4pt, fit=(SC)(SCgeo)] {};
  \node[fill=bgcard, rounded corners=10pt, inner sep=4pt, fit=(KK)(KKgeo)] {};
  \node[fill=bgcard, rounded corners=10pt, inner sep=4pt, fit=(H)(Hgeo)] {};
\end{scope}


\draw[darr] (Ggeo.south) --
  node[right=5pt, font=\scriptsize, black, align=left]{%
    specialization of\\[-1pt]one variable%
  } (SK.north);

\draw[darr] (SKgeo.south) to[]
  node[left=5pt, font=\scriptsize, black, align=left, pos=0.55]{%
    $x = YZ,\enspace y = -Y/Z$%
  } (N.north);

\node[circle, draw=arrowred, fill=white, inner sep=1.5pt, line width=1pt]
  (Merge2) at (-1.2, -11) {\tiny\bfseries$+$};

\draw[darr, line width=1.5pt] (SKgeo.south) to[out=330, in=90]
  node[right=5pt, font=\scriptsize, black, align=left, pos=0.7]{$x = y$}
  (Merge2);
\draw[carr, line width=1.5pt] (Merge2) to[out=-90, in=35]
  node[right=5pt, font=\scriptsize, black, align=left, pos=0.5]{}
  (D.north);

\node[circle, draw=arrowred, fill=white, inner sep=1.5pt, line width=1pt]
  (Merge) at (2.0, -11) {\tiny\bfseries$+$};

\draw[darr] (SKgeo.south) to[out=355, in=135]
  node[right=5pt, font=\scriptsize, black, align=left, pos=0.7]{$x \to 0$}
  (Merge);

\draw[carr, line width=1.5pt] (I.north) -- (Merge);

\draw[carr, line width=1.5pt] (I.north) -- (Merge2);

\draw[carr, line width=1.5pt] (Merge) -- (L.west);

\draw[garr] (SKgeo.east) -- (SCgeo);

\draw[garr] (SKgeo.east) -- (KK.west);

\draw[garr] (SKgeo.east) -- (H.west);

\draw[rulecol, thin] (-4.5, -17.5) -- (6.5, -17.5);

\node[anchor=north, font=\scriptsize, align=left, inner sep=0pt]
  at (1.0, -17.85) {%
  \begin{tabular}{@{}l@{\enspace}l@{}}
    \tikz[baseline=-0.6ex]{\draw[darr, line width=1.0pt](0,0)--(1.3cm,0);}
    & derivation (specialization\,/\,change of variables) \\[5pt]
    \tikz[baseline=-0.6ex]{\draw[carr, line width=1.0pt](0,0)--(1.3cm,0);}
    & derivation via combination with inversion formula \\[5pt]
    \tikz[baseline=-0.6ex]{\draw[garr](0,0)--(1.1cm,0);}
    & \textbf{geometric derivation using $V_{\mathrm{non\text{-}Fano}}$
        performed in this paper}
  \end{tabular}%
};

\end{tikzpicture}%
}
\end{figure}

\begin{rem}\label{history}
Historically,
the study of the functional equations of $Li_3$ was initiated in the late 18th century by Landen \cite{L1780} and others (cf. \cite[Chapter 6]{Lew81}).
Modern treatments of polylogarithms have been presented by
Zagier \cite{Z91}, Goncharov \cite{Go91}, Wojtkowiak \cite{W91}, Beilinson-Deligne \cite{BD94}, Gangl \cite{Ga03} and others
since the last decade of the twentieth century,
where the Bloch--Wigner--Ramakrishnan polylogarithms are recognized as the 
main components of 
the regulator maps in motivic cohomology theory.
Relations between the Spence--Kummer equation and the non-Fano arrangement 
are noted in more recent works of the twenty-first century (cf. \cite{Pe12},
\cite{Pi05},
\cite{Pi21},
\cite{Pi22},
\cite{R02}) in the context of web geometry and cluster algebras.
\end{rem}

This paper is organized as follows.
The notations used in this paper are described in Section 1.
In Section 2, we apply the Galois theory for $M_{0,5}$ to compute the fundamental groups of ${V}_{\rm B_3}$ and $V_{\rm non\text{-}Fano}$.
This plays an important role in the analysis of lower-degree terms of polylogarithm generating series.
In Section 3,
we review
the basic properties of the complex and $\ell$-adic Galois polylogarithms.
Finally, in Section 4, we prove our main theorems.
Using the geometry of $V_{\rm non\text{-}Fano}$ and the algebraic relations among the polylogarithm generating series, 
we derive the Spence--Kummer trilogarithm functional equation (d-$\bC$) together with the dilogarithm functional equations of Schaeffer's, Kummer's, and Hill's types (a,b,c-$\bC$)
in a geometric and algebraic manner based on \cite{NW12}.
The proofs of their $\ell$-adic Galois analogues (a,b,c,d-$\ell$) are given in parallel with the proofs for the complex case.
We also perform $\Z_\ell$-integrality tests for the functional equations (a,b,c,d-$\ell$) in terms of the generalized $\ell$-adic Soul\'e character.

\section*{Acknowledgements}\noindent
The author would like to express deep gratitude to Hiroaki Nakamura for his warm encouragement, suggestions,
and generous advice.
The author would also like to express sincere thanks to Hiroyuki Ogawa for providing access to the computer algebra system Magma, which was necessary for the computation of the fundamental groups in Section 2,
and for his helpful advice on using Magma.
This work was supported by JSPS KAKENHI Grant Numbers JP20J11018, JP24K22840 and JP26K16966.

\section{Notation}\label{Notations}
Throughout the present paper,
we fix a prime number $\ell$
and a subfield $K$ of $\mathbb{C}$.
Let $\overline{K}$ be the algebraic closure of $K$ in $\mathbb{C}$ and
$G_K:={\rm{Gal}}(\overline{K}/K)$ be the absolute Galois group of $K$.
For an algebraic variety $V$ over $\overline{K}$,
we denote by $V^{\mathrm{an}}$ the complex analytic space associated to $V$ via the inclusion $\overline{K} \hookrightarrow \mathbb{C}$.
Moreover,
given a morphism $f: V \to W$ of algebraic varieties over $\overline{K}$,
we denote by $f^{\mathrm{an}}: V^{\mathrm{an}} \to W^{\mathrm{an}}$ the induced morphism between the associated complex analytic spaces.

\subsection{Spaces}
First, we introduce the spaces that are used in the present paper.
Let 
\begin{eqnarray}
{\bP}^{1}_{\overline{K}} \backslash \{0,1,\infty\} = \Spec\left(\overline{K}\left[t,\frac{1}{t(1-t)}\right]\right)
\end{eqnarray}
be the projective line minus three points over $\overline{K}$,
and let
\begin{eqnarray}
{M_{0,5}} = \left\{(a_1,a_2,a_3,a_4,a_5) \in {\left(\bP^1_{\overline{K}}\right)}^5 \mid a_i \neq a_j~(i \neq j)\right\}/{{\rm PGL}(2,\overline{K})}
\end{eqnarray}
be the moduli space of ${\bP}^{1}_{\overline{K}}$ with five ordered points.
Hereafter,
we identify ${M_{0,5}}$ with the second configuration space $\Conf_2\left({\bP}^{1}_{\overline{K}} \backslash \{0,1,\infty\}\right)$ by sending
\begin{eqnarray}
[(a_1,a_2,a_3,a_4,a_5)] = [(1,t_1,t_2,0,\infty)] \mapsto (t_1,t_2).
\end{eqnarray}
Thus, ${M_{0,5}}$ can be regarded as the complement to the Braid arrangement $t_1 t_2 (1-t_1)(1-t_2)(t_1-t_2)$:
\begin{eqnarray}\label{M05}
{M_{0,5}} &=& \Conf_2\left({\bP}^{1}_{\overline{K}} \backslash \{0,1,\infty\}\right) \\
&=& \Spec\left({\overline{K}}\left[t_1,t_2,\frac{1}{t_1 t_2 (1-t_1)(1-t_2)(t_1-t_2)}\right]\right).\notag
\end{eqnarray}
We write
\begin{eqnarray}
{V}_{\rm B_3} &=& \Spec\left({\overline{K}}\left[s_1,s_2,\frac{1}{s_1 s_2(1-s_1^2)(1-s_2^2)(s_1-s_2)(1-s_1 s_2)}\right]\right)
\end{eqnarray}
for the complement to the arrangement $s_1 s_2(1-s_1^2)(1-s_2^2)(s_1-s_2)(1-s_1 s_2)$.
Through the change of variables $s'_1=\frac{1-s_1}{1+s_1}$ and $s'_2=\frac{1-s_2}{1+s_2}$,
the affine variety ${V}_{\rm B_3}$ is shown to be isomorphic to the complement of the Coxeter arrangement $s'_1 s'_2(1-{s'}_1^2)(1-{s'}_2^2)(s'_1-s'_2)(s'_1+s'_2)$ of type ${\rm B_3}$.
Then,
\begin{eqnarray}\label{fcov}
f_{\rm cov}: {V}_{\rm B_3} \to {M_{0,5}},\quad (s_1,s_2) \mapsto \left({\left(\frac{1-s_1}{1+s_1}\right)}^2,{\left(\frac{1-s_2}{1+s_2}\right)}^2\right)=(t_1,t_2)
\end{eqnarray}
is a finite \'etale Galois covering space with Galois group $\bZ/2\bZ \times \bZ/2\bZ$.
There is a natural open immersion from ${V}_{\rm B_3}$ to $V_{\rm non\text{-}Fano}$:
\begin{eqnarray}
{V}_{\rm B_3} \hookrightarrow
V_{\rm non\text{-}Fano}.
\end{eqnarray}
Through this open immersion,
we regard the points and \'etale paths on ${V}_{\rm B_3}$ as those on $V_{\rm non\text{-}Fano}$.

\begin{figure}[hbtp]
\centering
\caption{Line arrangements appearing in the present paper}
\label{fig:arr}
\begin{tikzpicture}[scale=0.7]
\draw[-] (-2.5,1.5) to (2.5,1.5);
\draw[-] (-2.5,0.5) to (2.5,0.5);
\draw[-] (-2.5,-0.5) to (2.5,-0.5);
\draw[-] (-2.5,-1.5) to (2.5,-1.5);
\draw[-] (1.5,-2.5) to (1.5,2.5);
\draw[-] (0.5,-2.5) to (0.5,2.5);
\draw[-] (-0.5,-2.5) to (-0.5,2.5);
\draw[-] (-1.5,-2.5) to (-1.5,2.5);
\draw[-] (-0.5,1.5) to (1.5,-0.5);
\draw[-] (1.5,-0.5) to  [out=-45,in=170] (2.5,-1);
\draw[-] (-1,2.5) to  [out=-80,in=135] (-0.5,1.5);
\draw[-] (-2.5,-1) to  [out=-10,in=135] (-1.5,-1.5);
\draw[-] (-1.5,-1.5) to  [out=-45,in=100] (-1,-2.5);
\draw[-] (2.5,2.5) to (-2.5,-2.5);
\node (c-1) at (-1.5, -2.7) {\tiny $-1$};
\node (c0) at (-0.5, -2.7) {\tiny $0$};
\node (c1) at (0.5, -2.7) {\tiny $1$};
\node (cinf) at (1.5, -2.7) {\tiny $\infty$};
\node (cc-1) at (-2.8,-1.5) {\tiny $-1$};
\node (cc0) at (-2.7,-0.5) {\tiny $0$};
\node (cc1) at (-2.7,0.5) {\tiny $1$};
\node (ccinf) at (-2.7,1.5) {\tiny $\infty$};
\node (cL1) at (-2.7,-2.7) {\tiny $s_1=s_2$};
\node (cL2) at (-3.35,-1) {\tiny $s_1 s_2=1$};
\node (cL3) at (3.35,-1) {\tiny $s_1 s_2=1$};

\draw[-] (-5,1) to (-9,1);
\draw[-] (-5,0) to (-9,0);
\draw[-] (-5,-1) to (-9,-1);
\draw[-] (-8,-2) to (-8,2);
\draw[-] (-7,-2) to (-7,2);
\draw[-] (-6,-2) to (-6,2);
\draw[-] (-5,2) to (-9,-2);
\draw[-] (-5,-2) to (-9,2);
\node (a0) at (-8, -2.2) {\tiny $0$};
\node (a1) at (-7, -2.2) {\tiny $1$};
\node (ainf) at (-6, -2.2) {\tiny $\infty$};
\node (aa0) at (-9.2, -1) {\tiny $0$};
\node (aa1) at (-9.2, 0) {\tiny $1$};
\node (aainf) at (-9.2, 1) {\tiny $\infty$};
\node (cL1) at (-9,-2.2) {\tiny $s_1=s_2$};
\node (cL1) at (-4.7,-2.2) {\tiny $1=s_1 s_2$};

\node () at (0,2.7) {};

\draw[-] (5,1) to (9,1);
\draw[-] (5,0) to (9,0);
\draw[-] (5,-1) to (9,-1);
\draw[-] (6,-2) to (6,2);
\draw[-] (7,-2) to (7,2);
\draw[-] (8,-2) to (8,2);
\draw[-] (5,-2) to (9,2);
\node (bL1) at (5,-2.2) {\tiny $t_1=t_2$};
\node (b0) at (6,-2.2) {\tiny $0$};
\node (b1) at (7,-2.2) {\tiny $1$};
\node (binf) at (8,-2.2) {\tiny $\infty$};
\node (b00) at (9.2,-1) {\tiny $0$};
\node (b11) at (9.2,0) {\tiny $1$};
\node (binfinf) at (9.2,1) {\tiny $\infty$};

\node (text1) at (-7, -3.5) {Non-Fano arrangement};
\node (text2) at (0, -3.5) {${\rm B_3}$ {arrangement}};
\node (text3) at (7, -3.5) {Braid arrangement};
\end{tikzpicture}
\end{figure}

\begin{rem} 
We note that a covering space isomorphic to (\ref{fcov}) is discussed in \cite[3.6.2]{M97}.
\end{rem} 

\begin{rem} 
The spaces
$V_{\rm non\text{-}Fano}^\an$,
$V_{\rm B_3}^\an$,
and $M_{0,5}^{\an}$
are derived from complex arrangements called the {non-Fano arrangement},
the ${\rm B_3}$ {arrangement},
and the {Braid arrangement}, respectively. 
More precisely,
the decones of what are called the Braid arrangement, 
the ${\rm B_3}$ arrangement, and the non-Fano arrangement in \cite{Su01}, with certain changes in variables, are called the Braid arrangement, the ${\rm B_3}$ arrangement, and the non-Fano arrangement in this paper.
\end{rem}

\subsection{Base points}
We fix the tangential base points of these spaces.
We write $K((t))$ for the field of Laurent power series over $K$.
Let
\begin{eqnarray}\label{tanbv3}
\vec{v}:\Spec\Big({K}((t))\Big) \to {V}_{\rm B_3}
\end{eqnarray}
be the $K$-rational tangential base point of ${V}_{\rm B_3}$ over the $K(t)$-rational point $(t^2,t)$.
Composing $\vec{v}:\Spec\Big({K}((t))\Big) \to {V}_{\rm B_3}$ with the open immersion ${V}_{\rm B_3} \hookrightarrow V_{\rm non\text{-}Fano}$,
we obtain the $K$-rational tangential base point of $V_{\rm non\text{-}Fano}$, which is also denoted by $\vec{v}$:
\begin{eqnarray}\label{tanbv}
\vec{v}:\Spec\Big({K}((t))\Big) \to V_{\rm non\text{-}Fano}.
\end{eqnarray}
The image of $\vec{v}:\Spec\Big({K}((t))\Big) \to V_{\rm non\text{-}Fano}$ under ${\rm pr}_2\left(=f_9\right): V_{\rm non\text{-}Fano} \to {\bP}^{1}_{\overline{K}} \backslash \{0,1,\infty\}, (s_1,s_2) \mapsto s_2$
is the standard $K$-rational tangential base point 
$
\overrightarrow{01}: \Spec\Big({K}((t))\Big) \to {\bP}^{1}_{\overline{K}} \backslash \{0,1,\infty\}.
$
We set
\begin{eqnarray}
(\tau_1,\tau_2):=f_{\rm cov}(t^2,t)=\left(\left(\frac{1-t^2}{1+t^2}\right)^2,\left(\frac{1-t}{1+t}\right)^2\right) \in M_{0,5}
\end{eqnarray}
and write
\begin{eqnarray}
\vec{\tau}: \Spec\Big({K}((t))\Big) \to M_{0,5}
\end{eqnarray}
for the $K$-rational tangential base point of $M_{0,5}$ over the $K(t)$-rational point $(\tau_1,\tau_2)$.
\begin{center}
\begin{tikzpicture}[scale=0.7]
\node (v) at (0, 2.5) {$\Spec\Big({K}((t))\Big)$};
\node (VB3) at (0, 0) {$\left({V}_{\rm B_3},(t^2,t) \right)$};
\node (Conf) at (6, 0) {$\left(M_{0,5},(\tau_1,\tau_2)\right).$};
\node (VnF) at (-6, 0) {$\left(V_{\rm non\text{-}Fano},(t^2,t)\right)$};
\node (P) at (-12, 0) {$\left({\bP}^{1}_{\overline{K}} \backslash \{0,1,\infty\},t\right)$};
\node (f9) at (-8.8, -0.4) {\scriptsize ${\rm pr}_2$};
\node (phi) at (2.9, -0.4) {\scriptsize $f_{\rm cov}$};
\draw[left hook ->] (VB3) to (VnF);
\draw[->>] (VB3) to (Conf);
\draw[->] (v) to (VB3);
\draw[->] (VnF) to (P);
\draw[->] (v) to (P);
\draw[->] (v) to (VnF);
\draw[->] (v) to (Conf);
\node (empty) at (0.4, 1.25) {$\vec{v}$};
\node (empty) at (-2.7, 1) {$\vec{v}$};
\node (empty) at (3.3, 1.5) {$\vec{\tau}$};
\node (empty) at (-6, 1.8) {$\overrightarrow{01}$};
\end{tikzpicture}
\end{center}
When discussing the complex case,
we regard these tangential base points of the algebraic varieties as those of the associated complex analytic spaces via the embedding $\overline{K} \hookrightarrow \mathbb{C}$.

\subsection{Fundamental groups}
Next, we establish some notation concerning fundamental groups.
For an algebraic variety $V$ over ${\overline{K}}$ and two (possibly tangential) base points
$\ast,
\ast'$ of $V$,
we write
\[
\pi_1^\ellet(V;\ast,\ast')
\]
for the pro-$\ell$ set of pro-$\ell$ \'etale paths on $V$ from $\ast$ to $\ast'$.
Let $V^{\an}:=V(\bC)$ be the associated analytic space of $V$.
For two (possibly tangential) base points
$\ast,
\ast'$ of $V^{\an}$,
we write
\[
\pi_1^{\rm top}(V^{\an};\ast,\ast')
\]
for the set of homotopy classes of topological paths on
$V^{\an}$ from $\ast$ to $\ast'$.
If
$\gamma_1 \in \pi_1^{\rm top}(V^{\an};\ast,\ast')$ and
$\gamma_2 \in \pi_1^{\rm top}(V^{\an};\ast',\ast'')$,
the composite
\begin{align}\label{pathcomp}
\gamma_1 \cdot \gamma_2:=\gamma_1 \gamma_2 \in \pi_1^{\rm top}(V^{\an};\ast,\ast'')
\end{align}
is defined 
so that paths are composed sequentially starting from left to right.
The composite of pro-$\ell$ \'etale paths is defined in the same way.
Moreover,
we write
\[
\pi_1^\ellet(V,\ast):=\pi_1^\ellet(V;\ast,\ast), \quad
\pi_1^{\rm top}(V^{\an},\ast):=\pi_1^{\rm top}(V^{\an};\ast,\ast)
\]
for the pro-$\ell$ fundamental group of $V$ and the topological fundamental group of $V^{\an}$,
respectively.
When two points $\ast,
\ast'$ on $V$ are regarded as points of $V^{\an}$ through the inclusion
$\overline{K} \hookrightarrow \bC$,
there is a canonical comparison map
\begin{align}\label{comp}
\pi_1^{\rm top}\left(V^{\an}; \ast, \ast' \right)
\to
\pi_1^\ellet(V;\ast,\ast').
\end{align}
The comparison map (\ref{comp}) induces an isomorphism between $\pi_1^\ellet(V,\ast)$ and the pro-$\ell$ completion of $\pi_1^{\rm top}(V^{\an},\ast)$,
and we identify 
these pro-$\ell$ groups
via this isomorphism.
\begin{figure}[hbtp]
\centering
\caption{Topological loops on ${\bP}^{1}(\bC) \backslash \{0,1,\infty\}$}
\label{fig:path2}
\begin{tikzpicture}[scale=1]


\draw (0.1,0) -- (0,0.1);
\draw (0.1,0) -- (0.2,0.1);

\draw (-3,0) -- (-3.1,0.1);
\draw (-3,0) -- (-2.9,0.1);

\draw (7.3,0) -- (7.4,-0.1);
\draw (7.3,0) -- (7.2,-0.1);



\draw (2.1,0) to [out=0,in=270] (2.8,1);
\draw (2.8,1) to [out=90,in=-10] (2,1.9);
\draw (2,1.9) to [out=170,in=90] (0.1,0);
\draw (0.1,0) to [out=270,in=190] (2,-1.9);
\draw (2,-1.9) to [out=10,in=270] (2.8,-0.7);
\draw (2.8,-0.7) to [out=90,in=360] (2.1,0);

\draw (2,0) to [out=0,in=210] (4,0.4);
\draw (4,0.4) to [out=30,in=180] (6,1.4);
\draw (6,1.4) to [out=0,in=90] (7.3,0);
\draw (7.3,0) to [out=270,in=0] (6,-1.4);
\draw (4,-0.4) to [out=330,in=180] (6,-1.4);
\draw (4,-0.4) to [out=150,in=0] (2,0);


\draw (2,0) to [out=0,in=270] (4,2);
\draw (4,2) to [out=90,in=0] (1.5,4);
\draw (1.5,4) to [out=180,in=90] (-3,0);
\draw (-3,0) to [out=270,in=180] (-2,-1.2);
\draw (-2,-1.2) to [out=0,in=-90] (-1,0);
\draw (-1,0) to [out=90,in=180] (1.8,3);
\draw (1.8,3) to [out=0,in=90] (3.5,1.5);
\draw (3.5,1.5) to [out=-90,in=0] (2,0);

\draw[dotted](-4.2,0)--(8.5,0);

\fill[white](2,0)circle(0.07);
\fill[white](6,0)circle(0.07);
\fill[white](-2,0)circle(0.07);

\node at (0,4.4) {~};
\node at (-0.2,0) {$l_0$};
\node at (1.8,-0.4) {$0$};
\node at (6.2,-0.4) {$1$};
\node at (7.7,0) {$l_1$};
\node at (-3.4,0) {$l_{\infty}$};
\node at (2,0) {${\circ}$};
\node at (6,0) {${\circ}$};
\node at (-2,0) {${\circ}$};
\node at (-2,-0.4) {$\infty$};
\end{tikzpicture}
\end{figure}

Recall the topological fundamental groups of 
${\bP}^{1}(\bC) \backslash \{0,1,\infty\}$ and 
$M_{0,5}^\an$.
First,
we write $l_0$, $l_1$, and $l_\infty$ for topological loops based at $\overrightarrow{01}$ on
${\bP}^{1}(\bC) \backslash \{0,1,\infty\}$, as in
FIGURE \ref{fig:path2}.
Then, $\pi_1^{\rm top}({\bP}^{1}(\bC) \backslash \{0,1,\infty\},\overrightarrow{01})$ coincides with the free group of rank $2$ with the generating system
$\vec{l}:=(l_0,
l_1)$:
\begin{align}\label{tripod}
\pi_1^{\rm top}({\bP}^{1}(\bC) \backslash \{0,1,\infty\},\overrightarrow{01})&=\left<l_0,
l_1,
l_{\infty} \mid l_0 \cdot l_1 \cdot l_{\infty}=1 \right>\\
&=\left<l_0,
l_1\right>. \notag
\end{align}
Next,
the topological fundamental group of $M_{0,5}^\an$ is the Teichm\"{u}ller modular group (mapping class group) of the Riemann sphere with five marked points (cf. \cite[\S 3.1]{N94}).
This group coincides with the quotient group of the pure braid group $P_4$ by its center $\left\langle \omega_4 \right\rangle$ and 
has the following description (cf. \cite{Lee10}):
\begin{align}
\pi_1^{\rm top}(M_{0,5}^\an,\vec{\tau})
&
=\left\langle
\left. 
\begin{array}{l}
A_{12},A_{13}, A_{14}, \\
A_{23}, A_{24} , A_{34} 
\end{array} 
\right|
\begin{array}{l}
(R1) \sim (R5)
\end{array}
\right\rangle 
\bigg(=P_4/\left\langle \omega_4 \right\rangle\bigg)
\\
&= \left\langle
\left. 
\begin{array}{l}
A_{12},A_{13}, A_{14}, A_{23}, \\
A_{24} , A_{34}, A_{25}, A_{35} 
\end{array} 
\right|
\begin{array}{l}
(R1) \sim (R6)
\end{array}
\right\rangle \notag
\end{align}
where
$\omega_4:=A_{12}A_{13}A_{14}A_{23}A_{24}A_{34}$ and the relations $(R1) \sim (R6)$ are as follows:
\begin{align}
(R1)~&A_{ij}=A_{ji},~A_{ii}=1~(1 \leq i,j \leq 4), \notag\\ \notag
(R2)~&A_{12}A_{34}=A_{34}A_{12},~~A_{14}A_{23}=A_{23}A_{14},\\ \notag
(R3)~&A_{12}A_{13}A_{23}=A_{23}A_{12}A_{13}=A_{13}A_{23}A_{12},\\ \notag
&A_{12}A_{14}A_{24}=A_{14}A_{24}A_{12}=A_{24}A_{12}A_{14},\\ 
&A_{23}A_{24}A_{34}=A_{24}A_{34}A_{23}=A_{34}A_{23}A_{24},\\ \notag
&A_{13}A_{14}A_{34}=A_{14}A_{34}A_{13}=A_{34}A_{13}A_{14},\\ \notag
(R4)~&A_{34}A_{24}A_{14}A_{13}=A_{13}A_{34}A_{24}A_{14},\\ \notag
(R5)~&\omega_4=1,\\ \notag
(R6)~&A_{12}A_{23}A_{24}A_{25}=1,\\ \notag
&A_{13}A_{23}A_{34}A_{35}=1.
\end{align}
In the present paper,
we define homotopy classes
\begin{align}
A_{ij}\left(=A_{ji}\right) \in \pi_1^{\rm top}(M_{0,5}^\an,\vec{\tau})~(1 \leq i,j \leq 5)
\end{align}
as shown below,
using the topological loops $l_{ij}~(i \in \{2,3\},~j \in \{1,2,3,4,5\},~i \neq j)$ in FIGURE \ref{fig:path4}.
\begin{figure}[hbtp]
\centering
\caption{Topological loops on ${\bP}^{1}(\bC) \backslash \{0,1,\infty\}$}
\label{fig:path4}
\begin{tikzpicture}[scale=0.7]
\draw (4,0.7) -- (4.1,0.6);
\draw (4,0.7) -- (4.1,0.8);

\draw (2,0.7) -- (2.1,0.6);
\draw (2,0.7) -- (2.1,0.8);

\draw (0,0.7) -- (0.1,0.6);
\draw (0,0.7) -- (0.1,0.8);

\draw (-2,0.7) -- (-1.9,0.6);
\draw (-2,0.7) -- (-1.9,0.8);

\draw (5.7,0) to [out=180,in=0] (4,0.7);
\draw (4,0.7) to [out=180,in=90] (3.3,0);
\draw (3.3,0) to [out=270,in=180] (4,-0.7);
\draw (4,-0.7) to [out=0,in=180] (5.7,0);
\draw (5.7,0) -- (6.0,0);

\draw (5.8,0) to [out=180,in=45] (4.5,-0.9);
\draw (4.5,-0.9) to [out=225,in=0] (3,-1.5);
\draw (3,-1.5) to [out=180,in=270] (1.3,0);
\draw (1.3,0) to [out=90,in=180] (2,0.7);
\draw (2,0.7) to [out=0,in=90] (2.7,0);
\draw (2.7,0) to [out=270,in=170] (3.5,-1);
\draw (3.5,-1) to [out=-10,in=210] (4.3,-0.9);
\draw (4.3,-0.9) to [out=30,in=180] (5.8,0);

\draw (5.9,0) to [out=180,in=40] (4,-2.5);
\draw (4,-2.5) to [out=225,in=0] (2.5,-3);
\draw (2.5,-3) to [out=180,in=270] (-0.7,0);
\draw (-0.7,0) to [out=90,in=180] (0,0.7);
\draw (0,0.7) to [out=0,in=90] (0.7,0);
\draw (0.7,0) to [out=-90,in=180] (3.1,-2);
\draw (3.1,-2) to [out=0,in=180] (5.9,0);

\draw (6,0) to [out=180,in=40] (4.3,-4);
\draw (4.3,-4) to [out=220,in=0] (2,-4.6);
\draw (2,-4.6) to [out=180,in=270] (-2.7,0);
\draw (-2.7,0) to [out=90,in=180] (-2,0.7);
\draw (-2,0.7) to [out=0,in=90] (-1.3,0);
\draw (-1.3,0) to [out=270,in=180] (2.5,-3.5);
\draw (2.5,-3.5) to [out=0,in=260] (4.7,-2);
\draw (4.7,-2) to [out=80,in=180] (6,0);



\node at (4,1.2) {$l_{23}$};
\node at (2,1.2) {$l_{24}$};
\node at (0,1.2) {$l_{25}$};
\node at (-2,1.2) {$l_{21}$};
\draw[dotted](-3.5,0)--(7.5,0);
\draw[dotted](8.5,0)--(19.5,0);

\fill[white](0,0)circle(0.07);
\fill[white](2,0)circle(0.07);
\fill[white](-2,0)circle(0.07);

\node at (2,-0.4) {$0$};
\node at (4,-0.4) {\small${\tau_2}$};
\node at (6,-0.4) {$\tau_1$};
\node at (0,-0.4) {$\infty$};
\node at (-2,-0.4) {$1$};
\node at (2,0) {${\circ}$};
\node at (6,0) {${\bullet}$};
\node at (4,0) {${\bullet}$};
\node at (0,0) {${\circ}$};
\node at (-2,0) {${\circ}$};

\draw (8,-5) -- (8,2);

\node at (0,2.2) {};

\draw (16,0.7) -- (16.1,0.6);
\draw (16,0.7) -- (16.1,0.8);

\draw (14,0.7) -- (14.1,0.6);
\draw (14,0.7) -- (14.1,0.8);

\draw (12,0.7) -- (12.1,0.6);
\draw (12,0.7) -- (12.1,0.8);

\draw (10,0.7) -- (10.1,0.6);
\draw (10,0.7) -- (10.1,0.8);

\draw (17.7,0) to [out=180,in=0] (16,0.7);
\draw (16,0.7) to [out=180,in=90] (15.3,0);
\draw (15.3,0) to [out=270,in=180] (16,-0.7);
\draw (16,-0.7) to [out=0,in=180] (17.7,0);
\draw (17.7,0) -- (18.0,0);

\draw (17.8,0) to [out=180,in=45] (16.5,-0.9);
\draw (16.5,-0.9) to [out=225,in=0] (15,-1.5);
\draw (15,-1.5) to [out=180,in=270] (13.3,0);
\draw (13.3,0) to [out=90,in=180] (14,0.7);
\draw (14,0.7) to [out=0,in=90] (14.7,0);
\draw (14.7,0) to [out=270,in=170] (15.5,-1);
\draw (15.5,-1) to [out=-10,in=210] (16.3,-0.9);
\draw (16.3,-0.9) to [out=30,in=180] (17.8,0);

\draw (17.9,0) to [out=180,in=40] (16,-2.5);
\draw (16,-2.5) to [out=225,in=0] (14.5,-3);
\draw (14.5,-3) to [out=180,in=270] (11.3,0);
\draw (11.3,0) to [out=90,in=180] (12,0.7);
\draw (12,0.7) to [out=0,in=90] (12.7,0);
\draw (12.7,0) to [out=-90,in=180] (15.1,-2);
\draw (15.1,-2) to [out=0,in=180] (17.9,0);

\draw (18,0) to [out=180,in=40] (16.3,-4);
\draw (16.3,-4) to [out=220,in=0] (14,-4.6);
\draw (14,-4.6) to [out=180,in=270] (9.3,0);
\draw (9.3,0) to [out=90,in=180] (10,0.7);
\draw (10,0.7) to [out=0,in=90] (10.7,0);
\draw (10.7,0) to [out=270,in=180] (14.5,-3.5);
\draw (14.5,-3.5) to [out=0,in=260] (16.7,-2);
\draw (16.7,-2) to [out=80,in=180] (18,0);



\node at (16,1.2) {$l_{34}$};
\node at (14,1.2) {$l_{35}$};
\node at (12,1.2) {$l_{31}$};
\node at (10,1.2) {$l_{32}$};

\fill[white](14,0)circle(0.07);
\fill[white](12,0)circle(0.07);
\fill[white](16,0)circle(0.07);

\node at (14,-0.4) {$\infty$};
\node at (16,-0.4) {\small$0$};
\node at (18,-0.4) {$\tau_2$};
\node at (12,-0.4) {$1$};
\node at (10,-0.4) {$\tau_1$};
\node at (14,0) {${\circ}$};
\node at (18,0) {${\bullet}$};
\node at (16,0) {${\circ}$};
\node at (12,0) {${\circ}$};
\node at (10,0) {${\bullet}$};
\end{tikzpicture}
\end{figure}
\begin{align*}
A_{12}&:=
\left\{
\begin{array}{ll}
t_1=l_{21}\\
t_2=\tau_2~(=const),
\end{array}
\right. \qquad
A_{13}:=
\left\{
\begin{array}{ll}
t_1=\tau_1~(=const) \\
t_2=l_{31},
\end{array}
\right. \\
A_{23}&:=\left\{
\begin{array}{ll}
t_1=l_{23}\\
t_2=\tau_2~(=const)
\end{array}
\right.=
\left\{
\begin{array}{ll}
t_1=\tau_1~(=const)\\
t_2=l_{32},
\end{array}
\right. \\
A_{24}&:=
\left\{
\begin{array}{ll}
t_1=l_{24}\\
t_2=\tau_2~(=const),
\end{array}
\right. \qquad
A_{34}:=\left\{
\begin{array}{ll}
t_1=\tau_1~(=const) \\
t_2=l_{34},
\end{array}
\right.\\
A_{25}&:=
\left\{
\begin{array}{ll}
t_1=l_{25}\\
t_2=\tau_2~(=const),
\end{array}
\right. \qquad
A_{35}:=
\left\{
\begin{array}{ll}
t_1=\tau_1~(=const) \\
t_2=l_{35},
\end{array}
\right.\\
A_{14}&:=A_{13}^{-1}A_{12}^{-1}A_{34}^{-1}A_{24}^{-1}A_{23}^{-1},\quad
A_{15}:=A_{14}^{-1}A_{13}^{-1}A_{12}^{-1},\quad
A_{45}:=A_{34}^{-1}A_{24}^{-1}A_{14}^{-1}.\\
A_{ii}&:=1~(1 \leq i \leq 5),\quad A_{ji}:=A_{ij}~(1 \leq i < j \leq 5).
\end{align*}
Then,
each $A_{ij} (i \neq j)$ is a meridian of a divisor on $M_{0,5}^\an$, as shown in TABLE \ref{table-M-meri}.

\renewcommand{\arraystretch}{1.8}
\tabcolsep = 0.3cm
\begin{table}[hbtp]
  \caption{Correspondence between each divisor on $M_{0,5}^\an$ and its meridian $A_{ij}$}
  \label{table-M-meri}
  \centering
  \begin{tabular}{|c||c|c|c|c|c|c|c|c|c|c|}
    \hline
    divisor & $t_1=0$ & $t_1=1$ & $t_1=\infty$ & $t_2=0$ & $t_2=1$ & $t_2=\infty$ & $t_1=t_2$   \\
    \hline 

    meridian  & $A_{24}$ & $A_{12}$  & $A_{25}$ & $A_{34}$ & $A_{13}$ & $A_{35}$ & $A_{23}$ \\ \hline
    
  \end{tabular}
\end{table}

\section{Fundamental groups of $V_{\rm B_3}^\an$ and $V_{\rm non\text{-}Fano}^\an$}
In this section,
we compute the topological fundamental groups of $V_{\rm B_3}^\an$ and $V_{\rm non\text{-}Fano}^\an$ by applying Galois theory of $M_{0,5}^\an$.
First,
by Galois theory for covering spaces,
the pointed finite Galois covering space
$
f_{\rm cov}^{\an}: \left({V}_{\rm B_3}^{\an},\vec{v}\right) \to \left(M_{0,5}^\an,\vec{\tau}\right)
$ in (\ref{fcov}) corresponds to
the normal subgroup of index $4$
\begin{align}
f_{\rm cov\ast}^{\an}\left( \pi_1^{\rm top}\left({V}_{\rm B_3}^{\an},\vec{v}\right) \right) \subset \pi_1^{\rm top}\left(M_{0,5}^\an,\vec{\tau}\right),
\end{align}
where
$f_{\rm cov\ast}^{\an}: \pi_1^{\rm top}\left({V}_{\rm B_3}^{\an},\vec{v}\right) \hookrightarrow \pi_1^{\rm top}\left(M_{0,5}^\an,\vec{\tau}\right)$ is the homomorphism induced by $f_{\rm cov}^{\an}$.
We set
\begin{align}
B_1,...,B_{10} \in \pi_1^{\rm top}\left(M_{0,5}^\an,\vec{\tau}\right)
\end{align}
as in TABLE \ref{table3}.
For each $i=1,...,10$,
we observe that the inverse image of $B_i$ under $f_{\rm cov}^{\an}$ is also a closed path.
Thus,
it holds that
$B_1,...,B_{10} \in f_{\rm cov\ast}^{\an}\left( \pi_1^{\rm top}\left({V}_{\rm B_3}^{\an},\vec{v}\right) \right)$.
Hereafter,
we identify the
closed paths of $\left(M_{0,5}^\an,\vec{\tau}\right)$ contained in $f_{\rm cov\ast}^{\an}\left( \pi_1^{\rm top}\left({V}_{\rm B_3}^{\an},\vec{v}\right) \right)$ with those of $\left({V}_{\rm B_3}^{\an},\vec{v}\right)$ by taking the inverse image under $f_{\rm cov}^{\an}$.
Then,
$
B_1,...,B_{10} \in \pi_1^{\rm top}\left({V}_{\rm B_3}^{\an},\vec{v}\right)
$
are meridians of divisors on ${V}_{\rm B_3}^{\an}$, as shown in TABLE \ref{table2}.

\renewcommand{\arraystretch}{1.8}
\tabcolsep = 0.2cm
\begin{table}[hbtp]
  \caption{$B_1, \ldots ,B_{10}$}
  \label{table3}
  \centering
  \begin{tabular}{|c||c|c|c|c|c|c|c|c|c|c|}
    \hline
    $i$ & $1$ & $2$ & $3$ & $4$ & $5$ & $6$ & $7$ & $8$ & $9$ & $10$   \\
    \hline 

    $B_i$  & $A_{12}$ & $A_{13}$  & $A_{23}$ & $A_{24}^2$ & $A_{34}^2$ & $A_{34}A_{23}A_{34}^{-1}$ & $A_{24}A_{12}A_{24}^{-1}$ & $A_{34}A_{13}A_{34}^{-1}$ & $A_{25}^2$ & $A_{35}^2$ \\ \hline 
\end{tabular}
\end{table}

\renewcommand{\arraystretch}{1.8}
\tabcolsep = 0.1cm
\begin{table}[hbtp]
  \caption{Correspondence between each divisor on ${V}_{\rm B_3}^{\an}$ and its meridian $B_{i}$}
  \label{table2}
  \centering
  \begin{tabular}{|c||c|c|c|c|c|c|c|c|c|c|c|c|}
    \hline
    divisor & $s_1=0$ & $s_2=0$ & $s_1=s_2$ & $s_1=1$ & $s_2=1$ & $s_1 s_2=1$ & $s_1=\infty$ & $s_2=\infty$ & $s_1=-1$ & $s_2=-1$   \\
    \hline 

    meridian & $B_1$ & $B_2$  & $B_3$ & $B_4$ & $B_5$ & $B_6$ & $B_7$ & $B_8$ & $B_{9}$ & $B_{10}$ \\ \hline
    
  \end{tabular}
\end{table}
Moreover,
we can verify that
the subgroup of $\pi_1^{\rm top}\left(M_{0,5}^\an,\vec{\tau} \right)$
generated by
$B_1,...,B_{10}$
is a normal subgroup of index $4$.
Therefore,
$\pi_1^{\rm top}\left({V}_{\rm B_3}^{\an},\vec{v}\right)$ is generated by meridians $B_1,...,B_{10}$.
Then
we obtain the following representation of $\pi_1^{\rm top}\left({V}_{\rm B_3}^{\an},\vec{v}\right)$ via the Reidemeister--Schreier rewriting process:
\begin{align}\label{fun_nf}
\pi_1^{\rm top}\left({V}_{\rm B_3}^{\an},\vec{v}\right)
&= \left\langle
\left. 
\begin{array}{l}
B_1, B_2, B_3, B_4, B_5, \\
B_6, B_7, B_8, B_9, B_{10}
\end{array} 
\right|
\begin{array}{l}
(R'1) \sim (R'12)
\end{array}
\right\rangle \\
& = \left\langle
\left. 
\begin{array}{l}
B_1, B_2, B_3, B_4, \\
B_5, B_6, B_7, B_8
\end{array} 
\right|
\begin{array}{l}
(R'1) \sim (R'10)
\end{array}
\right\rangle. \notag
\end{align}
The relations $(R'1) \sim (R'12)$ are described as follows:
\begin{align*}
(R'1)~&B_1 B_5 = B_5 B_1,~~B_4 B_8 = B_8 B_4,\\
(R'2)~&B_1 B_2 B_3 = B_2 B_3 B_1 = B_3 B_1 B_2,~~B_1 B_8 B_6 = B_8 B_6 B_1 = B_6 B_1 B_8,\\
(R'3)~& B_5 B_3 B_4 B_6 = B_6 B_5 B_3 B_4,~~B_3 B_4 B_1 B_2 = B_2 B_3 B_4 B_1,\\
(R'4)~&B_3 B_7 B_3^{-1} = B_5 B_3 B_7 B_3^{-1} B_5^{-1},\\
&B_6 B_5 B_6^{-1} = B_4^{-1} B_3^{-1} B_5 B_3 B_4,\\
&B_7 B_3 B_7^{-1} = B_3^{-1} B_8^{-1} B_3 B_8 B_3,\\
(R'5)~&B_4 B_6 B_5 B_2 B_3 = B_6 B_5 B_2 B_3 B_4,\\
(R'6)~&B_7 B_8 B_4 B_6 B_5 B_3 = B_8 B_3 B_7 B_4 B_6 B_5,~~B_6 B_5 B_2 B_3 B_7 B_4 = B_5 B_2 B_3 B_7 B_4 B_6,\\
(R'7)~&B_3 B_7 B_3^{-1} = B_6 B_5 B_2 B_3 B_7 B_3^{-1} B_2^{-1} B_5^{-1} B_6^{-1},\\
&B_6 B_2 B_6^{-1} = B_3 B_7^{-1} B_3^{-1} B_8^{-1} B_2 B_8 B_3 B_7 B_3^{-1},\\
&B_5 ^{-1} B_8 B_5 = B_7^{-1} B_3^{-1} B_2^{-1} B_5^{-1} B_8 B_5 B_2 B_3 B_7,\\
(R'8)~&B_5 B_8 B_5 B_2 B_3 B_7 B_4 B_6 B_1 = B_3 B_7 B_4 B_6 B_1 B_5 B_8 B_5 B_2,\\
(R'9)~&B_1^{-1} B_7 B_4 B_1 B_2^{-1} B_5^{-1} B_6^{-1} B_8^{-1} B_1^{-1} B_4^{-1} B_7^{-1} B_4 B_6 B_5 B_4^{-1} B_1 B_2 B_7 B_5^{-1} B_8 B_5 B_7^{-1} = 1,\\
&B_4 B_3^{-1} B_2^{-1} B_5^{-1} B_8^{-1} B_2 B_8 B_5 B_2 B_3 B_6^{-1} B_4^{-1} B_7^{-1} B_3^{-1} B_2^{-1} B_5^{-1} B_2^{-1} B_6 B_5 B_2 B_3 B_7 = 1,\\
(R'10)~&B_4 B_2^{-1} B_5^{-1} B_1^{-1} B_6^{-1} B_4^{-1} B_7^{-1} B_3^{-1} B_8 B_5 B_2 B_3 B_7 B_4 B_6 B_1 B_4^{-1} B_7^{-1} B_3^{-1} B_5^{-1} B_8^{-1} B_5 B_3 B_7 = 1,\\
(R'11)~&B_1 B_3 B_7 B_4 B_6 B_9 = 1,\\
(R'12)~&B_2 B_3 B_8 B_6 B_5 B_{10} = 1.
\end{align*}

The kernel of $\pi_1^{\rm top}\left({V}_{\rm B_3}^{\an},\vec{v}\right) \twoheadrightarrow \pi_1^{\rm top}\left(V_{\rm non\text{-}Fano}^{\an},\vec{v}\right)$ induced by
the inclusion ${V}_{\rm B_3}^{\an} \hookrightarrow V_{\rm non\text{-}Fano}^{\an}$
is the free group of rank $2$ generated by
$B_9, B_{10}$,
which are meridians of
$s_1=-1$,
$s_2=-1$,
respectively.
\begin{align}
1 \to \la B_9, B_{10} \ra \to \pi_1^{\rm top}\left({V}_{\rm B_3}^{\an},\vec{v}\right) \to \pi_1^{\rm top}\left(V_{\rm non\text{-}Fano}^{\an},\vec{v}\right) \to 1.
\end{align}
For $i=1,2,\cdots,10$, we set $(R''i):=(R'i)$.
The relations $(R''11),~(R''12)$ are described as follows:
\begin{align*}
&(R''11)~B_1 B_3 B_7 B_4 B_6 = 1,\\
&(R''12)~B_2 B_3 B_8 B_6 B_5 = 1.
\end{align*}
Hence,
$\pi_1^{\rm top}\left(V_{\rm non\text{-}Fano}^{\an},\vec{v}\right)$ is represented as follows:
\begin{align} \label{fun-nf}
\pi_1^{\rm top}\left(V_{\rm non\text{-}Fano}^{\an},\vec{v}\right)
&= \left\langle
\left. 
\begin{array}{l}
B_1, B_2, B_3, B_4, \\
B_5, B_6, B_7, B_8
\end{array} 
\right|
\begin{array}{l}
(R''1) \sim (R''12)
\end{array}
\right\rangle \\
&= \left\langle
\left. 
\begin{array}{l}
B_1, B_2, B_3, \\
B_4, B_5, B_6
\end{array} 
\right|
\begin{array}{l}
(R''1) \sim (R''10)
\end{array}
\right\rangle. \notag
\end{align}
To obtain a more detailed structure of $\pi_1^{\rm top}\left(V_{\rm non\text{-}Fano}^{\an},\vec{v}\right)$,
we consider the following diagram:
\begin{center}
\begin{tikzpicture}[auto]
\node (cic) at (7.5, -1) {$\circlearrowleft$};
\node (cic) at (7.5, 1) {$\circlearrowleft$};

\node (nf) at (5, 2) {$V_{\rm non\text{-}Fano}^{\an}$};
\node (P33) at (10, 2) {${\bP}^{1}(\bC) \backslash \{0,1,\infty\}$};

\node (B3) at (5, 0) {${V}_{\rm B_3}^{\an}$};
\node (P4) at (10, 0) {${\bP}^{1}(\bC) \backslash \{0,1,-1,\infty\}$};

\node  (Conf) at (5, -2) {$M_{0,5}^\an$};
\node  (P3) at (10, -2) {${\bP}^{1}(\bC) \backslash \{0,1,\infty\}$};

\node  (t2) at (12.5, 2) {$t$};
\node  (ni3) at (11.8, 2) {$\ni$};

\node  (t) at (12.5, 0) {$t$};
\node  (tt) at (12.5, -2) {$\left(\dfrac{1-t}{1+t}\right)^2.$};

\node  (ni1) at (12, 0) {$\ni$};
\node  (ni2) at (11.55, -2) {$\ni$};

\node (pr3) at (7.2, 2.3) {$f_8^{\an}\left(={\rm pr}_{1}\right)$};
\node (pr2) at (6.8, 0.3) {${\rm pr}_{1}$};
\node (pr22) at (7.2, -2.3) {${\rm pr}_1$};
\node (fcov) at (4.5, -1) {$f_{\rm cov}^\an$};


\draw[|->] (t) to node {} node[swap] {} (tt);
\draw[|->] (t) to node {} node[swap] {} (t2);

\draw[right hook->] (B3) to node {} node[swap] {} (nf);
\draw[right hook->] (P4) to node {} node[swap] {} (P33);
\draw[->>] (nf) to node {} node[swap] {} (P33);

\draw[->>] (B3) to node {} node[swap] {} (P4);
\draw[->>] (Conf) to node {} node[swap] {} (P3);

\draw[->>] (B3) to node {} node[swap] {} (Conf);
\draw[->>] (P4) to node {} node[swap] {} (P3);
\end{tikzpicture}
\end{center}
Note that ${\rm pr}_1$ is the projection to the first component.
The image of $\vec{v}$ under the projection ${\rm pr}_1$ (in the middle column) is Galois equivalent to $\overrightarrow{01}$.
Therefore, we obtain the following diagram by taking $\pi_1^{\rm top}$:
\begin{center}
\begin{tikzpicture}[auto]

\node (s) at (7.3, 1.3) {$s$};

\node (f4) at (1, 2) {$\left<B_2,B_3,B_5,B_6\right>$};
\node (nf) at (5, 2) {$\pi_1^{\rm top}\left(V_{\rm non\text{-}Fano}^{\an},\vec{v}\right)$};
\node (P33) at (10.2, 2) {$\pi_1^{\rm top}({\bP}^{1}(\bC) \backslash \{0,1,\infty\},\overrightarrow{01})$};

\node (f5) at (1, 0) {$\left<B_2,B_3,B_5,B_6,B_8\right>$};
\node (B3) at (5, 0) {$\pi_1^{\rm top}\left({V}_{\rm B_3}^{\an},\vec{v}\right)$};
\node (P4) at (10.2, 0) {$\pi_1^{\rm top}({\bP}^{1}(\bC) \backslash \{0,1,-1,\infty\},\overrightarrow{01})$};

\node  (f3) at (1, -2) {$\left<A_{13},A_{23},A_{34}\right>$};
\node  (Conf) at (5, -2) {$P_4/\left\langle \omega_4 \right\rangle$};
\node  (P3) at (10.2, -2) {$\pi_1^{\rm top}({\bP}^{1}(\bC) \backslash \{0,1,\infty\},\tau_1)$};

\node (pr222) at (7.3, 2.3) {$f_{8\ast}^{\an}$};
\node (pr2) at (6.9, 0.3) {${\rm pr}_{1\ast}$};
\node (pr22) at (6.9, -1.7) {${\rm pr}_{1\ast}$};
\node (fcov) at (4.5, -1) {$f_{\rm cov\ast}^\an$};

\node (1a) at (-2, 0) {$1$};
\node (1b) at (14, 0) {$1$};
\node (1c) at (-2, -2) {$1$};
\node (1d) at (14, -2) {$1$};
\node (1e) at (-2, 2) {$1$};
\node (1f) at (14, 2) {$1$};

\draw[->] (1a) to node {} node[swap] {} (f5);
\draw[->] (1c) to node {} node[swap] {} (f3);
\draw[->] (P4) to node {} node[swap] {} (1b);
\draw[->] (P3) to node {} node[swap] {} (1d);
\draw[->] (P33) to node {} node[swap] {} (1f);
\draw[->] (1e) to node {} node[swap] {} (f4);

\draw[right hook->] (f4) to node {} node[swap] {} (nf);
\draw[right hook->] (f5) to node {} node[swap] {} (B3);
\draw[right hook->] (f3) to node {} node[swap] {} (Conf);

\draw[->>] (P4) to node {} node[swap] {} (P33);
\draw[->>] (B3) to node {} node[swap] {} (nf);
\draw[->>] (f5) to node {} node[swap] {} (f4);

\draw[->>] (nf) to node {} node[swap] {} (P33);
\draw[->>] (B3) to node {} node[swap] {} (P4);
\draw[->>] (Conf) to node {} node[swap] {} (P3);

\draw[right hook->] (f5) to node {} node[swap] {} (f3);
\draw[right hook->] (B3) to node {} node[swap] {} (Conf);
\draw[right hook->] (P4) to node {} node[swap] {} (P3);

\draw[->] (8,1.8) to [out=210,in=-30] (6.5,1.8);
\end{tikzpicture}
\end{center}
We take a group-theoretic section of $f_{8\ast}$
\begin{align}
s: \pi_1^{\rm top}({\bP}^{1}(\bC) \backslash \{0,1,\infty\},\overrightarrow{01}) \to \pi_1^{\rm top}\left(V_{\rm non\text{-}Fano}^{\an},\vec{v}\right)
\end{align}
defined by $l_0 \mapsto B_1$ and $l_1 \mapsto B_4$.
Then, we have 
\begin{align}\label{semiprod} 
\pi_1^{\rm top}\left(V_{\rm non\text{-}Fano}^{\an},\vec{v}\right) &= \left<B_1,B_4\right> \ltimes \left<B_2,B_3,B_5,B_6\right>\\
& = F_2 \ltimes F_4. \notag
\end{align}
One can see that
the conjugation action $F_2^{} \to \Aut\left(F_4^{\rm ab}\right)$ is trivial by using $(R''1) \sim (R''12)$.
Hence, we obtain
\begin{align}\label{prodfun} 
\pi_1^{\rm top}\left(V_{\rm non\text{-}Fano}^{\an},\vec{v}\right)^{\rm ab} &= \left<\bar{B}_1,\bar{B}_4\right> \times \left<\bar{B}_2,\bar{B}_3,\bar{B}_5,\bar{B}_6\right>\\
& = F_2^{\ab} \times F_4^{\ab}, \notag
\end{align}
where 
$\bar{g}$ denotes the image of an element $g \in \pi$ in the abelianization $\pi^{\ab}$ for a group $\pi$.

\begin{rem}
In \cite{Su01},
the topological fundamental groups of $V_{\rm non\text{-}Fano}^{\an}$ and ${V}_{\rm B_3}^{\an}$ are computed by the braid monodromy of the arrangement.
Instead,
in this paper,
we compute them by applying Galois theory of $M_{0,5}$.
The form (\ref{fun-nf}) of the topological fundamental group obtained in this paper has the advantage that it is easier to compute the image of $B_j$ under $f_{i\ast}^{\rm an}$~(see TABLE \ref{table-main}).
\end{rem}

\section{Review of Polylogarithms}
In this section,
we recall the definition and some properties of complex/$\ell$-adic polylogarithms and their Lie versions.

\subsection{Complex polylogarithms}\label{comp-poly}

In this subsection,
we recall the basic properties of complex polylogarithms \cite{NW12}, \cite{F04}, \cite{F14}, \cite{D90}.
The complex logarithm $\log(z;\gamma)$ depending on $\gamma \in \pi_1^{\rm top}\left(\mathbb{P}^1(\mathbb{C}) \backslash \{0,1,\infty\}; \overrightarrow{01}, z\right)$ is defined as follows:
\begin{align}\label{log}
\log(z;\gamma):=\int_{\delta_{\overrightarrow{10}}^{-1} \cdot \gamma}\frac{dt}{t},
\end{align}
where $\delta_{\overrightarrow{10}} \in \pi_1^{\rm top}\left(\mathbb{P}^1(\mathbb{C}) \backslash \{0,1,\infty\}; \overrightarrow{01}, \overrightarrow{10}\right)$ is the straight path along the unit interval $(0,1) \subset \mathbb{P}^1(\mathbb{R}) \backslash \{0,1,\infty\}$ as shown in FIGURE~\ref{fig:path}.
For each $\gamma \in \pi_1^{\rm top}\left(\mathbb{P}^1(\mathbb{C}) \backslash \{0,1,\infty\}; \overrightarrow{01}, z\right)$,
we set 
\begin{align} \label{gamma'}
\gamma':=\delta_{\overrightarrow{10}} \cdot \phi_{\overrightarrow{10}}(\gamma) \in \pi_1^{\rm top}\left(\mathbb{P}^1(\mathbb{C}) \backslash \{0,1,\infty\}; \overrightarrow{01}, 1-z\right),
\end{align}
where $\phi_{\overrightarrow{10}} \in \Aut\left( \mathbb{P}^1(\mathbb{C}) \backslash \{0,1,\infty\} \right)$ is defined by $\phi_{\overrightarrow{10}}(t)=1-t$.
For $k \in \mathbb{N}$,
we define the complex polylogarithm $Li_{{k}}(z;\gamma)$ associated with $\gamma
\in \pi_1^{\rm top}\left(\mathbb{P}^1(\mathbb{C}) \backslash \{0,1,\infty\}; \overrightarrow{01}, z\right)$ as the iterated integral
\begin{align} \label{cpoly1}
& {Li}_{{k}}(z;\gamma):=
\int_{\gamma} \frac{1}{t}Li_{k-1}(t;\gamma_t)dt \quad (k \geq 2),
 \\ \label{cpoly2}
& Li_1(z;\gamma):=-\log(1-z;\gamma')=\int_{\gamma}\frac{dt}{1-t}.
\end{align}
Therefore, the complex polylogarithm and logarithm may be regarded as maps
\[
{Li}_{{k}}(z),~\log(z): \pi_1^{\rm top}\left(\mathbb{P}^1(\mathbb{C}) \backslash \{0,1,\infty\}; \overrightarrow{01}, z\right) \to \bC,
\]
sending $\gamma \mapsto {Li}_{{k}}(z;\gamma),~\log(z;\gamma)$, respectively, where $z$ is a $\bC$-rational (possibly tangential) base point of $\mathbb{P}^1 \backslash \{0,1,\infty\}$.
Note that $\log(z)$ factors through $\pi_1^{\rm top}\left(\mathbb{P}^1(\mathbb{C}) \backslash \{0,\infty\}; \overrightarrow{01}, z\right)$.
In particular,
we define 
the zeta value by
\begin{align} \label{zeta}
\zeta({k}):={Li}_{{k}}\left(\overrightarrow{10};\delta_{\overrightarrow{10}}\right) \in \mathbb{R}
\quad (k \geq 2).
\end{align}

The complex polylogarithm ${Li}_k(z)$ admits a generating series closely related to the KZ equation, as described below.
Let
\begin{align}\label{Cvariable}
X:=\left(\dfrac{dz}{z}\right)^{\ast},~
Y:=\left(\dfrac{dz}{z-1}\right)^{\ast}~
\in \Omega^1_{\rm log}\left( \mathbb{P}^1(\mathbb{C}) \backslash \{0,1,\infty\} \right)^{\ast}
\end{align}
be the duals of the canonical differential forms on $\mathbb{P}^1(\mathbb{C}) \backslash \{0,1,\infty\}$.
Here,
\[
\Omega^1_{\rm log}\left( \mathbb{P}^1(\mathbb{C}) \backslash \{0,1,\infty\} \right)\]
is the space of meromorphic 1-forms on $\mathbb{P}^1(\mathbb{C})$ with logarithmic singularities along $\{0,1,\infty\}$.
The natural isomorphism $\Omega^1_{\rm log}\left( \mathbb{P}^1(\mathbb{C}) \backslash \{0,1,\infty\} \right)^{\ast} \simeq \pi_1^{\rm top}\left(\mathbb{P}^1(\mathbb{C}) \backslash \{0,1,\infty\}, \overrightarrow{01}\right)^{\rm ab} \otimes \mathbb{C}$
gives the identifications
$
X=\dfrac{\bar{l}_0}{2\pi \sqrt{-1}}$,
$Y=\dfrac{\bar{l}_1}{2\pi \sqrt{-1}}$,
and
let
$$Z:=-Y-X\left(=\dfrac{\bar{l}_\infty}{2\pi \sqrt{-1}}\right)$$
where $l_0$, $l_1$ and $l_\infty$ are as described in (\ref{tripod}).
Given a path $\gamma \in \pi_1^{\rm top}\left(\mathbb{P}^1(\mathbb{C}) \backslash \{0,1,\infty\}; \overrightarrow{01}, z\right)$ and the $1$-form
\[
\omega:=\frac{dz}{z}X+\frac{dz}{z-1}Y \in \Omega^1_{\rm log}\left( \mathbb{P}^1(\mathbb{C}) \backslash \{0,1,\infty\} \right) \otimes \mathbb{C} \langle \langle X, Y \rangle \rangle, 
\]
we can form a formal power series
\begin{align}\label{Lambda}
\Lambda_{\overrightarrow{01}}\left(z; \gamma \right):=1+\sum_{i=1}^{\infty} \int_{\gamma} {\underbrace{\omega \ldots \omega}_{i~\text{times}}} \in \bC \langle \langle X,
Y \rangle \rangle^{\times}
\end{align}
where  $\mathbb{C} \langle \langle X, Y \rangle \rangle$ is the noncommutative formal power series ring in two variables over $\mathbb{C}$.
It is easily verified that
the expansion of (\ref{Lambda}) has the form
\begin{align}\label{G0}
\Lambda_{\overrightarrow{01}}\left(z; \gamma \right)=&1+\sum_{i=1}^{\infty}\frac{\log^i(z;\gamma)}{i!}X^i+\sum_{i=1}^{\infty}\frac{\log^i(1-z;\gamma')}{i!}Y^i\\
&-\sum_{i=1}^{\infty} Li_{i+1}(z;\gamma) YX^{i}+\cdots. \notag
\end{align}
We note the relationship between $\Lambda_{\overrightarrow{01}}\left(z; \gamma \right)$ and the fundamental solution of the KZ equation (\ref{KZeq}).
The formal KZ equation on $\mathbb{P}^1(\mathbb{C}) \backslash \{0,1,\infty\}$ is the differential equation
\begin{align}\label{KZeq}
\dfrac{d}{d z}G(z)=\left( \dfrac{X}{z}+\dfrac{Y}{z-1} \right)G(z)
\end{align}
where
$G(z)$ is an analytic function (i.e., each of its coefficients is analytic) with values in $\mathbb{C} \langle \langle X, Y \rangle \rangle$.
Let $G_{\overrightarrow{01}}\left(z; \gamma \right)$ be the fundamental solution of (\ref{KZeq})
characterized by the asymptotic behavior
\[
G_{\overrightarrow{01}}\left(z; \gamma \right) \approx \sum_{i=0}^{\infty}\frac{\log^{i}(z;\gamma)}{i!}X^i
\]
as $z \to \overrightarrow{01}$ along $\gamma$.
For a word~$w=w_1 \cdots w_n~(w_1,\cdots,w_n \in \{X,Y\})$,
we define $\overline{w}:=w_n \cdots w_1$.
For $\Lambda=\sum_{w} \Coeff_{w}(\Lambda) \cdot w  \in \mathbb{C} \langle \langle X,Y \rangle \rangle$,
we define the reversal of $\Lambda$ by $\overline{\Lambda}:=\sum_{w} \Coeff_{w}(\Lambda) \cdot \overline{w}$.
Then,
the following equation holds:
\begin{align}\label{G0'}
G_{\overrightarrow{01}}\left(z; \gamma \right)=\overline{\Lambda_{\overrightarrow{01}}\left(z; \gamma \right)},
\end{align}
that is,
\begin{align}
Li_{k}(z;\gamma)=-\Coeff_{X^{k-1}Y}\left( G_{\overrightarrow{01}}\left(z; \gamma \right) \right).
\end{align}

Moreover,
the Lie version of the complex polylogarithm is defined as follows.
We denote by
${\rm Lie}_{\bC}\langle \langle X,
Y \rangle \rangle$
the complete free Lie algebra consisting of Lie-like elements in $\bC \langle \langle X,
Y \rangle \rangle$.
Since $\Lambda_{\overrightarrow{01}}\left(z; \gamma \right) \in \bC \langle \langle X,
Y \rangle \rangle$ is group-like,
we can take its inverse and 
obtain a Lie formal power series
\[
{\bf log}\left(\Lambda_{\overrightarrow{01}}\left(z;\gamma \right)^{-1}\right) \in {\rm Lie}_{\bC}\langle \langle X,
Y \rangle \rangle
\]
where
\[
{\bf log}: \Lambda^{-1} \mapsto \sum_{n=1}^{\infty}\frac{(-1)^{n-1}\left(\Lambda^{-1}-1\right)^n}{n}.
\]
We write 
\[
\varphi_{k}: {\rm Lie}_{\bC}\langle \langle X,
Y \rangle \rangle \to \mathbb{C}
\]
for the $\mathbb{C}$-linear form that 
selects 
the coefficient of $e_k$ with respect to the Hall basis
\[
e_1:=Y, \quad
e_m:=\left[X, e_{m-1}\right]={\rm ad}(X)^{m-1}(Y).
\]
We define
\begin{align} \label{clidef}
&{\rm li}_{k}(z;\gamma):=\frac{1}{(2 \pi \sqrt{-1})^k} \cdot \varphi_{k} \left({\bf log}\left(\Lambda_{\overrightarrow{01}}\left(z; \gamma \right)^{-1}\right)\right) \quad (k \geq 1),\\
&\mathrm{li}_0\left(z;\gamma\right):=-\frac{1}{2 \pi \sqrt{-1}} \log\left(z;\gamma\right), \notag
\end{align}
which is called the Lie version of the complex polylogarithm.
Then, Nakamura and Wojtkowiak \cite[Proposition 5.2]{NW12} established the following formula:
\begin{align} \label{c-explicit}
{\rm li}_{n}(z;\gamma)=\dfrac{(-1)^{n+1}}{(2 \pi \sqrt{-1})^n}\sum_{k=0}^{n-1}\dfrac{B_k}{k!}\log^{k}(z;\gamma){Li}_{n-k}(z;\gamma),
\end{align}
where
$\{B_n\}_{n}$ is the sequence of Bernoulli numbers defined by
$\sum_{n=0}^{\infty}\frac{B_n}{n!}T^n=\frac{T}{e^T-1}$.
Note that $B_1=-\frac{1}{2}$.

\subsection{$\ell$-adic Galois polylogarithms}\label{l-poly}

In this subsection,
we recall the basic properties of $\ell$-adic Galois polylogarithms \cite{NW12}, \cite{NW20a}, \cite{NW99}.

For an algebraic variety $V$ over ${\overline{K}}$ with its $K$-rational (possibly tangential base) points
$\ast$ and
$\ast'$,
the absolute Galois group $G_K=\Gal(\overline{K}/K)$ acts on the pro-$\ell$ set $\pi_1^\ellet(V;\ast,\ast')$ (cf. \cite[2.8]{N99}, \cite[(1.1)]{NW99}).
For $\sigma \in G_K$ and $p \in \pi_1^{\rm top}(V^{\an}; \ast, \ast')$,
regarding $p$ as $p \in \pi_1^\ellet(V;\ast,\ast')$ via the comparison map (\ref{comp}),
we define a pro-$\ell$ \'etale loop
\begin{align} \label{f}
{\mathfrak f}^{\ast',p}_{\ast,\sigma}:=
p \cdot \sigma(p)^{-1} \in \pi_1^\ellet(V,\ast).
\end{align}
If $p=\delta_{\overrightarrow{10}}$, 
then this is called the $\ell$-adic Ihara associator in \cite[Definition 2.32]{F07}.

In the following,
we mainly consider the case in which $V=\mathbb{P}^1_{\overline{K}} \backslash \{0,1,\infty\}$ and $\ast=\overrightarrow{01}$.
In this case,
the pro-$\ell$ fundamental group $\pi_1^\ellet\left(\mathbb{P}^1_{\overline{K}} \backslash \{0,1,\infty\},\overrightarrow{01}\right)$
is the free pro-$\ell$ group of rank 2 with a system of topological generators $\vec{l}:=(l_0, l_1)$ as described in (\ref{tripod}).
Let $z$ be a $K$-rational (possibly tangential) base point on $\mathbb{P}^1_{\overline{K}} \backslash \{0,1,\infty\}$.
In contrast to the complex case (\ref{Cvariable}),
we set
\begin{align}\label{lG-variable}
X:={\bf log}(l_0),~Y:={\bf log}(l_1),~Z:={\bf log}(l_\infty)
\end{align}
in the complete group ring of $\pi_1^\ellet\left(\mathbb{P}^1_{\overline{K}} \backslash \{0,1,\infty\},\overrightarrow{01}\right)$ over $\mathbb{Q}_{\ell}$.
For $\gamma \in \pi_1^{\rm top}\left(\mathbb{P}^1(\mathbb{C}) \backslash \{0,1,\infty\}; \overrightarrow{01}, z\right)$ and $\sigma \in G_K$,
we obtain a formal power series
\begin{align} \label{f2}
{\mathfrak f}^{z,\gamma}_{\overrightarrow{01},\sigma} \in \mathbb{Q}_{\ell} \langle \langle X,
Y \rangle \rangle^{\times}
\end{align}
by taking the image of ${\mathfrak f}^{z,\gamma}_{\overrightarrow{01},\sigma} \in \pi_1^\ellet\left(\mathbb{P}^1_{\overline{K}} \backslash \{0,1,\infty\},\overrightarrow{01}\right)$ under the multiplicative $\ell$-adic Magnus embedding into the multiplicative group of the noncommutative formal power series ring
\begin{align}\label{ladicmagnus}
\pi_1^\ellet\left(\mathbb{P}^1_{\overline{K}} \backslash \{0,1,\infty\},\overrightarrow{01}\right) \hookrightarrow \mathbb{Q}_{\ell} \langle \langle X,
Y \rangle \rangle^\times
\end{align}
defined by
$l_0 \mapsto {\bf exp}(X):=\sum_{n=0}^{\infty}\frac{1}{n!}{X}^n$ and
$l_1 \mapsto {\bf exp}(Y)$.
The power series ${\mathfrak f}^{z,\gamma}_{\overrightarrow{01},\sigma}$ is an $\ell$-adic Galois analogue of the KZ fundamental solution 
$G_{\overrightarrow{01}}\left(z; \gamma \right)$ in (\ref{G0'}).
For $k \in \mathbb{N}$,
we define the $\ell$-adic Galois polylogarithm and the $\ell$-adic Galois zeta value by
\begin{align} \label{lpoly1}
 {Li}^\ell_{{k}}(z;\gamma,\sigma)
&:=-\Coeff_{X^{k-1}Y}\left( {\mathfrak f}^{z,\gamma}_{\overrightarrow{01},\sigma} \right), \\
\label{lpoly2}
{\boldsymbol \zeta}_{{k}}^\ell(\sigma)&:={Li}^\ell_{{k}}\left(\overrightarrow{10};\delta_{\overrightarrow{10}},\sigma\right).
\end{align}
Note that ${\boldsymbol \zeta}_{{k}}^\ell(\sigma)$ is called the $\ell$-adic Soul\'e element in \cite[Definition 2.32]{F07}
and is described by the Soul\'e character (cf. \cite[Examples 2.33]{F07}).

Therefore, our $\ell$-adic Galois polylogarithm may be regarded as the map
\[
Li_k^\ell(z): \pi_1^{\rm top}\left(\mathbb{P}^1(\mathbb{C}) \backslash \{0,1,\infty\}; \overrightarrow{01}, z\right) \times G_K \to \bQ_\ell,
\]
which sends $(\gamma,\sigma) \mapsto Li_k^\ell(z;\gamma,\sigma)$, where $z$ is a $K$-rational (possibly tangential) base point of $\mathbb{P}^1 \backslash \{0,1,\infty\}$.

Here, we recall the character version of the 
$\ell$-adic Galois polylogarithm
which is closely related to the Soul\'e character (cf. \cite[REMARK 2]{NW99}).
For each $n \in \mathbb{N}$,
we denote by
$\zeta_n:={\rm exp}(\frac{2 \pi \sqrt{-1}}{n}) \in \overline{K}$
a primitive $n$-th root of unity
and choose 
a system of
$n$-th roots
$z^{1/n}$
($n \in \mathbb{N}$) determined
by $\gamma \in \pi_1^{\rm top}\left(\mathbb{P}^1(\mathbb{C}) \backslash \{0,1,\infty\}; \overrightarrow{01}, z\right)$~(cf. the paragraph above DEFINITION 2 in \cite{NW99} p. 286).
For each $k \in \mathbb{Z}_{\geq 1}$,
the $\ell$-adic Galois polylogarithmic character (or the generalized $\ell$-adic Soul\'e character) associated with $\gamma \in \pi_1^{\rm top}\left(\mathbb{P}^1(\mathbb{C}) \backslash \{0,1,\infty\}; \overrightarrow{01}, z\right)$
\begin{align}\label{gsoule}
\tilde{\chi}_{k}
^{z,\gamma}: G_K \to {\mathbb Z}_{\ell}
\end{align}
is defined by the following Kummer property:
\[\zeta_{\ell^n}^{\tilde{\chi}_{k}
^{z,\gamma}(\sigma)~{\rm mod}~\ell^n}=
\sigma \left(\prod_{i=0}^{\ell^n -1}
(1-\zeta_{\ell^n}^{\chi(\sigma)
^{-1}i}z^{1/\ell^n})
^{\frac{i^{k-1}}{\ell^n}}\right) 
\bigg/ \prod_{i=0}^{\ell^n -1}
(1-\zeta_{\ell^n}
^{i+\rho_{z,\gamma}(\sigma)}z^{1/\ell^n})
^{\frac{i^{k-1}}{\ell^n}}\]
where
$\chi: G_K \to {\mathbb Z}_{\ell}^\times$ is the $\ell$-adic cyclotomic character
and 
\begin{align}\label{kummer1}
\rho_{z,\gamma}: G_K \to {\mathbb Z}_{\ell}
\end{align}
is the Kummer 1-cocycle
defined by
$
\sigma(z^{1/{\ell^n}})=z^{1/{\ell^n}} \cdot \zeta_{\ell^n}^{\rho_{z,\gamma}(\sigma)~{\rm mod}~\ell^n}.
$
Then,
the $\ell$-adic Magnus expansion of ${\mathfrak f}^{z,\gamma}_{\overrightarrow{01},\sigma} \in \mathbb{Q}_{\ell} \langle \langle X,
Y \rangle \rangle^{\times}$ has the form
\begin{align}\label{expli}
{\mathfrak f}^{z,\gamma}_{\overrightarrow{01},\sigma}
=1
&+\sum_{i=1}^{\infty}\dfrac{(-\rho_{z,\gamma}(\sigma))^i}{i!}X^{i}
+\sum_{i=1}^{\infty}\dfrac{(-\rho_{1-z,\gamma'}(\sigma))^i}{i!}Y^{i}
-\sum_{i=1}^\infty \frac{{\tilde{\chi}_{i+1}}
^{z,\gamma}(\sigma)}{i!} Y{X}^i\\
&-\sum_{i=1}^{\infty}{Li}^{\ell}_{i+1}(z;\gamma,\sigma) \cdot X^{i}Y+ \cdots, \notag
\end{align}
where $\gamma' \in \pi_1^{\rm top}\left(\mathbb{P}^1(\mathbb{C}) \backslash \{0,1,\infty\}; \overrightarrow{01}, 1-z\right)$ is as described in (\ref{gamma'}).

Moreover,
the Lie version of the $\ell$-adic Galois polylogarithm is defined as follows.
We denote by
${\rm Lie}_{\bQ_{\ell}}\langle \langle X,
Y \rangle \rangle$
the complete free Lie algebra consisting of Lie-like elements in $\Q_{\ell} \langle \langle X,
Y \rangle \rangle$.
Since ${\mathfrak f}^{z,\gamma}_{\overrightarrow{01},\sigma} \in \bQ_{\ell} \langle \langle X,
Y \rangle \rangle^{\times}$ is group-like,
we can take its inverse and 
obtain a Lie formal power series
\[
{\bf log}\left(\left({\mathfrak f}^{z,\gamma}_{\overrightarrow{01},\sigma}\right)^{-1}\right) \in {\rm Lie}_{\bQ_{\ell}}\langle \langle X,
Y \rangle \rangle.
\]
We write 
\[
\varphi_{k,\vec{l}}: {\rm Lie}_{\bQ_{\ell}}\langle \langle X,
Y \rangle \rangle \to \mathbb{Q}_{\ell}
\]
for the $\Q_{\ell}$-linear form that selects the coefficient of $e_k$ with respect to the Hall basis
\[
e_1:=Y, \quad
e_m:=\left[X,e_{m-1}\right]={\rm ad}(X)^{m-1}(Y)
\]
of ${\rm Lie}_{\bQ_{\ell}}\langle \langle X,
Y \rangle \rangle$.
We define
\begin{align} \label{llidef}
&{\ell i}_{k}(z;\gamma,\sigma):=\varphi_{k,\vec{l}}\left({\bf log}\left(\left({\mathfrak f}^{z,\gamma}_{\overrightarrow{01},\sigma}\right)^{-1}\right)\right) \quad (k \geq 1),\\
&{\ell i}_{0}(z;\gamma,\sigma):=\rho_{z,\gamma}(\sigma), \notag
\end{align}
which is called the Lie version of the $\ell$-adic Galois polylogarithm.
By an argument analogous to that of \cite[Proposition 5.2]{NW12},
we obtain the following conversion formula:
\begin{align}\label{explicit2}
{\ell i}_{n}(z;\gamma,\sigma)=\sum_{k=0}^{n-1}\dfrac{B_k}{k!}\left(-\rho_{z,\gamma}(\sigma)\right)^{k}{Li}^\ell_{{n-k}}(z;\gamma,\sigma).
\end{align}
Moreover,
the following explicit formula holds, as given in \cite{NS25}:
\begin{align}\label{explicit3}
{Li}^\ell_{n}(z;\gamma,\sigma)=(-1)^{n-1} \sum_{k=0}^{n-1} \frac{\rho_{z,\gamma}(\sigma)^k}{k!}
\frac{\tilde{\chi}_{n-k}^{z,\gamma}(\sigma)}{(n-1-k)!}.
\end{align}

The $\ell$-adic Galois polylogarithm is analogous to the complex polylogarithm, as in TABLE \ref{table_analog}.

\begin{rem}\label{rem:history-and-perspectives}
The $\ell$-adic Galois polylogarithm was originally introduced by Wojtkowiak \cite{W0}, and its fundamental properties were extensively studied in a series of papers \cite{W1}-\cite{W5}.
In \cite{NW99}, an explicit formula was established to express the $\ell$-adic Galois polylogarithms in terms of the generalized $\ell$-adic Soul\'e characters.
Subsequently, various functional equations for the $\ell$-adic Galois polylogarithms were derived in \cite{NW12}, \cite{NW20a} and \cite{NW20b}.
More recently, functional equations for $\ell$-adic Galois multiple polylogarithms, which serve as higher-depth generalizations of $\ell$-adic Galois polylogarithms, have been investigated by \cite{NS25}, \cite{Sh24} and \cite{Sh26}.
In \cite{Sa21}, Sakugawa demonstrated that the generalized $\ell$-adic Soul\'e characters coincide with the $\ell$-adic \'etale realization of the generalized Beilinson elements in K-groups.
Furthermore, a systematic study of various realizations of multiple polylogarithms, including the $\ell$-adic \'etale realization, was developed by Furusho \cite{F04} from the perspective of a Tannakian formalism.
Connections between the $\ell$-adic Galois polylogarithm and Coleman's $p$-adic polylogarithm (the $p$-adic crystalline counterpart) were studied in \cite{NSW17a} and \cite{NSW17b}, while the relationship between the $\ell$-adic Galois polylogarithm and the finite polylogarithm (in the sense of Kaneko--Zagier) was investigated in \cite{SS17}.
\end{rem}

\vspace*{\stretch{1}}
\renewcommand{\arraystretch}{2.9}
\tabcolsep = 0.3cm
\begin{table}[htb]\label{table_analog}
\centering
  \caption{Analogy between complex polylogarithms and $\ell$-adic Galois polylogarithms}
\label{table_analog}
  \begin{tabular}{|c||c|}  \hline
    $\ell$-adic Galois side & complex side  \\ \hline \hline

$z$ : $K$-rational base point on $\mathbb{P}^1 \backslash \{0,1,\infty\}$ & $z$ : $\mathbb{C}$-rational base point on $\mathbb{P}^1 \backslash \{0,1,\infty\}$ \\ \hline

$X={\bf log}(l_0),~Y={\bf log}(l_1)$ & $X=\left(\dfrac{dz}{z}\right)^{\ast},~Y=\left(\dfrac{dz}{z-1}\right)^{\ast}$ \\ 


\hline


$Z={\bf log}({\bf exp}(-Y){\bf exp}(-X))$ 
& $Z=-Y-X$ \\ 

\hline

   The action of $G_K$ on $\pi_1^\ellet\left(\mathbb{P}^1_{\overline{K}} \backslash \{0,1,\infty\}\right)$   & The formal KZ equation on $\mathbb{P}^1(\mathbb{C}) \backslash \{0,1,\infty\}$ \\ \hline 

    ${\mathfrak f}^{z,\gamma}_{\overrightarrow{01},\sigma} \in \mathbb{Q}_{\ell} \langle \langle X,
Y \rangle \rangle^{\times},$ & $G_{\overrightarrow{01}}\left(z; \gamma \right)=\overline{\Lambda_{\overrightarrow{01}}\left(z; \gamma \right)} \in \mathbb{C} \langle \langle X,Y \rangle \rangle^{\times},$ \\ 

$(\gamma,\sigma) \in \pi_1^{\rm top}\left(\mathbb{P}^1(\mathbb{C}) \backslash \{0,1,\infty\}; \overrightarrow{01}, z\right) \times G_K$ & $\gamma \in \pi_1^{\rm top}\left(\mathbb{P}^1(\mathbb{C}) \backslash \{0,1,\infty\}; \overrightarrow{01}, z\right)$ \\ \hline

  ${Li}^{\ell}_{k}(z;\gamma,\sigma) \in  \mathbb{Q}_{\ell}$ & ${Li}_{k}(z;\gamma) \in  \mathbb{C}$ \\ \hline 

   ${\boldsymbol \zeta}_{{k}}^{\ell}(\sigma) \in \mathbb{Q}_{\ell}$ & $\zeta({k}) \in \mathbb{R}$ \\ \hline 

${Li}^{\ell}_{1}(z;\gamma,\sigma)=\rho_{1-z,\gamma'}(\sigma) \in \bZ_{\ell}$ & $Li_{1}(z;\gamma)=-\log(1-z;\gamma') \in \bC$ \\ \hline 

$\tilde{\chi}_{k}
^{z,\gamma}(\sigma) \in \bZ_{\ell}$ & $-(k-1)! \cdot \Coeff_{X^{k-1}Y}\left( \Lambda_{\overrightarrow{01}}\left(z; \gamma \right) \right) \in \bC$ \\ \hline 

${\ell i}_{k}(z;\gamma,\sigma) \in \bQ_{\ell}$ & $\varphi_{k} \left({\bf log}\left(G_{\overrightarrow{01}}\left(z; \gamma \right)^{-1}\right)\right) \in \bC$ \\ \hline 

$\varphi_{k,\vec{l}}\left({\bf log}\left(\left(\overline{{\mathfrak f}^{z,\gamma}_{\overrightarrow{01},\sigma}}\right)^{-1}\right)\right) \in \bQ_{\ell}$ & ${\rm li}_{k}(z;\gamma) \in \bC$ \\ \hline 


$\chi(\sigma) \in \mathbb{Z}_{\ell}^\times$ & $2 \pi \sqrt{-1} \in \mathbb{C}$ \\ \hline 

  \end{tabular}
\end{table}
\vspace{\stretch{2}}
\pagebreak

\section{Proof of the main results}
In this section,
we prove the main results, Theorem~\ref{main1} and Theorem~\ref{main2}.
We fix
a path
$\gamma_0 \in \pi_1^{\rm top}\left(V_{\rm non\text{-}Fano}^\an;{\vec{v}},(x,y)\right)$.
For $i=1,...,9$,
we set $\delta_i \in \pi_1^{\rm top}\left({\bP}^{1}(\bC) \backslash \{0,1,\infty\};\overrightarrow{01}, f_i^{\an}\left(\vec{v}\right)\right)$ and
$
\gamma_i \in \pi_1^{\rm top}\left({\bP}^{1}(\bC) \backslash \{0,1,\infty\};\overrightarrow{01}, f^{\an}_{i}(x,y)\right)
$
as in TABLE \ref{table_delta} and
(\ref{gammai}).
For a (topological) group $\pi$ (resp. Lie algebra $L$),
we denote by
$\{\Gamma^k \pi\}_k$~(resp.~$\{\Gamma^k L\}_k$)
the lower central series of $\pi$~(resp.~$L$) with
$\Gamma^1 \pi=\pi$~(resp.~$\Gamma^1 L=L$).
We set
\[
{\rm gr}_{\Gamma}^{k} \pi:=\Gamma^k \pi/\Gamma^{k+1} \pi~
(\text{resp.}~{\rm gr}_{\Gamma}^{k} L:=\Gamma^k L/\Gamma^{k+1} L).
\]
The commutator bracket $(g,g'):=gg'g^{-1}{g'}^{-1}~(g,g' \in \pi)$ of $\pi$ then induces the Lie bracket
\[
[\alpha,\beta]:=(\alpha,\beta)~{\rm mod}~\Gamma^{n+m+1} \pi~\in {\rm gr}_{\Gamma}^{n+m} \pi
\quad (\alpha \in {\rm gr}_{\Gamma}^{n} \pi,~\beta \in {\rm gr}_{\Gamma}^{m} \pi)
\]
on the graded sum $\oplus_{k=1}^{\infty} {\rm gr}_{\Gamma}^{k} \pi$.

\subsection{Complex case}
In this subsection,
we derive the Spence--Kummer trilogarithm functional equation. For this purpose,
we also derive 
the dilogarithm functional equations of Schaeffer, Kummer, and Hill.
These functional equations are derived by using Nakamura--Wojtkowiak's tensor and homotopy criteria \cite[Theorem~5.7,~Proposition 5.11]{NW12}.
We fix a $\bC$-rational point $(x,y)$
of $V_{\rm non\text{-}Fano}$.

\begin{figure}[htbp]
\centering
\caption{Structure of the proofs for main results (complex case).}
\label{fig:proof_structure_complex}
\vspace{0.8cm}
\begin{tikzpicture}[
node distance = 0.7cm and 0.5cm,
box/.style = {draw, rectangle, rounded corners, minimum width=4cm, minimum height=1cm,
              align=center, fill=blue!5, inner sep=2ex, font=\small},
plain/.style = {align=center, font=\small},
arrow/.style = {-{Latex}, thick},
swatch/.style = {draw, minimum width=0.45cm, minimum height=0.45cm, inner sep=0pt}
]
 
\node (path) [box, fill=green!15, text width=6cm]
  {\textbf{Path composition} (\ref{gammai})
};
 
\node (chain) [box, below=0.7cm of path, fill=yellow!25, text width=6cm]
  {\textbf{Chain rule} (\ref{chain-c})};
 
\node (eq_c) [box, below=0.7cm of chain, fill=red!15, text width=9.5cm]
  {\textbf{General functional equations for\\
    complex iterated integrals (\ref{c-func-wii}), (\ref{c-func-wii3})}\\[0.5ex]
   {\footnotesize by the tensor-homotopy criteria \cite{NW12}\\
   (RHS$\,{=}\,0$; no error terms)}};
 
\coordinate (split) at ([yshift=-0.6cm]eq_c.south);
 
\node (lhs_route) [plain, below left=0.5cm and 0.6cm of split, text width=5.5cm]
  {\textbf{[ LHS computation route ]}\\[0.5ex]
   Polylog--BCH formula \cite[Proposition 5.9]{NW12} + Computations of
   Drinfeld associators (\ref{c_assoc})};
 
\node (rhs_route) [plain, below right=0.5cm and 0.6cm of split, text width=5.5cm]
  {\textbf{[Lower-degree term route ]}\\[0.5ex]
   Table~\ref{table-main} + Computations of
   Drinfeld associators (\ref{c_assoc})};
 
\node (eq_cii) [box, below=0.5cm of lhs_route, text width=5.5cm]
  {Computations of complex iterated integrals (\ref{cii2}), (\ref{cii3})};
 
\node (compassoc) [box, below=0.5cm of rhs_route, text width=5.5cm]
  {Computations of lower-degree terms in
   ${\bf log}(\Lambda^{-1})$: (\ref{c_compassoc})};
 
\node (logbranch) [box, below=0.6cm of compassoc, text width=5.5cm]
  {Explicit coefficients of lower-degree terms in
   ${\bf log}(\Lambda^{-1})$: (\ref{C-CompC})};
 
 
\node (sub_lhs) [box, below=0.6cm of eq_cii, text width=5.5cm] 
  {Substitution into LHS of (\ref{c-func-wii}), (\ref{c-func-wii3})};
 
\coordinate (join) at ([yshift=-1.3cm]$(sub_lhs)!0.5!(logbranch)$);
 
\node (sub_conv) [box, below=0.9cm of join, text width=10cm]
  {Apply conversion formulas (\ref{liCtoLi}), (\ref{c_assoc})
   and functional equations for complex logarithms (\ref{logbranch})};
 
\node (thm_c) [box, fill=orange!20, below=0.7cm of sub_conv, text width=10cm]
  {\textbf{Main theorems: Theorem \ref{c-dilog}, Theorem \ref{main1}}\\
   (Explicit functional equations for complex polylogarithms)};
 
\draw [arrow] (path) -- (chain);
\draw [arrow] (chain) -- (eq_c);
 
\draw [thick] (eq_c.south) -- (split);
\draw [thick] (split) -| (lhs_route);
\draw [thick] (split) -| (rhs_route);
 
\draw [arrow] (lhs_route) -- (eq_cii);
\draw [arrow] (rhs_route) -- (compassoc);
\draw [arrow] (compassoc) -- (logbranch);
 
\draw [arrow] (eq_cii) -- (sub_lhs);
 
\draw [thick] (sub_lhs.south) |- (join);
\draw [thick] (logbranch.south) |- (join);
\draw [arrow] (join) -- (sub_conv);
 
\draw [arrow] (sub_conv) -- (thm_c);
 
\node (legtitle) [below=1.4cm of thm_c] {\textbf{}};
 
\node (sw1) [swatch, fill=green!15, below=-1.0cm of legtitle, xshift=-2cm] {};
\node (tx1) [right=0.25cm of sw1, anchor=west, font=\small] {Geometric input};
 
\node (sw2) [swatch, fill=yellow!25, below=0.3cm of sw1] {};
\node (tx2) [right=0.25cm of sw2, anchor=west, font=\small] {Key identity};
 
\node (sw3) [swatch, fill=red!15, below=0.3cm of sw2] {};
\node (tx3) [right=0.25cm of sw3, anchor=west, font=\small] {Result from \cite{NW12}};
 
\node (sw4) [swatch, fill=orange!20, below=0.3cm of sw3] {};
\node (tx4) [right=0.25cm of sw4, anchor=west, font=\small] {Final result};
 
\node (sw5) [swatch, fill=blue!5, below=0.3cm of sw4] {};
\node (tx5) [right=0.25cm of sw5, anchor=west, font=\small] {Auxiliary computation};
 
\end{tikzpicture}
\end{figure}

First,
we prepare to verify the tensor homotopy criteria \cite[Theorem~5.7, $({\rm i})_{\bC}$,~$({\rm ii})_{\bC}$]{NW12}.
For $i=1,...,9$,
we write
\[
{f_{i\ast}^{\an}}: \pi_1^{\rm top}\left(V_{\rm non\text{-}Fano}^{\an}, \vec{v}\right) \to \pi_1^{\rm top}\left({\bP}^{1}(\bC) \backslash \{0,1,\infty\}, f_i^{\an}\left(\vec{v}\right)\right)
\]
for the homomorphism induced by $f_{i}^{\an}$
where $f_i~(i=1,...,9)$ are as in (\ref{fi}).
We denote by
\[
\iota_{\delta_{i}}^{\an}: \pi_1^{\rm top}\left({\bP}^{1}(\bC) \backslash \{0,1,\infty\}, f_i^{\an}\left(\vec{v}\right)\right) \xrightarrow{\simeq} \pi_1^{\rm top}\left({\bP}^{1}(\bC) \backslash \{0,1,\infty\}, \overrightarrow{01}\right)
\]
the change-of-basepoint homomorphism defined by
$p \mapsto \delta_{i} \cdot p \cdot {\delta_{i}}^{-1}$
where $\delta_i~(i=1,...,9)$ are as in TABLE \ref{table_delta}.
The images of $B_j~(j=1,...,8)$ in TABLE \ref{table3} under $\iota_{\delta_{i}}^{\an} \circ {f_{i\ast}^{\an}}$ are calculated as shown in TABLE \ref{table-main}.
\renewcommand{\arraystretch}{1.8}
\tabcolsep = 0.3cm
\begin{table}[hbtp]
  \caption{$\iota_{\delta_{i}}^{\an}\left(f^{\an}_i(B_j)\right)$}
  \label{table-main}
  \centering
  \begin{tabular}{|c||c|c|c|c|c|c|c|c|c|}
    \hline
    \# & $B_1$ & $B_2$ & $B_3$ & $B_4$ & $B_5$ & $B_6$ & $B_7$ & $B_8$ \\
    \hline \hline

    $\iota_{\delta_{1}}^{\an}\left(f^{\an}_{1\ast}(\#)\right)$  & $l_0$ & $l_0^{-1} \cdot l_1^{-1}$  & $l_1$ & $l_{\infty}^2$ & $l_0^2$ & $l_0 \cdot l_1 \cdot l_0^{-1}$ & $l_1^{-1} \cdot l_0 \cdot l_1$ & $l_{\infty}$ \\ \hline

    $\iota_{\delta_{2}}^{\an}\left(f^{\an}_{2\ast}(\#)\right)$  & $l_0$ & $l_0$  & $1$ & $1$ & $1$ & $l_1$ & $l_0^{-1} \cdot l_1^{-1}$ & $l_0^{-1} \cdot l_1^{-1}$ \\ \hline

    $\iota_{\delta_{3}}^{\an}\left(f^{\an}_{3\ast}(\#)\right)$  & $l_0$ & $l_0^{-1} \cdot l_1^{-1}$  & $l_1$ & $1$ & $1$ & $1$ & $l_\infty$ & $l_0$ \\ \hline

    $\iota_{\delta_{4}}^{\an}\left(f^{\an}_{4\ast}(\#)\right)$  & $l_0$ & $l_0^{-1} \cdot l_1^{-1}$  &  $l_1$ & $l_\infty$ & $l_0$ & $1$ & $1$ & $1$ \\
    \hline

    $\iota_{\delta_{5}}^{\an}\left(f^{\an}_{5\ast}(\#)\right)$ & $l_0$ & $1$  &  $1$ & $l_\infty$ & $l_0$ & $l_0 \cdot l_1 \cdot l_0^{-1}$ & $1$ & $l_\infty$ \\
    \hline

    $\iota_{\delta_{6}}^{\an}\left(f^{\an}_{6\ast}(\#)\right)$  & $1$ & $1$  &  $l_1$ & $l_\infty$ & $l_0$ & $1$ & $l_1^{-1} \cdot l_0 \cdot l_1$ & $l_\infty$ \\
    \hline

    $\iota_{\delta_{7}}^{\an}\left(f^{\an}_{7\ast}(\#)\right)$  & $1$ & $l_\infty$  &  $1$ & $l_\infty$ & $l_0$ & $l_0 \cdot l_1 \cdot l_0^{-1}$ & $l_0$ & $1$ \\
    \hline

    $\iota_{\delta_{8}}^{\an}\left(f^{\an}_{8\ast}(\#)\right)$  & $l_0$ &  $1$  &  $1$ & $l_1$ & $1$ & $1$ & $l_0^{-1} \cdot l_1^{-1}$ & $1$ \\
    \hline

    $\iota_{\delta_{9}}^{\an}\left(f^{\an}_{9\ast}(\#)\right)$  & $1$ &  $l_0$ &  $1$ & $1$ & $l_1$ & $1$ & $1$ & $l_0^{-1} \cdot l_1^{-1}$ \\
    \hline
  \end{tabular}
\end{table}

\noindent
We focus on (\ref{prodfun}):
\begin{align*} 
\pi_1^{\rm top}\left(V_{\rm non\text{-}Fano}^{\an},\vec{v}\right)^{\rm ab} &= \left<\bar{B}_1,\bar{B}_4\right> \times \left<\bar{B}_2,\bar{B}_3,\bar{B}_5,\bar{B}_6\right>\\
& = F_2^{\ab} \times F_4^{\ab}.
\end{align*}
Since
${\rm rank}_{\mathbb{Z}}({\rm gr}_{\Gamma}^{2}F_2)=1$ and 
${\rm rank}_{\mathbb{Z}}({\rm gr}_{\Gamma}^{2}F_4)=6$
by Witt's formula,
\[
{\rm gr}_{\Gamma}^{2}\pi_1^{\rm top}\left(V_{\rm non\text{-}Fano}^{\an},\vec{v}\right)={\rm gr}_{\Gamma}^{2}F_2 \times {\rm gr}_{\Gamma}^{2}F_4
\]
has a free generating system
\begin{align} \label{gr2gen}
[\bar{B}_1,\bar{B}_4],\quad
[\bar{B}_2,\bar{B}_3],\quad
[\bar{B}_2,\bar{B}_5],\quad
[\bar{B}_2,\bar{B}_6],\quad
[\bar{B}_3,\bar{B}_5],\quad
[\bar{B}_3,\bar{B}_6],\quad
[\bar{B}_5,\bar{B}_6]
\end{align}
as a finitely generated free $\mathbb{Z}$-module.
Since
${\rm rank}_{\mathbb{Z}}({\rm gr}_{\Gamma}^{3}F_2)=2$ and
${\rm rank}_{\mathbb{Z}}({\rm gr}_{\Gamma}^{3}F_4)=20$,
\[
{\rm gr}_{\Gamma}^{3}\pi_1^{\rm top}\left(V_{\rm non\text{-}Fano}^{\an},\vec{v}\right)={\rm gr}_{\Gamma}^{3}F_2 \times {\rm gr}_{\Gamma}^{3}F_4
\]
has a free generating system
\[
\left[\bar{B}_1,\left[\bar{B}_1,\bar{B}_4\right]\right],\quad
\left[\bar{B}_4,\left[\bar{B}_1,\bar{B}_4\right]\right],\quad
\left[\bar{B}_2,\left[\bar{B}_2,\bar{B}_3\right]\right],\quad
\left[\bar{B}_3,\left[\bar{B}_2,\bar{B}_3\right]\right],\quad
\left[\bar{B}_2,\left[\bar{B}_2,\bar{B}_5\right]\right],\quad
\]
\[
\left[\bar{B}_5,\left[\bar{B}_2,\bar{B}_5\right]\right],\quad
\left[\bar{B}_2,\left[\bar{B}_2,\bar{B}_6\right]\right],\quad
\left[\bar{B}_6,\left[\bar{B}_2,\bar{B}_6\right]\right],\quad
\left[\bar{B}_3,\left[\bar{B}_3,\bar{B}_5\right]\right],\quad
\left[\bar{B}_5,\left[\bar{B}_3,\bar{B}_5\right]\right],\quad
\]
\[
\left[\bar{B}_3,\left[\bar{B}_3,\bar{B}_6\right]\right],\quad
\left[\bar{B}_6,\left[\bar{B}_3,\bar{B}_6\right]\right],\quad
\left[\bar{B}_5,\left[\bar{B}_5,\bar{B}_6\right]\right],\quad
\left[\bar{B}_6,\left[\bar{B}_5,\bar{B}_6\right]\right],\quad
\left[\bar{B}_2,\left[\bar{B}_3,\bar{B}_5\right]\right],\quad
\]
\[
\left[\bar{B}_3,\left[\bar{B}_5,\bar{B}_2\right]\right],\quad
\left[\bar{B}_2,\left[\bar{B}_3,\bar{B}_6\right]\right],\quad
\left[\bar{B}_3,\left[\bar{B}_6,\bar{B}_2\right]\right],\quad
\left[\bar{B}_2,\left[\bar{B}_5,\bar{B}_6\right]\right],\quad
\left[\bar{B}_5,\left[\bar{B}_6,\bar{B}_2\right]\right],\quad
\]
\[
\left[\bar{B}_3,\left[\bar{B}_5,\bar{B}_6\right]\right],\quad
\left[\bar{B}_5,\left[\bar{B}_6,\bar{B}_3\right]\right].
\]
Moreover,
for $i=1,...,9$,
by considering the composite
$V_{\rm non\text{-}Fano}^{\an} \xrightarrow{f_i^{\an}} {\bP}^{1}(\bC) \backslash \{0,1,\infty\} \hookrightarrow \bG_{m}(\bC)
$,
we may regard $f_i^{\an}$ as an element of the unit group
\[
f_i^\an \in \left({\CalO}^{\an}\right)^{\times},
\]
where ${\CalO}^{\an}$ is the coordinate ring 
of $V_{\rm non\text{-}Fano}^{\an}$:
\[
{\CalO}^{\an}={\bC}\left[s_1,s_2,\frac{1}{s_1 s_2(1-s_1)(1-s_2)(s_1-s_2)(1-s_1 s_2)}\right].
\]
Then,
\[
f_i^{\an}-1 \in \left({\CalO}^{\an}\right)^{\times}
\]
holds.
Under the above preparations,
the tensor homotopy criteria \cite[Theorem~5.7, $({\rm i})_{\bC}$,~$({\rm ii})_{\bC}$]{NW12} are verified in the proof of Theorem \ref{c-dilog} and Theorem \ref{main1}.

Next,
we prepare to calculate the functional equations of
complex iterated integrals \cite[Theorem~5.7, $({\rm iii})_{\bC}$]{NW12}.
By (\ref{clidef}) and (\ref{c-explicit}),
for $i=1,...,9$,
we have
\begin{equation}
\label{liCtoLi}
\begin{dcases}
\mathrm{li}_0\left(f^{\an}_{i}(x,y);\gamma_i\right)=&-\frac{1}{2 \pi \sqrt{-1}} \log\left(f^{\an}_{i}(x,y);\gamma_i\right), \\
\mathrm{li}_1\left(f^{\an}_{i}(x,y);\gamma_i\right)=&-\frac{1}{2 \pi \sqrt{-1}} \log\left(1-f^{\an}_{i}(x,y);\gamma'_i\right), \\
\mathrm{li}_2\left(f^{\an}_{i}(x,y);\gamma_i\right)=&\frac{1}{4\pi^2} \left(Li_2\left(f^{\an}_{i}(x,y);\gamma_i\right)+\frac12 \log\left(f^{\an}_{i}(x,y);\gamma_i\right)\log\left(1-f^{\an}_{i}(x,y);\gamma'_i\right)\right), \\
\mathrm{li}_3\left(f^{\an}_{i}(x,y);\gamma_i\right)=&\frac{-1}{8\pi^3 \sqrt{-1}} \left(Li_3\left(f^{\an}_{i}(x,y);\gamma_i\right)-\frac12 \log\left(f^{\an}_{i}(x,y);\gamma_i\right) Li_2\left(f^{\an}_{i}(x,y);\gamma_i\right)\right. \\
&\left.-\frac{1}{12} \log^{2}\left(f^{\an}_{i}(x,y);\gamma_i\right) \log\left(1-f^{\an}_{i}(x,y);\gamma'_i\right)\right),
\end{dcases}
\end{equation}
where
$
\gamma'_i \in \pi_1^{\rm top}\left(\mathbb{P}^1(\mathbb{C}) \backslash \{0,1,\infty\}; \overrightarrow{01}, f^{\an}_{i}(x,y)\right)
$
is the path associated
with $\gamma_i$ as in (\ref{gamma'}).
Furthermore,
by calculating the Drinfeld associators (cf. \cite[\S5.4]{NW12}),
we have
\begin{equation} \label{c_assoc}
\begin{dcases}
&\left(-\mathrm{li}_j\left(\ovec{01};\delta_i\right)\right)_{0\le j \le 3}=(0,0,0,0)~~(i=1,2,3,4,8,9), \\
&\left(-\mathrm{li}_j\left(\ovec{0\infty};\delta_5\right)\right)_{0\le j\le 3}=\left(\frac{1}{2},0,0,0\right),\\
&\left(-\mathrm{li}_j\left(\ovec{10};\delta_6\right)\right)_{0\le j \le 3}
=\left(0,0, -\frac{1}{24}, \frac{1}{8\pi^3 \sqrt{-1}}Li_3(1)\right), \\
&\left(-\mathrm{li}_j\left(\ovec{\infty0};\delta_7\right)\right)_{0\le j\le 3}=\left(\frac{1}{2},0,\frac{1}{24},0\right).
\end{dcases}
\end{equation}
Here,
we consider a generating function of complex iterated integrals
\[
\Lambda_{f_i^{\an}\left(\vec{v}\right)}\left(f^{\an}_{i}(x,y); f_i^{\an}\left(\gamma_0\right) \right):=1+\sum_{n=1}^{\infty} \int_{f_i^{\an}\left(\gamma_0\right)} {\underbrace{\omega \ldots \omega}_{n~\text{times}}} \quad (i=1,...,9),
\]
where $\omega=\frac{dz}{z}X+\frac{dz}{z-1}Y$.
Then,
by the path composition~$\gamma_i=\delta_i \cdot f_i^{\an}\left(\gamma_0\right)$,
we obtain an algebraic relation
\begin{align}\label{chain-c}
\Lambda_{\overrightarrow{01}}\left(f^{\an}_{i}(x,y); \gamma_i \right)=
\Lambda_{\overrightarrow{01}}\left(f_i^{\an}\left(\vec{v}\right); \delta_i \right)
\cdot \Lambda_{f_i^{\an}\left(\vec{v}\right)}\left(f^{\an}_{i}(x,y); f_i^{\an}\left(\gamma_0\right) \right)
\quad (i=1,...,9)
\end{align}
in $\mathbb{C} \langle \langle X, Y \rangle \rangle$.
By (\ref{chain-c}), (\ref{c_assoc}) and the polylog-BCH formula \cite[Proposition 5.9]{NW12},
we compute Wojtkowiak's complex iterated integrals (\cite[Definition 4.4]{NW12})
\[
\mathcal{L}^{\varphi_k}_{\C}\left(f^{\an}_{i}(x,y);f_i^{\an}\left(\vec{v}\right), f_i^{\an}\left(\gamma_0\right)\right):=\varphi_{k} \left({\bf log}\left(\Lambda_{f_i^{\an}\left(\vec{v}\right)}\left(f^{\an}_{i}(x,y); f_i^{\an}\left(\gamma_0\right) \right)^{-1}\right)\right) \quad (i=1,...,9)
\]
associated with
$f_i^{\an}\left(\gamma_0\right)$
as follows:
\begin{equation} \label{cii2}
\begin{dcases}
\mathcal{L}^{\varphi_2}_{\C}\left(f^{\an}_{i}(x,y);f_i^{\an}\left(\vec{v}\right), f_i^{\an}\left(\gamma_0\right)\right)
=&\mathrm{li}_2\left(f^{\an}_{i}(x,y);\gamma_i\right)~\quad~(i=1,2,3,4,8,9), 
\\
\mathcal{L}^{\varphi_2}_{\C}\left(f^{\an}_{5}(x,y);f_5^{\an}\left(\vec{v}\right), f_5^{\an}\left(\gamma_0\right)\right)
=&
\mathrm{li}_2(f^{\an}_{5}(x,y);\gamma_5)+\tfrac12 \mathrm{li}_0\left(\overrightarrow{0\infty};\delta_5\right) \mathrm{li}_1(f^{\an}_{5}(x,y);\gamma_5),\\
\mathcal{L}^{\varphi_2}_{\C}\left(f^{\an}_{6}(x,y);f_6^{\an}\left(\vec{v}\right), f_6^{\an}\left(\gamma_0\right)\right)
=&
\mathrm{li}_2(f^{\an}_{6}(x,y);\gamma_6)-\mathrm{li}_2\left(\ovec{10};\delta_6\right),
\\
\mathcal{L}^{\varphi_2}_{\C}\left(f^{\an}_{7}(x,y);f_7^{\an}\left(\vec{v}\right), f_7^{\an}\left(\gamma_0\right)\right)
=&
\mathrm{li}_2(f^{\an}_{7}(x,y);\gamma_7)-\mathrm{li}_2\left(\overrightarrow{\infty0};\delta_7\right)
-\tfrac12
\mathrm{li}_0(f^{\an}_{7}(x,y);\gamma_7)
\mathrm{li}_1\left(\overrightarrow{\infty0};\delta_7\right)\\
&+\tfrac12
\mathrm{li}_1(f^{\an}_{7}(x,y);\gamma_7)
\mathrm{li}_0\left(\overrightarrow{\infty0};\delta_7\right),
\end{dcases}
\end{equation}
\begin{equation} \label{cii3}
\begin{dcases}
\mathcal{L}^{\varphi_3}_{\C}\left(f^{\an}_{i}(x,y);f_i^{\an}\left(\vec{v}\right), f_i^{\an}\left(\gamma_0\right)\right)
=&\mathrm{li}_3\left(f^{\an}_{i}(x,y);\gamma_i\right)~\quad~(i=1,2,3,4,8,9), 
\\
\mathcal{L}^{\varphi_3}_{\C}\left(f^{\an}_{5}(x,y);f_5^{\an}\left(\vec{v}\right), f_5^{\an}\left(\gamma_0\right)\right)
=&
\mathrm{li}_3\left(f^{\an}_{5}(x,y);\gamma_5\right)+\tfrac12 \mathrm{li}_0\left(\overrightarrow{0\infty};\delta_5\right) \mathrm{li}_2\left(f^{\an}_{5}(x,y);\gamma_5\right) 
\\
&+\tfrac{1}{12} \mathrm{li}_0\left(\overrightarrow{0\infty};\delta_5\right) \mathrm{li}_0\left(f^{\an}_{5}(x,y);\gamma_5\right) \mathrm{li}_1\left(f^{\an}_{5}(x,y);\gamma_5\right)\\
&+\tfrac{1}{12} \left(\mathrm{li}_0\left(\overrightarrow{0\infty};\delta_5\right)\right)^2 \mathrm{li}_1\left(f^{\an}_{5}(x,y);\gamma_5\right),\\
\mathcal{L}^{\varphi_3}_{\C}\left(f^{\an}_{6}(x,y);f_6^{\an}\left(\vec{v}\right), f_6^{\an}\left(\gamma_0\right)\right)
=&
\mathrm{li}_3\left(f^{\an}_{6}(x,y);\gamma_6\right)-\mathrm{li}_3\left(\ovec{10};\delta_6\right)
-\tfrac12
\mathrm{li}_0\left(f^{\an}_{6}(x,y);\gamma_6\right)
\mathrm{li}_2\left(\ovec{10};\delta_6\right)
\\
\mathcal{L}^{\varphi_3}_{\C}\left(f^{\an}_{7}(x,y);f_7^{\an}\left(\vec{v}\right), f_7^{\an}\left(\gamma_0\right)\right)
=&
\mathrm{li}_3\left(f^{\an}_{7}(x,y);\gamma_7\right)-\mathrm{li}_3\left(\overrightarrow{\infty0};\delta_7\right)
-\tfrac12
\mathrm{li}_0\left(f^{\an}_{7}(x,y);\gamma_7\right)
\mathrm{li}_2\left(\overrightarrow{\infty0};\delta_7\right)\\
&+\tfrac12
\mathrm{li}_2\left(f^{\an}_{7}(x,y);\gamma_7\right)
\mathrm{li}_0\left(\overrightarrow{\infty0};\delta_7\right) \\
&+\tfrac{1}{12} \left(\mathrm{li}_0\left(\overrightarrow{\infty0};\delta_7\right)\right)^2 \mathrm{li}_1\left(f^{\an}_{7}(x,y);\gamma_7\right)\\
&+\tfrac{1}{12} \mathrm{li}_0\left(\overrightarrow{\infty0};\delta_7\right) \mathrm{li}_0\left(f^{\an}_{7}(x,y);\gamma_7\right) \mathrm{li}_1\left(f^{\an}_{7}(x,y);\gamma_7\right).
\end{dcases}
\end{equation}
We also examine the relationships among $\left\{\log\left(f_{i}^{\an}(x,y);\gamma_i\right), \log\left(1-f_{i}^{\an}(x,y);\gamma'_i\right)\right\}_{i=1,...,9}$.
Let
\[
L\left( \pi_1^{\rm top}\left(V_{\rm non\text{-}Fano}^{\an},\vec{v}\right) \right) \quad \left({\rm resp.~}U\left( \pi_1^{\rm top}\left(V_{\rm non\text{-}Fano}^{\an},\vec{v}\right) \right)\right)
\]
be the complete Lie algebra of $\pi_1^{\rm top}\left(V_{\rm non\text{-}Fano}^{\an},\vec{v}\right)$ over $\bC$ (resp. the complete Hopf algebra, given as the universal enveloping algebra of $L\left( \pi_1^{\rm top}\left(V_{\rm non\text{-}Fano}^{\an},\vec{v}\right) \right)$).
Then,\
there is a natural inclusion
\[
\pi_1^{\rm top}\left(V_{\rm non\text{-}Fano}^{\an},\vec{v}\right) \hookrightarrow L\left( \pi_1^{\rm top}\left(V_{\rm non\text{-}Fano}^{\an},\vec{v}\right) \right),~B_i \mapsto X_i:={\bf log}(B_i)
\]
where $B_i~(i=1,...,8)$ are as in TABLE \ref{table3}.
Each element of $L\left( \pi_1^{\rm top}\left(V_{\rm non\text{-}Fano}^{\an},\vec{v}\right) \right)$
has an expansion as a formal Lie series in 
$X_1,...,X_6$.
Let
\[
\Lambda_{\vec{v}}\left((x,y); \gamma_0\right) \in U\left( \pi_1^{\rm top}\left(V_{\rm non\text{-}Fano}^{\an},\vec{v}\right) \right)
\]
be a horizontal section along $\gamma_0 \in \pi_1^{\rm top}\left(V_{\rm non\text{-}Fano}^\an;{\vec{v}},(x,y)\right)$ of the trivial principal bundle
$V_{\rm non\text{-}Fano}^{\an} \times U\left( \pi_1^{\rm top}\left(V_{\rm non\text{-}Fano}^{\an},\vec{v}\right) \right) \to V_{\rm non\text{-}Fano}^{\an}$~(cf. \cite[\S 1]{W97}, \cite[\S 4.1]{NW12}).
The associated horizontal section 
starting from $\left(\vec{v},1\right)$ over $\gamma_0$ terminates at the point $\left((x,y),\Lambda_{\vec{v}}\left((x,y); \gamma_0\right)\right)$.
We denote by
$L_{n}\left(\pi_1^{\rm top}\left(V_{\rm non\text{-}Fano}^{\an},\vec{v}\right)\right)$
the part with homogeneity degree $n$ 
and by $L_{<n}\left(\pi_1^{\rm top}\left(V_{\rm non\text{-}Fano}^{\an},\vec{v}\right)\right)$
the part whose homogeneity degree is less than $n$.
Then,
we have a decomposition
\[
L\left( \pi_1^{\rm top}\left(V_{\rm non\text{-}Fano}^{\an},\vec{v}\right) \right)=L_{<n}\left(\pi_1^{\rm top}\left(V_{\rm non\text{-}Fano}^{\an},\vec{v}\right)\right) \oplus \Gamma^nL\left( \pi_1^{\rm top}\left(V_{\rm non\text{-}Fano}^{\an},\vec{v}\right) \right)
\]
and
\[
{\boldsymbol \log} \left(\Lambda_{\vec{v}}\left((x,y); \gamma_0\right)^{-1}\right)=\left[{\boldsymbol \log} \left(\Lambda_{\vec{v}}\left((x,y); \gamma_0\right)^{-1}\right)\right]_{<n} \oplus \left[{\boldsymbol \log} \left(\Lambda_{\vec{v}}\left((x,y); \gamma_0\right)^{-1}\right)\right]_{\geq n}.
\]
There is an isomorphism
\[
L_{n}\left(\pi_1^{\rm top}\left(V_{\rm non\text{-}Fano}^{\an},\vec{v}\right)\right) \simeq {\rm gr}_{\Gamma}^{n}\left(\pi_1^{\rm top}\left(V_{\rm non\text{-}Fano}^{\an},\vec{v}\right)\right) \otimes \bC
\]
induced by
$L_{1}\left(\pi_1^{\rm top}\left(V_{\rm non\text{-}Fano}^{\an},\vec{v}\right)\right) \ni X_i \mapsto \bar{B}_i \in {\rm gr}_{\Gamma}^{1}\left(\pi_1^{\rm top}\left(V_{\rm non\text{-}Fano}^{\an},\vec{v}\right)\right)$.
Therefore, we write
\[
\left[{\boldsymbol \log} \left(\Lambda_{\vec{v}}\left((x,y); \gamma_0\right)^{-1}\right)\right]_{<2}
=C_1 X_1+C_2 X_2+C_3 X_3+C_4 X_4+C_5 X_5+C_6 X_6.
\]
Using TABLE \ref{table-main},
we compute
\begin{align} \label{c_compassoc} 
&\left[{\bf log}\left(\Lambda_{f_i^{\an}\left(\vec{v}\right)}\left(f^{\an}_{i}(x,y); f_i^{\an}\left(\gamma_0\right) \right)^{-1}\right)\right]_{<2}\\
&=
f^{\an}_{i\ast}\left(\left[{\boldsymbol \log} \left(\Lambda_{\vec{v}}\left((x,y); \gamma_0\right)^{-1}\right)\right]_{<2}\right)\notag \\
&=
\begin{dcases}
&\left(C_1-C_2-2C_4+2C_5\right)X+\left(-C_2+C_3-2C_4+C_6\right)Y \quad (i=1), \\
&\left(C_1+C_2\right)X+C_6 Y\quad (i=2),\\
&\left(C_1-C_2\right)X+\left(-C_2+C_3\right)Y\quad (i=3),\\
&\left(C_1-C_2-C_4+C_5\right)X+\left(-C_2+C_3-C_4\right)Y\quad (i=4),\\
&\left(C_1-C_4+C_5\right)X+\left(-C_4+C_6\right)Y\quad (i=5),\\
&\left(-C_4+C_5\right)X+\left(C_3-C_4\right)Y\quad (i=6),\\
&\left(-C_2-C_4+C_5\right)X+\left(-C_2-C_4+C_6\right)Y\quad (i=7),\\
&C_1 X+C_4 Y\quad (i=8),\\
&C_2 X+C_5 Y\quad (i=9).
\end{dcases} \notag
\end{align} 
By (\ref{c_compassoc}), (\ref{chain-c}), (\ref{liCtoLi}) and (\ref{c_assoc}),
we obtain
\begin{align}\label{C-CompC}
&C_1=-\log(x;\gamma_8),&
&C_2=-\log(y;\gamma_9),&
&C_3=-\log\left(1-\frac{x}{y};\gamma'_3\right)-\log(y;\gamma_9), \\
&C_4=-\log(1-x;\gamma'_8),&
&C_5=-\log(1-y;\gamma'_9),&
&C_6=-\log(1-xy;\gamma'_2),\notag \\
\notag
\end{align}
and
\begin{equation} \label{logbranch}
\begin{dcases}
\log\left(\frac{x(1-y)^2}{y(1-x)^2};\gamma_1\right)
&=
\log\left(x;\gamma_8\right)+2\log\left(1-y;\gamma'_9\right)-\log\left(y;\gamma_9\right)-2\log\left(1-x;\gamma'_8\right), \\
\log\left(1-\frac{x(1-y)^2}{y(1-x)^2};\gamma'_1\right)
&=
\log\left(1-\frac{x}{y};\gamma'_3\right)+\log(1-xy;\gamma'_2)-2\log\left(1-x;\gamma'_8\right),\\
\log\left(xy;\gamma_2\right)
&=
\log(x;\gamma_8)+\log(y;\gamma_9),\\
\log\left(\frac{x}{y};\gamma_3\right)
&=
\log(x;\gamma_8)-\log(y;\gamma_9),\\
\log\left(\frac{x(1-y)}{y(1-x)};\gamma_4\right)
&=
\log\left(x;\gamma_8\right)+\log\left(1-y;\gamma'_9\right)-\log\left(y;\gamma_9\right)-\log\left(1-x;\gamma'_8\right),\\
\log\left(1-\frac{x(1-y)}{y(1-x)};\gamma'_4\right)
&=
\log\left(1-\frac{x}{y};\gamma'_3\right)-\log\left(1-x;\gamma'_8\right),\\
\log\left(\frac{x(1-y)}{x-1};\gamma_5\right)
&=
\log(x;\gamma_8)+\log\left(1-y;\gamma'_9\right)-\log\left(1-x;\gamma'_8\right)+\pi i,\\
\log\left(1-\frac{x(1-y)}{x-1};\gamma'_5\right)
&=
\log(1-xy;\gamma'_2)-\log\left(1-x;\gamma'_8\right),\\
\log\left(\frac{1-y}{1-x};\gamma_6\right)
&=
\log\left(1-y;\gamma'_9\right)-\log\left(1-x;\gamma'_8\right),\\
\log\left(\frac{1-y}{y(x-1)};\gamma_7\right)
&=
\log\left(1-y;\gamma'_9\right)-\log\left(y;\gamma_9\right)-\log\left(1-x;\gamma'_8\right)+\pi i,\\
\log\left(1-\frac{1-y}{y(x-1)};\gamma'_7\right)
&=
\log(1-xy;\gamma'_2)-\log\left(y;\gamma_9\right)-\log\left(1-x;\gamma'_8\right).
\end{dcases}
\end{equation}

By combining these formulas,
we prove the functional equations of complex polylogarithms.
We set
\[
a_i, b_i, c_i, d_i \in \bZ
\]
as in TABLE \ref{table-abcd},
which will be used for the coefficients in the functional equations to be proved.
\renewcommand{\arraystretch}{1.8}
\tabcolsep = 0.3cm
\begin{table}[hbtp]
  \caption{$a_i, b_i, c_i, d_i$}
  \label{table-abcd}
  \centering
  \begin{tabular}{|c||c|c|c|c|c|c|c|c|c|c|}
    \hline
    $i$ & $1$ & $2$ & $3$ & $4$ & $5$ & $6$ & $7$ & $8$ & $9$ \\
    \hline \hline

    $f^{}_{i}(x,y)$ & $\frac{x(1-y)^2}{y(1-x)^2}$ & $xy$ & $\frac{x}{y}$ & $\frac{x(1-y)}{y(1-x)}$ & $\frac{x(1-y)}{x-1}$ & $\frac{1-y}{1-x}$ & $\frac{1-y}{y(x-1)}$ & $x$ & $y$ \\
    \hline \hline

    $a_i$  & $0$ & $0$  & $-1$ & $1$ & $0$ & $-1$ & $0$ & $1$ & $-1$ \\ \hline

    $b_i$  & $1$ & $0$  & $0$ & $-1$ & $-1$ & $-1$ & $-1$ & $0$ & $0$ \\ \hline

    $c_i$  & $0$ & $1$  & $0$ & $0$ & $-1$ & $0$ & $1$ & $-1$ & $-1$ \\ \hline

    $d_i$  & $1$ & $1$  & $1$ & $-2$ & $-2$ & $-2$ & $-2$ & $-2$ & $-2$ \\
    \hline
  \end{tabular}
\end{table}

\begin{thm}[Functional equations for complex dilogarithms]\label{c-dilog}
Given a $\bC$-rational point
$
(x,y) \in V_{\rm non\text{-}Fano}(\bC)
$
and a path $\gamma_0 \in \pi_1^{\rm top}\left(V_{\rm non\text{-}Fano}^\an;{\vec{v}},(x,y)\right)$,
define the path system $\{\gamma_i\}_{i=1,\ldots,9}$ associated with $\gamma_0$ as in (\ref{gammai}).
Then, the following holds.
\begin{description}
   \item[(a-$\bC$)~Schaeffer equation]
   \[
Li_{2}\left({\frac{x(1-y)}{y(1-x)};\gamma_4}\right)
-Li_{2}\left(y;\gamma_9\right)
+Li_{2}\left(x;\gamma_8\right)
-Li_{2}\left(\frac{x}{y};\gamma_3\right)
\]
\[
-Li_{2}\left({\frac{1-y}{1-x};\gamma_6}\right)
=\log(y;\gamma_9)\log\left(\dfrac{1-y}{1-x};\gamma_6\right)-\dfrac{\pi^2}{6}.
\]
   \item[(b-$\bC$)~Kummer equation]
\[
Li_{2}\left({\frac{x(1-y)^2}{y(1-x)^2};\gamma_1}\right)
-Li_{2}\left({\frac{x(1-y)}{x-1};\gamma_5}\right)
-Li_{2}\left({\frac{1-y}{y(x-1)};\gamma_7}\right)
\]
\[
-Li_{2}\left({\frac{x(1-y)}{y(1-x)};\gamma_4}\right)
-Li_{2}\left({\frac{1-y}{1-x};\gamma_6}\right)
=\frac{1}{2}\log^2(y;\gamma_9).
\]
   \item[(c-$\bC$)~Hill equation]
\[
Li_{2}\left({\frac{1-y}{y(x-1)};\gamma_7}\right)
+Li_{2}\left({xy;\gamma_2}\right)
-Li_{2}\left({x;\gamma_8}\right)
-Li_{2}\left({y;\gamma_9}\right)
-Li_{2}\left({\frac{x(1-y)}{x-1};\gamma_5}\right)
\]
\[
=-\dfrac{\pi^2}{6}
+\log(y;\gamma_9)\log\left(\frac{1-y}{1-x};\gamma_6\right)-\frac{1}{2}\log^2(y;\gamma_9).
\]
\end{description}
\end{thm}

\begin{proof}
Let $k_i \in \{a_i,b_i,c_i\}$, where $a_i, b_i$, and $c_i$ $(i=1,\ldots,9)$ are the coefficients given in TABLE \ref{table-abcd}. 
By combining the data scattered across TABLE \ref{table-sub1}, TABLE \ref{table-sub2}, and TABLE \ref{table-sub3}, 
we can readily verify that the homotopy criterion \cite[Theorem~5.7, $({\rm i})_{\bC}$]{NW12}
\[
\sum_{i=1}^{9} k_i \cdot \varphi_2\left(\operatorname{gr}_{\Gamma}^{2}\left( \iota_{\delta_i}^{\an} \circ f_{i\ast}^\an \right) \right)=0
\quad \text{in } \operatorname{Hom}_{\bZ}\left(\operatorname{gr}_{\Gamma}^{2}\left(\pi_1^{\rm top}\left(V_{\rm non\text{-}Fano}^{\an},\vec{v}\right)\right) ,\bZ \right)
\]
holds. Next, a straightforward calculation shows that the tensor criterion \cite[Theorem~5.7, $({\rm ii})_{\bC}$]{NW12}
\[
\sum_{i=1}^{9} k_i \cdot \left( f_i^{\an} \wedge \left(f_i^{\an}-1\right) \right) = 0
\quad \text{in } \left(\left(\mathcal{O}^{\an}\right)^{\times}/\bC^{\times}\right) \wedge \left(\left(\mathcal{O}^{\an}\right)^{\times}/\bC^{\times}\right)
\]
is also satisfied. Consequently, we obtain the following functional equation via \cite[Theorem~5.7, $({\rm iii})_{\bC}$]{NW12}:
\begin{align}\label{c-func-wii}
\sum_{i=1}^{9} k_i \cdot \mathcal{L}^{\varphi_2}_{\bC}\left(f^{\an}_{i}(x,y);f_i^\an \left(\vec{v}\right),f_i^\an \left(\gamma_{0}\right)\right) = 0.
\end{align}
By substituting (\ref{cii2}) into (\ref{c-func-wii}) and applying (\ref{liCtoLi}), (\ref{c_assoc}), and (\ref{logbranch}), 
we reach the desired equations (a-$\bC$), (b-$\bC$), and (c-$\bC$) for each choice of $k_i$, respectively. 
This completes the proof of Theorem \ref{c-dilog}.

\renewcommand{\arraystretch}{1.8}
\tabcolsep = 0.2cm
\begin{table}[hbtp]
  \caption{Homotopy criterion for Schaeffer's equation}
  \label{table-sub1}
  \centering
  \begin{tabular}{|c||c|c|c|c|c|c|c|}
    \hline
    \# & $[\bar{B}_1,\bar{B}_4]$ & $[\bar{B}_2,\bar{B}_3]$ & $[\bar{B}_2,\bar{B}_5]$ & $[\bar{B}_2,\bar{B}_6]$ & $[\bar{B}_3,\bar{B}_5]$ & $[\bar{B}_3,\bar{B}_6]$ & $[\bar{B}_5,\bar{B}_6]$  \\
    \hline \hline
    $\operatorname{gr}_\Gamma^2\left(\iota_{\delta_{4}}^{\an} \circ f^{\an}_{4\ast}\right)(\#)$  & $-[\bar{l}_0,\bar{l}_1]$ & $-[\bar{l}_0,\bar{l}_1]$  & $[\bar{l}_0,\bar{l}_1]$ & $0$ & $-[\bar{l}_0,\bar{l}_1]$ & $0$ & $0$  \\ \hline
    $\operatorname{gr}_\Gamma^2\left(\iota_{\delta_{9}}^{\an} \circ f^{\an}_{9\ast}\right)(\#)$  & $0$ & $0$  & $[\bar{l}_0,\bar{l}_1]$ & $0$ & $0$ & $0$ & $0$  \\ \hline
    $\operatorname{gr}_\Gamma^2\left(\iota_{\delta_{8}}^{\an} \circ f^{\an}_{8\ast}\right)(\#)$  & $[\bar{l}_0,\bar{l}_1]$ & $0$  & $0$ & $0$ & $0$ & $0$ & $0$  \\ \hline
    $\operatorname{gr}_\Gamma^2\left(\iota_{\delta_{3}}^{\an} \circ f^{\an}_{3\ast}\right)(\#)$  & $0$ & $-[\bar{l}_0,\bar{l}_1]$  &  $0$ & $0$ & $0$ & $0$ & $0$  \\ \hline
    $\operatorname{gr}_\Gamma^2\left(\iota_{\delta_{6}}^{\an} \circ f^{\an}_{6\ast}\right)(\#)$ & $0$ & $0$  &  $0$ & $0$ & $-[\bar{l}_0,\bar{l}_1]$ & $0$ & $0$  \\
    \hline
  \end{tabular}
\end{table}

\begin{table}[hbtp]
  \caption{Homotopy criterion for Kummer's equation}
  \label{table-sub2}
  \centering
  \begin{tabular}{|c||c|c|c|c|c|c|c|}
    \hline
    \# & $[\bar{B}_1,\bar{B}_4]$ & $[\bar{B}_2,\bar{B}_3]$ & $[\bar{B}_2,\bar{B}_5]$ & $[\bar{B}_2,\bar{B}_6]$ & $[\bar{B}_3,\bar{B}_5]$ & $[\bar{B}_3,\bar{B}_6]$ & $[\bar{B}_5,\bar{B}_6]$  \\
    \hline \hline
    $\operatorname{gr}_\Gamma^2\left(\iota_{\delta_{1}}^{\an} \circ f^{\an}_{1\ast}\right)(\#)$  & $-2[\bar{l}_0,\bar{l}_1]$ & $-[\bar{l}_0,\bar{l}_1]$  &  $2[\bar{l}_0,\bar{l}_1]$ & $-[\bar{l}_0,\bar{l}_1]$ & $-2[\bar{l}_0,\bar{l}_1]$ & $0$ & $2[\bar{l}_0,\bar{l}_1]$  \\ \hline
    $\operatorname{gr}_\Gamma^2\left(\iota_{\delta_{5}}^{\an} \circ f^{\an}_{5\ast}\right)(\#)$  & $-[\bar{l}_0,\bar{l}_1]$ & $0$  & $0$ & $0$ & $0$ & $0$ & $[\bar{l}_0,\bar{l}_1]$  \\ \hline
    $\operatorname{gr}_\Gamma^2\left(\iota_{\delta_{7}}^{\an} \circ f^{\an}_{7\ast}\right)(\#)$  & $0$ & $0$  & $[\bar{l}_0,\bar{l}_1]$ & $-[\bar{l}_0,\bar{l}_1]$ & $0$ & $0$ & $[\bar{l}_0,\bar{l}_1]$  \\ \hline
    $\operatorname{gr}_\Gamma^2\left(\iota_{\delta_{4}}^{\an} \circ f^{\an}_{4\ast}\right)(\#)$  & $-[\bar{l}_0,\bar{l}_1]$ & $-[\bar{l}_0,\bar{l}_1]$  & $[\bar{l}_0,\bar{l}_1]$ & $0$ & $-[\bar{l}_0,\bar{l}_1]$ & $0$ & $0$  \\ \hline
    $\operatorname{gr}_\Gamma^2\left(\iota_{\delta_{6}}^{\an} \circ f^{\an}_{6\ast}\right)(\#)$ & $0$ & $0$  &  $0$ & $0$ & $-[\bar{l}_0,\bar{l}_1]$ & $0$ & $0$  \\
    \hline
  \end{tabular}
\end{table}

\begin{table}[hbtp]
  \caption{Homotopy criterion for Hill's equation}
  \label{table-sub3}
  \centering
  \begin{tabular}{|c||c|c|c|c|c|c|c|}
    \hline
    \# & $[\bar{B}_1,\bar{B}_4]$ & $[\bar{B}_2,\bar{B}_3]$ & $[\bar{B}_2,\bar{B}_5]$ & $[\bar{B}_2,\bar{B}_6]$ & $[\bar{B}_3,\bar{B}_5]$ & $[\bar{B}_3,\bar{B}_6]$ & $[\bar{B}_5,\bar{B}_6]$  \\
    \hline \hline
    $\operatorname{gr}_\Gamma^2\left(\iota_{\delta_{7}}^{\an} \circ f^{\an}_{7\ast}\right)(\#)$  & $0$ & $0$  & $[\bar{l}_0,\bar{l}_1]$ & $-[\bar{l}_0,\bar{l}_1]$ & $0$ & $0$ & $[\bar{l}_0,\bar{l}_1]$  \\ \hline
    $\operatorname{gr}_\Gamma^2\left(\iota_{\delta_{2}}^{\an} \circ f^{\an}_{2\ast}\right)(\#)$  & $0$ & $0$  & $0$ & $[\bar{l}_0,\bar{l}_1]$ & $0$ & $0$ & $0$  \\ \hline
    $\operatorname{gr}_\Gamma^2\left(\iota_{\delta_{8}}^{\an} \circ f^{\an}_{8\ast}\right)(\#)$  & $[\bar{l}_0,\bar{l}_1]$ & $0$  & $0$ & $0$ & $0$ & $0$ & $0$  \\ \hline
    $\operatorname{gr}_\Gamma^2\left(\iota_{\delta_{9}}^{\an} \circ f^{\an}_{9\ast}\right)(\#)$  & $0$ & $0$  & $[\bar{l}_0,\bar{l}_1]$ & $0$ & $0$ & $0$ & $0$  \\ \hline
    $\operatorname{gr}_\Gamma^2\left(\iota_{\delta_{5}}^{\an} \circ f^{\an}_{5\ast}\right)(\#)$  & $-[\bar{l}_0,\bar{l}_1]$ & $0$  & $0$ & $0$ & $0$ & $0$ & $[\bar{l}_0,\bar{l}_1]$  \\
    \hline
  \end{tabular}
\end{table}
\end{proof}

\begin{proof}[Proof of Theorem \ref{main1}]
First, by making use of the data presented in TABLE \ref{table-main}, we can verify the homotopy criterion \cite[Theorem~5.7, $({\rm i})_{\bC}$]{NW12}:
\begin{align}\label{c-hom}
\sum_{i=1}^{9} d_i \cdot \varphi_3\left(\operatorname{gr}_{\Gamma}^{3}\left( \iota_{\delta_i}^{\an} \circ f_{i\ast}^\an \right) \right) = 0
\quad \text{in } \operatorname{Hom}_{\bZ}\left(\operatorname{gr}_{\Gamma}^{3}\left(\pi_1^{\rm top}\left(V_{\rm non\text{-}Fano}^{\an},\vec{v}\right)\right) ,\bZ \right),
\end{align}
where $d_1,d_2,\ldots,d_9$ are the coefficients shown in TABLE \ref{table-abcd}. Next, explicit calculation confirms that the tensor criterion \cite[Theorem~5.7, $({\rm ii})_{\bC}$]{NW12}
\begin{align}\label{c-ten}
\sum_{i=1}^{9} d_i \cdot \left(f_i^{\an} \otimes \left( f_i^{\an} \wedge \left(f_i^{\an}-1\right) \right) \right) = 0
\quad \text{in } \left(\left(\mathcal{O}^{\an}\right)^{\times}/\bC^{\times}\right) \otimes \left(\left(\left(\mathcal{O}^{\an}\right)^{\times}/\bC^{\times}\right) \wedge \left(\left(\mathcal{O}^{\an}\right)^{\times}/\bC^{\times}\right) \right)
\end{align}
holds. Therefore, we obtain the functional equation \cite[Theorem~5.7, $({\rm iii})_{\bC}$]{NW12}:
\begin{align}\label{c-func-wii3}
\sum_{i=1}^{9} d_i \cdot \mathcal{L}^{\varphi_3}_{\bC}\left(f^{\an}_{i}(x,y);f_i^\an \left(\vec{v}\right),f_i^\an \left(\gamma_{0}\right)\right) = 0.
\end{align}
By substituting (\ref{cii3}) into (\ref{c-func-wii3}) and applying (\ref{liCtoLi}), (\ref{c_assoc}), and (\ref{logbranch}), we obtain the desired equation (d-$\bC$). During this elimination process, the nine ${Li}_2$ terms appearing on the left-hand side of (\ref{c-func-wii3}) cancel out completely via the relations (a-$\bC$), (b-$\bC$), and (c-$\bC$). This completes the proof of Theorem \ref{main1}.
\end{proof}

\subsection{$\ell$-adic Galois case}
In this subsection,
we derive the Spence--Kummer equation for the $\ell$-adic Galois trilogarithm.
To this end, we also establish the Schaeffer, Kummer, and Hill equations for the $\ell$-adic Galois dilogarithm.
These functional equations are derived utilizing Nakamura--Wojtkowiak's tensor and homotopy criteria \cite[Theorem~5.7,~Proposition 5.11]{NW12}.
We fix a $K$-rational point $(x,y)$
of $V_{\rm non\text{-}Fano}$
and take the $K$-rational tangential base point $\vec{v}$ of $V_{\rm non\text{-}Fano}$ as in (\ref{tanbv}).

\begin{figure}[htbp]
\centering
\caption{Structure of the proofs for main results ($\ell$-adic Galois case).}
\label{fig:proof_structure}
\vspace{0.8cm}
\begin{tikzpicture}[
node distance = 0.7cm and 0.5cm,
box/.style = {draw, rectangle, rounded corners, minimum width=4cm, minimum height=1cm,
              align=center, fill=blue!5, inner sep=2ex, font=\small},
plain/.style = {align=center, font=\small},
arrow/.style = {-{Latex}, thick},
swatch/.style = {draw, minimum width=0.45cm, minimum height=0.45cm, inner sep=0pt}
]

\node (path) [box, fill=green!15, text width=6cm]
  {\textbf{Path composition} (\ref{gammai})
};

\node (chain) [box, below=0.7cm of path, fill=yellow!25, text width=6cm]
  {\textbf{Chain rule} (\ref{l-assoc})};

\node (eq428) [box, below=0.7cm of chain, fill=red!15, text width=9.5cm]
  {\textbf{General functional equations for\\ $\ell$-adic iterated integrals (\ref{l-func-dilog}), (\ref{l_ii3})}\\[0.5ex]
   {\footnotesize by the tensor-homotopy criteria \cite{NW12}}\\
   (RHS$\,{=}\,$non-trivial; error terms)};

\coordinate (split) at ([yshift=-0.6cm]eq428.south);

\node (lhs_route) [plain, below left=0.5cm and 0.6cm of split, text width=5.5cm]
  {\textbf{[ LHS computation route ]}\\[0.5ex]
   Polylog--BCH formula \cite[Proposition 5.9]{NW12} + Computations of Ihara associators (\ref{LHSofGalBasic33})};

\node (rhs_route) [plain, below right=0.5cm and 0.6cm of split, text width=5.5cm]
  {\textbf{[ RHS (error term) and Lower-degree term route ]}\\[0.5ex]
   Table~\ref{table-main} + BCH formula + Computations of Ihara associators (\ref{LHSofGalBasic33})};

\node (eq419) [box, below=0.5cm of lhs_route, text width=5.5cm]
  {Computations of $\ell$-adic iterated integrals (\ref{LHSofGalBasic}), (\ref{LHSofGalBasic2})};

\node (table12) [box, below=0.5cm of rhs_route, text width=5.5cm]
  {Computations of lower-degree terms in ${\bf log}({\mathfrak f}_{\sigma}^{-1})$: Table~\ref{table-assoc}};

\node (eq421) [box, below=0.6cm of table12, text width=5.5cm]
  {Explicit coefficients of lower-degree terms in ${\bf log}({\mathfrak f}_{\sigma}^{-1})$:
(\ref{rhobranch2})};

\node (sub_rhs) [box, below=0.6cm of eq421, text width=5.5cm]
  {Substitution into RHS of (\ref{l-func-dilog}), (\ref{l_ii3})};

\node (sub_lhs) [box, text width=5.5cm] at (eq419 |- sub_rhs)
  {Substitution into LHS of (\ref{l-func-dilog}), (\ref{l_ii3})};

\coordinate (join) at ([yshift=-1cm]$(sub_lhs)!0.5!(sub_rhs)$);

\node (sub_conv) [box, below=0.9cm of join, text width=10cm]
  {Apply conversion formulas (\ref{li-GaltoLi}), (\ref{LHSofGalBasic3}) and functional equations for $\ell$-adic Kummer $1$-cocycles (\ref{rhobranch})};


\node (thm02) [box, fill=orange!20, below=0.7cm of sub_conv, text width=10cm]
  {\textbf{Main theorems: Theorem \ref{l-dilog}, Theorem \ref{main2}}\\ (Explicit functional equations for $\ell$-adic Galois polylogarithms)};

\draw [arrow] (path) -- (chain);
\draw [arrow] (chain) -- (eq428);

\draw [thick] (eq428.south) -- (split);
\draw [thick] (split) -| (lhs_route);
\draw [thick] (split) -| (rhs_route);

\draw [arrow] (lhs_route) -- (eq419);
\draw [arrow] (rhs_route) -- (table12);
\draw [arrow] (table12) -- (eq421);

\draw [arrow] (eq419) -- (sub_lhs);
\draw [arrow] (eq421) -- (sub_rhs);

\draw [thick] (sub_lhs.south) |- (join);
\draw [thick] (sub_rhs.south) |- (join);
\draw [arrow] (join) -- (sub_conv);

\draw [arrow] (sub_conv) -- (thm02);

\node (legtitle) [below=1.4cm of thm02] {\textbf{}};

\node (sw1) [swatch, fill=green!15, below=-1.0cm of legtitle, xshift=-2cm] {};
\node (tx1) [right=0.25cm of sw1, anchor=west, font=\small] {Geometric input};

\node (sw2) [swatch, fill=yellow!25, below=0.3cm of sw1] {};
\node (tx2) [right=0.25cm of sw2, anchor=west, font=\small] {Key identity};

\node (sw3) [swatch, fill=red!15, below=0.3cm of sw2] {};
\node (tx3) [right=0.25cm of sw3, anchor=west, font=\small] {Result from \cite{NW12}};

\node (sw4) [swatch, fill=orange!20, below=0.3cm of sw3] {};
\node (tx4) [right=0.25cm of sw4, anchor=west, font=\small] {Final result};

\node (sw5) [swatch, fill=blue!5, below=0.3cm of sw4] {};
\node (tx5) [right=0.25cm of sw5, anchor=west, font=\small] {Auxiliary computation};


\end{tikzpicture}
\end{figure}

First, we set up the verification of the tensor and homotopy criteria \cite[Theorem~5.7, $({\rm i})_{\ell}$,~$({\rm ii})_{\ell}$]{NW12}.
By the comparison map (\ref{comp}),
we regard $\pi_1^{{\ell\text{-\'et}}}\left(V_{\rm non\text{-}Fano},{\vec{v}}\right)$ as the pro-$\ell$ completion of $\pi_1^{\rm top}\left(V_{\rm non\text{-}Fano}^{\an},\vec{v}\right)$ with topological generators $B_j~(j=1,...,8)$ as in (\ref{fun-nf}):
\[
\pi_1^{{\ell\text{-\'et}}}\left(V_{\rm non\text{-}Fano},{\vec{v}}\right)
= \overline{\left\langle
\left. 
\begin{array}{l}
B_1, B_2, B_3, B_4, \\
B_5, B_6, B_7, B_8
\end{array} 
\right|
\begin{array}{l}
(R''1) \sim (R''12)
\end{array}
\right\rangle}.
\]
We write
\[
{f_{i\ast}}: \pi_1^{{\ell\text{-\'et}}}\left(V_{\rm non\text{-}Fano},{\vec{v}}\right) \to \pi_1^{{\ell\text{-\'et}}}\left({\bP}^{1}_{\overline{K}} \backslash \{0,1,\infty\}, f_i\left(\vec{v}\right)\right)
\]
for the homomorphism induced by $f_{i}$
where $f_i~(i=1,...,9)$ are as in (\ref{fi}).
We denote by
\[
\iota_{\delta_{i}}: \pi_1^{{\ell\text{-\'et}}}\left({\bP}^{1}_{\overline{K}} \backslash \{0,1,\infty\}, f_i\left(\vec{v}\right)\right) \xrightarrow{\simeq} \pi_1^{{\ell\text{-\'et}}}\left({\bP}^{1}_{\overline{K}} \backslash \{0,1,\infty\}, \overrightarrow{01}\right)
\]
the change-of-basepoint homomorphism defined by
$p \mapsto \delta_{i} \cdot p \cdot {\delta_{i}}^{-1}$
where $\delta_i~(i=1,...,9)$ are as in TABLE \ref{table_delta}.
Then,
$\left(\iota_{\delta_{i}} \circ {f_{i\ast}}\right)\left(B_j\right)$ are calculated as in the complex case (TABLE \ref{table-main}).
Moreover,
for $i=1,...,9$,
by considering the composite
$V_{\rm non\text{-}Fano} \xrightarrow{f_i} {\bP}^{1}_{\overline{K}} \backslash \{0,1,\infty\} \hookrightarrow \bG_{m,\overline{K}}$,
we may regard $f_i$ as an element of the unit group
\[
f_i \in {\CalO}^{\times},
\]
where ${\CalO}$ is the coordinate ring ${\overline{K}}\left[s_1,s_2,\frac{1}{s_1 s_2(1-s_1)(1-s_2)(s_1-s_2)(1-s_1 s_2)}\right]$ of $V_{\rm non\text{-}Fano}$.
Then,
\[
f_i-1 \in {\CalO}^{\times}
\]
holds.
Under the above preparations,
the tensor homotopy criteria \cite[Theorem~5.7, $({\rm i})_{\ell}$,~$({\rm ii})_{\ell}$]{NW12} are verified in the proof of Theorem \ref{l-dilog} and Theorem \ref{main2}.

Next, we proceed to calculate the functional equations of
$\ell$-adic iterated integrals \cite[Theorem~5.7, $({\rm iii})_{\ell}$]{NW12}.
Let $\sigma \in G_K$.
By (\ref{explicit2}),
for $i=1,...,9$,
we have
\begin{equation} 
\label{li-GaltoLi}
\begin{dcases}
\ell i_0\left(f^{}_{i}(x,y);\gamma_i,\sigma\right)=&\rho_{f^{}_{i}(x,y),\gamma_i}(\sigma), \\
\ell i_1\left(f^{}_{i}(x,y);\gamma_i,\sigma\right)=&\rho_{1-f^{}_{i}(x,y),\gamma'_i}(\sigma), \\
\ell i_2\left(f^{}_{i}(x,y);\gamma_i,\sigma\right)=&{Li}_2^{\ell}\left(f^{}_{i}(x,y);\gamma_i,\sigma\right)+\frac12\rho_{f^{}_{i}(x,y),\gamma_i}(\sigma)\rho_{1-f^{}_{i}(x,y),\gamma'_i}(\sigma), \\
\ell i_3\left(f^{}_{i}(x,y);\gamma_i,\sigma\right)=&{Li}_3^{\ell}\left(f^{}_{i}(x,y);\gamma_i,\sigma\right)
+\frac12 \rho_{f^{}_{i}(x,y),\gamma_i}(\sigma) {Li}_2^{\ell}\left(f^{}_{i}(x,y);\gamma_i,\sigma\right)\\
&+ \frac{1}{12}\left(\rho_{f^{}_{i}(x,y),\gamma_i}(\sigma)\right)^2\rho_{1-f^{}_{i}(x,y),\gamma'_i}(\sigma),
\end{dcases}
\end{equation}
where $\gamma'_i \in \pi_1^{\rm top}\left(\mathbb{P}^1(\mathbb{C}) \backslash \{0,1,\infty\}; \overrightarrow{01}, f^{\an}_{i}(x,y)\right)$ is the path associated
with $\gamma_i$, as in (\ref{gamma'}).
By (\ref{llidef}),
we obtain the following relations of $\ell$-adic Ihara associators (cf. \cite[p.106]{I90}):
\begin{equation}
\label{l-assoc-rel}
\begin{dcases}
&{{\mathfrak f}^{f_5 \left( \vec{v} \right),\delta_5}_{\overrightarrow{01},\sigma}}(l_0,l_1)=l_0^{\frac{1-\chi(\sigma)}{2}}, \\
&{{\mathfrak f}^{f_7 \left( \vec{v} \right),\delta_7}_{\overrightarrow{01},\sigma}}(l_0,l_1)={{\mathfrak f}^{f_6 \left( \vec{v} \right),\delta_6}_{\overrightarrow{01},\sigma}}(l_0,l_{\infty}) \cdot {{\mathfrak f}^{f_5 \left( \vec{v} \right),\delta_5}_{\overrightarrow{01},\sigma}}(l_0,l_1),
\end{dcases}
\end{equation}
where ${\mathfrak f}(l_\ast,l_{\ast'})$ is the image of ${\mathfrak f}$ under the map $\pi_1^{{\ell\text{-\'et}}}\left({\bP}^{1}_{\overline{K}} \backslash \{0,1,\infty\}, \overrightarrow{01}\right) \to \pi_1^{{\ell\text{-\'et}}}\left({\bP}^{1}_{\overline{K}} \backslash \{0,1,\infty\}, \overrightarrow{01}\right)$ given by $l_0,l_1 \mapsto l_\ast,l_{\ast'}$,
and $\chi: G_K \to \bZ_{\ell}^{\times}$ is the $\ell$-adic cyclotomic character.
Then,
we obtain
\begin{align}
\label{LHSofGalBasic33}
\begin{cases}
{\bf log} \left( \left( {{\mathfrak f}^{f_i \left( \vec{v} \right),\delta_i}_{\overrightarrow{01},\sigma}}(l_0,l_1)\right)^{-1}\right)=&0
~\quad~(i=1,2,3,4,8,9), \\
{\bf log} \left( \left( {{\mathfrak f}^{\ovec{0\infty},\delta_5}_{\overrightarrow{01},\sigma}}(l_0,l_1)\right)^{-1}\right)
=&\left(\frac{\chi(\sigma)-1}{2}\right)X, \\
{\bf log} \left( \left( {{\mathfrak f}^{\ovec{10},\delta_6}_{\overrightarrow{01},\sigma}}(l_0,l_1)\right)^{-1}\right)
=&\ell i_{2}\left({\overrightarrow{10};\delta_6,\sigma}\right) [X,Y]+\ell i_{3}\left({\overrightarrow{10};\delta_6,\sigma}\right) [X,[X,Y]]+\cdots,\\
{\bf log} \left( \left( {{\mathfrak f}^{\ovec{\infty0},\delta_7}_{\overrightarrow{01},\sigma}}(l_0,l_1)\right)^{-1}\right)
=&\left(\frac{\chi(\sigma)-1}{2}\right)X-\ell i_{2}\left({\overrightarrow{10};\delta_6,\sigma}\right) [X,Y]\\
&-\frac{1}{2}\left(\frac{\chi(\sigma)+1}{2}\right)\ell i_{2}\left({\overrightarrow{10};\delta_6,\sigma}\right)[X,[X,Y]]+\cdots,
\end{cases}
\end{align}
in $\bQ_{\ell}\langle \langle X,
Y \rangle \rangle$,
that is,
\begin{align}
\label{LHSofGalBasic3}
\begin{cases}
&\left(-\ell i_j(\ovec{01};\delta_i,\sigma)\right)_{0\le j \le 3}=(0,0,0,0)~\quad~(i=1,2,3,4,8,9), \\
&\left(-\ell i_j(\ovec{0\infty};\delta_5,\sigma)\right)_{0\le j \le 3}
=\left(\frac{1-\chi(\sigma)}{2},0,0,0\right), \\
&\left(-\ell i_j(\ovec{10};\delta_6,\sigma)\right)_{0\le j\le 3}
=\left(0,0,-\ell i_2\left(\ovec{10};\delta_6,\sigma\right),-\ell i_3\left(\ovec{10};\delta_6,\sigma\right)\right),\\
&\left(-\ell i_j(\ovec{\infty0};\delta_7,\sigma)\right)_{0\le j\le 3}=\left(
\frac{1-\chi(\sigma)}{2},0,\ell i_2\left(\ovec{10};\delta_6,\sigma\right),\frac{1}{2}\left(\frac{1+\chi(\sigma)}{2}\right)\ell i_2\left(\ovec{10};\delta_6,\sigma\right)\right).
\end{cases}
\end{align}

\begin{rem}\label{imprem1}
The second formula in (\ref{l-assoc-rel}) marks the crucial point where the non-linearity of the Baker–Campbell–Hausdorff sum
\begin{align*}
Z={\bf log}(l_\infty)={\bf log}({\bf exp}(-Y){\bf exp}(-X))
\end{align*}
mentioned in Remark \ref{Zl-integral},
first enters the computation in this paper.
Specifically,
the substitution
$Y \mapsto Z = -Y - X - \frac{1}{2}[X,Y]+\cdots$
in the expression for ${\bf log}\left(\left({{\mathfrak f}^{f_6 \left( \vec{v} \right),\delta_6}_{\overrightarrow{01},\sigma}}(l_0,l_{1})\right)^{-1}\right)$
yields
\begin{align*}
{\bf log}\left(\left({{\mathfrak f}^{f_6 \left( \vec{v} \right),\delta_6}_{\overrightarrow{01},\sigma}}(l_0,l_{\infty})\right)^{-1}\right)
& = \ell i_{2}\left({\overrightarrow{10};\delta_6,\sigma}\right) [X,Z]+\cdots \\
& = -\ell i_{2}\left({\overrightarrow{10};\delta_6,\sigma}\right)[X,Y]
- \frac{1}{2}\ell i_{2}\left({\overrightarrow{10};\delta_6,\sigma}\right)[X,[X,Y]]+\cdots,
\end{align*}
and a further application of the BCH formula to the product in the second formula of (\ref{l-assoc-rel}) then produces the non-trivial term
\[
-\frac{1}{2}\left(\frac{\chi(\sigma)+1}{2}\right)\ell i_{2}\left({\overrightarrow{10};\delta_6,\sigma}\right)[X,[X,Y]]
\]
in
the fourth case of (\ref{LHSofGalBasic33}).
In contrast,
in the complex setting,
the analogous computation uses the linear relation
$Z= -Y - X$,
so no such higher-order term appears and the corresponding value in the fourth case of (\ref{c_assoc}) is simply $0$,
as the higher-order commutator terms are completely absent.
\end{rem}

For $i=1,\ldots,9$, 
we consider the path $\gamma_i \in \pi_1^{\rm top}\left(\mathbb{P}^1(\mathbb{C}) \backslash \{0,1,\infty\}; \overrightarrow{01}, f^{\an}_{i}(x,y)\right)$ in (\ref{gammai}) as a pro-$\ell$ \'etale path $\gamma_i \in \pi_1^{\ell \text{-\'et}}\left(\mathbb{P}^1_{\overline{K}} \backslash \{0,1,\infty\}; \overrightarrow{01}, f^{}_{i}(x,y)\right)$ by the comparison map (\ref{comp}).
Then,
we have~$\gamma_i=\delta_i \cdot f_i\left(\gamma_0\right) \in \pi_1^\ellet\left(\mathbb{P}^1_{\overline{K}} \backslash \{0,1,\infty\}; \overrightarrow{01}, f^{}_{i}(x,y)\right)$.
By this path composition and the definition \ref{f},
we obtain an algebraic relation
\begin{align}\label{l-assoc}
{\mathfrak f}^{f^{}_{i}(x,y),\gamma_i}_{\overrightarrow{01},\sigma}
=\left(\left(\iota_{\delta_{i}} \circ f_{i\ast}\right)\left( {{\mathfrak f}^{(x,y),\gamma_0}_{\vec{v},\sigma}} \right)\right) \cdot {{\mathfrak f}^{f_i \left( \vec{v} \right),\delta_i}_{\overrightarrow{01},\sigma}} \quad (i=1,...,9).
\end{align}
By (\ref{l-assoc}), (\ref{LHSofGalBasic3}) and the polylog-BCH formula \cite[Proposition 5.9]{NW12},
we compute Wojtkowiak's native $\ell$-adic iterated integral (cf. \cite[Definition 4.7]{NW12})
\begin{align*}
\mathcal{L}^{\varphi_{k}(f_i)_{\vec{l}}}_{\rm nv}\left(f^{}_{i}(x,y),f_i\left(\vec{v}\right);f_i \left(\gamma_{0}\right),\sigma\right)
:=\varphi_{k,\vec{l}} \left({\boldsymbol \log}\left(\left(\iota_{\delta_{i}} \circ f_{i\ast}\right) \left({{\mathfrak f}^{(x,y),\gamma_0}_{\vec{v},\sigma}}\right)^{-1}\right)\right)
\quad (i=1,...,9)
\end{align*}
associated with
$f_i \left(\gamma_{0}\right)$
as follows:
\begin{align}
\label{LHSofGalBasic}
\begin{cases}
\mathcal{L}^{\varphi_2(f_i)_{\vec{l}}}_{\rm nv}\left(f^{}_{i}(x,y),f_i\left(\vec{v}\right);f_i \left(\gamma_{0}\right),\sigma\right) 
&={\ell i}_2(f^{}_{i}(x,y);\gamma_i,\sigma)~\quad~(i=1,2,3,4,8,9),
 \\
\mathcal{L}^{\varphi_2(f_5)_{\vec{l}}}_{\rm nv}\left(f^{}_{5}(x,y),f_5\left(\vec{v}\right);f_5 \left(\gamma_{0}\right),\sigma\right) 
&= {\ell i}_2(f^{}_{5}(x,y);\gamma_5,\sigma)-\tfrac12 {\ell i}_0\left(\overrightarrow{0\infty};\delta_5,\sigma\right) {\ell i}_1(f^{}_{5}(x,y);\gamma_5,\sigma),
\\
\mathcal{L}^{\varphi_2(f_6)_{\vec{l}}}_{\rm nv}\left(f^{}_{6}(x,y),f_6\left(\vec{v}\right);f_6 \left(\gamma_{0}\right),\sigma\right) 
&= 
{\ell i}_2(f^{}_{6}(x,y);\gamma_6,\sigma)-{\ell i}_2\left(\ovec{10};\delta_6,\sigma\right)\\
&\quad +\tfrac12
{\ell i}_0(f^{}_{6}(x,y);\gamma_6,\sigma)
{\ell i}_1\left(\ovec{10};\delta_6,\sigma\right),\\
\mathcal{L}^{\varphi_2(f_7)_{\vec{l}}}_{\rm nv}\left(f^{}_{7}(x,y),f_7\left(\vec{v}\right);f_7 \left(\gamma_{0}\right),\sigma\right) 
&= 
{\ell i}_2(f^{}_{7}(x,y);\gamma_7,\sigma)-{\ell i}_2\left(\ovec{\infty0};\delta_7,\sigma\right)\\
&\quad-\tfrac12
{\ell i}_1(f^{}_{7}(x,y);\gamma_7,\sigma)
{\ell i}_0\left(\ovec{\infty0};\delta_7,\sigma\right)\\
&\quad-\tfrac12
{\ell i}_0(f^{}_{7}(x,y);\gamma_7,\sigma)
{\ell i}_1\left(\ovec{\infty0};\delta_7,\sigma\right)
\end{cases}
\end{align}
and
\begin{align}
\label{LHSofGalBasic2}
\begin{cases}
\mathcal{L}^{\varphi_3(f_i)_{\vec{l}}}_{\rm nv}\left(f^{}_{i}(x,y),f_i\left(\vec{v}\right);f_i \left(\gamma_{0}\right),\sigma\right) 
&={\ell i}_3(f^{}_{i}(x,y);\gamma_i,\sigma)~\quad~(i=1,2,3,4,8,9),
 \\
\mathcal{L}^{\varphi_3(f_5)_{\vec{l}}}_{\rm nv}\left(f^{}_{5}(x,y),f_5\left(\vec{v}\right);f_5 \left(\gamma_{0}\right),\sigma\right) 
&= {\ell i}_3(f^{}_{5}(x,y);\gamma_5,\sigma)-\tfrac12 {\ell i}_0\left(\overrightarrow{0\infty};\delta_5,\sigma\right) {\ell i}_2(f^{}_{5}(x,y);\gamma_5,\sigma)
\\
&\quad+\tfrac{1}{12} \left({\ell i}_0\left(\overrightarrow{0\infty};\delta_5,\sigma\right)\right)^2 {\ell i}_1(f^{}_{5}(x,y);\gamma_5,\sigma)\\
&\quad+\tfrac{1}{12} {\ell i}_0\left(\overrightarrow{0\infty};\delta_5,\sigma\right) {\ell i}_0(f^{}_{5}(x,y);\gamma_5,\sigma) {\ell i}_1(f^{}_{5}(x,y);\gamma_5,\sigma),\\
\mathcal{L}^{\varphi_3(f_6)_{\vec{l}}}_{\rm nv}\left(f^{}_{6}(x,y),f_6\left(\vec{v}\right);f_6 \left(\gamma_{0}\right),\sigma\right) 
&= 
{\ell i}_3(f^{}_{6}(x,y);\gamma_6,\sigma)-{\ell i}_3\left(\ovec{10};\delta_6,\sigma\right)\\
&\quad +\tfrac12
{\ell i}_0(f^{}_{6}(x,y);\gamma_6,\sigma)
{\ell i}_2\left(\ovec{10};\delta_6,\sigma\right),\\
\mathcal{L}^{\varphi_3(f_7)_{\vec{l}}}_{\rm nv}\left(f^{}_{7}(x,y),f_7\left(\vec{v}\right);f_7 \left(\gamma_{0}\right),\sigma\right) 
&= 
{\ell i}_3(f^{}_{7}(x,y);\gamma_7,\sigma)-{\ell i}_3\left(\ovec{\infty0};\delta_7,\sigma\right)\\
&\quad -\tfrac12
{\ell i}_2(f^{}_{7}(x,y);\gamma_7,\sigma)
{\ell i}_0\left(\ovec{\infty0};\delta_7,\sigma\right)\\
&\quad-\tfrac12
{\ell i}_0(f^{}_{7}(x,y);\gamma_7,\sigma)
{\ell i}_2\left(\ovec{\infty0};\delta_7,\sigma\right)\\
&\quad+\tfrac{1}{12} \left({\ell i}_0\left(\overrightarrow{\infty0};\delta_7,\sigma\right)\right)^2 {\ell i}_1(f^{}_{7}(x,y);\gamma_7,\sigma)\\
&\quad+\tfrac{1}{12} {\ell i}_0\left(\overrightarrow{\infty0};\delta_7,\sigma\right) {\ell i}_0(f^{}_{7}(x,y);\gamma_7,\sigma) {\ell i}_1(f^{}_{7}(x,y);\gamma_7,\sigma).
\end{cases}
\end{align}

We write
\[
L\left( \pi_1^{{\ell\text{-\'et}}}\left(V_{\rm non\text{-}Fano},{\vec{v}}\right) \right)
\]
for the complete $\ell$-adic Lie algebra of $\pi_1^{{\ell\text{-\'et}}}\left(V_{\rm non\text{-}Fano},{\vec{v}}\right)$ over $\bQ_{\ell}$.
Then,
there is a natural inclusion
\[
\pi_1^{{\ell\text{-\'et}}}\left(V_{\rm non\text{-}Fano},{\vec{v}}\right) \hookrightarrow L\left( \pi_1^{{\ell\text{-\'et}}}\left(V_{\rm non\text{-}Fano},{\vec{v}}\right) \right),~B_i \mapsto X_i:={\bf log}(B_i).
\]
Each element of $L\left( \pi_1^{{\ell\text{-\'et}}}\left(V_{\rm non\text{-}Fano},{\vec{v}}\right) \right)$
has an expansion as a formal Lie series in 
$X_1,...,X_6$.
We denote by
\[
L_{n}\left(\pi_1^{{\ell\text{-\'et}}}\left(V_{\rm non\text{-}Fano},{\vec{v}}\right)\right)
\]
the part with homogeneity degree $n$
and denote by
\[
L_{<n}\left(\pi_1^{{\ell\text{-\'et}}}\left(V_{\rm non\text{-}Fano},{\vec{v}}\right)\right)
\]
the part whose homogeneity degree is less than $n$.
Then,
we have a decomposition
\[
L\left( \pi_1^{{\ell\text{-\'et}}}\left(V_{\rm non\text{-}Fano},{\vec{v}}\right) \right)=L_{<n}\left(\pi_1^{{\ell\text{-\'et}}}\left(V_{\rm non\text{-}Fano},{\vec{v}}\right)\right) \oplus \Gamma^nL\left( \pi_1^{{\ell\text{-\'et}}}\left(V_{\rm non\text{-}Fano},{\vec{v}}\right) \right)
\]
and
\[
{\boldsymbol \log} \left(\left({{\mathfrak f}^{(x,y),\gamma_0}_{\vec{v},\sigma}}\right)^{-1}\right)
=\left[{\boldsymbol \log} \left(\left({{\mathfrak f}^{(x,y),\gamma_0}_{\vec{v},\sigma}}\right)^{-1}\right)\right]_{<n}+\left[{\boldsymbol \log} \left(\left({{\mathfrak f}^{(x,y),\gamma_0}_{\vec{v},\sigma}}\right)^{-1}\right)\right]_{\geq n}.
\]
Here, 
there is an isomorphism
\[
L_{n}\left(\pi_1^{{\ell\text{-\'et}}}\left(V_{\rm non\text{-}Fano},{\vec{v}}\right)\right) \simeq {\rm gr}_{\Gamma}^{n}\left(\pi_1^{{\ell\text{-\'et}}}\left(V_{\rm non\text{-}Fano},{\vec{v}}\right)\right) \otimes \bQ_{\ell}
\]
induced by
$
L_{1}\left(\pi_1^{{\ell\text{-\'et}}}\left(V_{\rm non\text{-}Fano},{\vec{v}}\right)\right) \ni X_i \mapsto \bar{B}_i \in {\rm gr}_{\Gamma}^{1}\left(\pi_1^{{\ell\text{-\'et}}}\left(V_{\rm non\text{-}Fano},{\vec{v}}\right)\right)$.
Therefore,
we write
\begin{align*}
\left[{\boldsymbol \log} \left(\left({{\mathfrak f}^{(x,y),\gamma_0}_{\vec{v},\sigma}}\right)^{-1}\right)\right]_{<2}
&=C_1 X_1+C_2 X_2+C_3 X_3+C_4 X_4+C_5 X_5+C_6 X_6
\end{align*}
and
\begin{align*}
\left[{\boldsymbol \log} \left(\left({{\mathfrak f}^{(x,y),\gamma_0}_{\vec{v},\sigma}}\right)^{-1}\right)\right]_{<3}
&=C_1 X_1+C_2 X_2+C_3 X_3+C_4 X_4+C_5 X_5+C_6 X_6\\
&+C_7 [X_1,X_4]+C_8 [X_2,X_3]+C_9 [X_2,X_5]+C_{10} [X_2,X_6]\\
&+C_{11} [X_3,X_5]+C_{12} [X_3,X_6]+C_{13} [X_5,X_6].
\end{align*}
Using the
Baker--Campbell--Hausdorff formula, TABLE \ref{table-main}, (\ref{l-assoc}) and (\ref{LHSofGalBasic33}),
we compute
\begin{align}\label{diffcomp}
\left(\iota_{\delta_{i}} \circ f_{i\ast}\right)\left(\left[{\boldsymbol \log}\left(\left({{\mathfrak f}^{(x,y),\gamma_0}_{\vec{v},\sigma}}\right)^{-1}\right)\right]_{<n}\right) \in {\rm Lie}_{\bQ_{\ell}}\langle \langle X,
Y \rangle \rangle
\end{align}
for $n=2,3$.
Consequently,
we obtain TABLE \ref{table-assoc}.

\begin{rem}\label{imprem2}
The computation of (\ref{diffcomp}) represents the second instance in this paper where the non-linearity of the BCH sum
$Z={\bf log}({\bf exp}(-Y){\bf exp}(-X))$,
mentioned in Remark \ref{Zl-integral},
re-emerges as a critical factor, following the fourth case of (\ref{LHSofGalBasic33}).
When applying $\iota_{\delta_{i}} \circ f_{i\ast}$ to
\[
\left[{\boldsymbol \log}\left(\left({{\mathfrak f}^{(x,y),\gamma_0}_{\vec{v},\sigma}}\right)^{-1}\right)\right]_{<n},
\]
one must express the image of each generator $B_j$ in terms of $l_0$, $l_1$ and $l_\infty$ (cf. TABLE \ref{table-main}).
For all cases whose image involves $l_\infty$, one must substitute
\[
Z = -Y - X - \frac{1}{2}[X,Y] -\frac{1}{12}\left([X,[X,Y]]+[Y,[Y,X]]\right) + \cdots.
\] 
The higher-order terms
in this expansion---which vanish in the complex setting where the linearity $Z=-Y-X$ holds---produce the non-trivial $[X,Y]$- and $[X,[X,Y]]$-coefficients that appear in the $\ell$-adic Galois setting (cf. TABLE \ref{table-assoc}) but have no counterparts in the complex setting.
In particular,
the weight-2 contributions
\[
C_7,~C_8,~C_9,~C_{10},~C_{11},~C_{13}
\]
appearing in the $[X,Y]$-column of TABLE \ref{table-assoc} originate precisely from these higher-order BCH terms, 
and it is their propagation to weight 3 that ultimately gives rise to the $\ell$-adic error terms,
i.e.,
the right-hand side of (\ref{l-func-dilog}) and (\ref{l_ii3}).
In contrast,
in the complex side,
these corresponding weight-2 contributions trivially vanish,
which is exactly why the right-hand side of the functional equations (\ref{c-func-wii}) and (\ref{c-func-wii3}) ultimately becomes zero.
\end{rem}

By (\ref{l-assoc}), (\ref{LHSofGalBasic33}) and TABLE \ref{table-assoc},
we obtain
\begin{equation} \label{rhobranch2}
\begin{dcases}
&C_1=\rho_{x,\gamma_8}(\sigma), \quad \quad \quad
C_2=\rho_{y,\gamma_9}(\sigma), \quad \quad \quad
C_3=\rho_{1-\frac{x}{y},\gamma'_3}(\sigma)+\rho_{y,\gamma_9}(\sigma),\\
&C_4=\rho_{1-x,\gamma'_8}(\sigma),\quad \quad \quad
C_5=\rho_{1-y,\gamma'_9}(\sigma), \quad \quad \quad
C_6=\rho_{1-xy,\gamma'_2}(\sigma), \\
&C_7={\ell i}_2(x;\gamma_8,\sigma), \quad \quad \quad
C_8=\frac{1}{2}\rho_{y,\gamma_9}(\sigma)-{\ell i}_2\left(\frac{x}{y};\gamma_3,\sigma\right),\\
&C_9={\ell i}_2(y;\gamma_9,\sigma), \quad \quad
C_{10}={\ell i}_2(xy;\gamma_2,\sigma),\\
&C_{11}=-\frac{1}{2}\rho_{1-x,\gamma'_8}(\sigma)-{\ell i}_2\left(\frac{1-y}{1-x};\gamma_6,\sigma\right), \\
&C_{13}=\frac{1}{2}\rho_{1-x,\gamma'_8}(\sigma)-\rho_{1-xy,\gamma'_2}(\sigma)+{\ell i}_2(x;\gamma_8,\sigma)+{\ell i}_2\left(\frac{x(1-y)}{x-1};\gamma_5,\sigma\right),
\end{dcases}
\end{equation} 
and
\begin{equation} \label{rhobranch}
\begin{dcases}
\rho_{\frac{x(1-y)^2}{y(1-x)^2};\gamma_1}(\sigma)
&=
\rho_{x,\gamma_8}(\sigma)+2\rho_{1-y,\gamma'_9}(\sigma)-\rho_{y,\gamma_9}(\sigma)-2\rho_{1-x,\gamma'_8}(\sigma), \\
\rho_{1-\frac{x(1-y)^2}{y(1-x)^2};\gamma'_1}(\sigma)
&=
\rho_{1-\frac{x}{y},\gamma'_3}(\sigma)+\rho_{1-xy,\gamma'_2}(\sigma)-2\rho_{1-x,\gamma'_8}(\sigma),\\
\rho_{xy;\gamma_2}(\sigma)
&=
\rho_{x,\gamma_8}(\sigma)+\rho_{y,\gamma_9}(\sigma),\\
\rho_{\frac{x}{y};\gamma_3}(\sigma)
&=
\rho_{x,\gamma_8}(\sigma)-\rho_{y,\gamma_9}(\sigma),\\
\rho_{\frac{x(1-y)}{y(1-x)};\gamma_4}(\sigma)
&=
\rho_{x,\gamma_8}(\sigma)+\rho_{1-y,\gamma'_9}(\sigma)-\rho_{y,\gamma_9}(\sigma)-\rho_{1-x,\gamma'_8}(\sigma),\\
\rho_{1-\frac{x(1-y)}{y(1-x)};\gamma'_4}
&=
\rho_{1-\frac{x}{y},\gamma'_3}(\sigma)-\rho_{1-x,\gamma'_8}(\sigma),\\
\rho_{\frac{x(1-y)}{x-1};\gamma_5}(\sigma)
&=
\rho_{x,\gamma_8}(\sigma)+\rho_{1-y,\gamma'_9}(\sigma)-\rho_{1-x,\gamma'_8}(\sigma)+\left(\frac{\chi(\sigma)-1}{2}\right),\\
\rho_{1-\frac{x(1-y)}{x-1};\gamma'_5}(\sigma)
&=
\rho_{1-xy,\gamma'_2}(\sigma)-\rho_{1-x,\gamma'_8}(\sigma),\\
\rho_{\frac{1-y}{1-x};\gamma_6}(\sigma)
&=
\rho_{1-y,\gamma'_9}(\sigma)-\rho_{1-x,\gamma'_8}(\sigma),\\
\rho_{\frac{1-y}{y(x-1)};\gamma_7}(\sigma)
&=
\rho_{1-y,\gamma'_9}(\sigma)-\rho_{y,\gamma_9}(\sigma)-\rho_{1-x,\gamma'_8}(\sigma)+\left(\frac{\chi(\sigma)-1}{2}\right),\\
\rho_{1-\frac{1-y}{y(x-1)};\gamma'_7}(\sigma)
&=
\rho_{1-xy,\gamma'_2}(\sigma)-\rho_{y,\gamma_9}(\sigma)-\rho_{1-x,\gamma'_8}(\sigma).
\end{dcases}
\end{equation} 
Note that $C_{12}$ remains undetermined.

\begin{rem}
We explain why $C_{12}$ remains undetermined in the above computation.
Recall that $C_{12}$ is the coefficient of $$[X_3, X_6] = [{\bf log}\left(B_3\right), {\bf log}\left(B_6\right)]$$
in $\left[{\bf log}\left({\mathfrak f}^{(x,y),\gamma_0}_{\vec{v},\sigma}\right)^{-1}\right]_{<3}$
where $B_3$ and $B_6$ are the meridians of the divisors $s_1 = s_2$ and
$1 = s_1 s_2$ on $V_{\mathrm{non\text{-}Fano}}$, respectively.
As one can verify from TABLE~\ref{table-assoc},
the coefficient $C_{12}$ does not appear in
\[
(\iota_{\delta_i} \circ f_{i*})\left(\left[{\bf log}\left({\mathfrak f}^{(x,y),\gamma_0}_{\vec{v},\sigma}\right)^{-1}\right]_{<3}\right)
\]
for any $i = 1, \ldots, 9$.
For $C_{12}$ to appear in TABLE~\ref{table-assoc},
the images $\mathrm{gr}^1_\Gamma(\iota_{\delta_i} \circ f_{i*})\left(\bar{B}_3\right)$ and $\mathrm{gr}^1_\Gamma(\iota_{\delta_i} \circ f_{i*})\left(\bar{B}_6\right)$ would need to be linearly independent in
\[
\mathrm{gr}^1_\Gamma\,\pi_1^{\mathrm{top}} \left(\mathbb{P}^1(\mathbb{C})\setminus\{0,1,\infty\},\overrightarrow{01}\right) \cong \mathbb{Z}\bar{l}_0 \oplus \mathbb{Z}\bar{l}_1
\]
for some $i$.
However, inspecting TABLE~\ref{table-main},
one sees that this independence always fails:
for $i \in \{2,5,7,8,9\}$ the image of $B_3$ is trivial,
for $i \in \{3,4,6\}$ the image of $B_6$ is trivial,
and for $i=1$ both images equal $\bar{l}_1$ in $\mathrm{gr}^1_\Gamma$,
so their Lie bracket vanishes.
Hence $C_{12}$ does not appear in any of the functional equations (\ref{l-func-dilog}) and (\ref{l_ii3}), 
and its value plays no role in the proofs of Theorem~\ref{l-dilog} and Theorem~\ref{main2}. 
Its undetermined status is therefore not a gap in the computation but rather an inherent structural feature:
the underlying nine morphisms $\{f_i\}_{i=1}^9$ governing the Spence--Kummer equation are fundamentally insensitive to the $[X_3, X_6]$-component.
\end{rem}

\newpage
\vspace*{\stretch{1}}
\renewcommand{\arraystretch}{2}
\tabcolsep = 0.2cm
\begin{table}[hbtp]
  \caption{Computation of ${\bf log}({\mathfrak f}_{\sigma}^{-1})$}
  \label{table-assoc}
  \centering
  \begin{tabular}{|c||c|c|}
    \hline
    $i$ & \tiny$\left(\iota_{\delta_{i}} \circ f_{i\ast}\right)\left(\left[{\boldsymbol \log}\left(\left({{\mathfrak f}^{(x,y),\gamma_0}_{\vec{v},\sigma}}\right)^{-1}\right)\right]_{<2}\right)$ & \tiny$ \left(\iota_{\delta_{i}} \circ f_{i\ast}\right)\left(\left[{\boldsymbol \log}\left(\left({{\mathfrak f}^{(x,y),\gamma_0}_{\vec{v},\sigma}}\right)^{-1}\right)\right]_{<3}\right)$  \\
    \hline \hline

    $1$ & 
\begin{tabular}{c}
\tiny$\left(C_1-C_2-2C_4+2C_5\right)X$\\
\tiny$+\left(-C_2+C_3-2C_4+C_6\right)Y$\\
\tiny$+\left(\frac{1}{2}C_2-C_4+C_6-2C_7-C_8\right)[X,Y]$\\
\tiny$+\cdots$ \\
\end{tabular}
 & 
\begin{tabular}{c}
\tiny$\left(C_1-C_2-2C_4+2C_5\right)X+\left(-C_2+C_3-2C_4+C_6\right)Y$\\
\tiny$+\left(\frac{1}{2}C_2-C_4+C_6-2C_7-C_8+2C_9-C_{10}-2C_{11}+2C_{13}\right)[X,Y]$ \\
\tiny$+\left(\frac{1}{12}C_2-\frac{1}{6}C_4+\frac{1}{2}C_6-C_7-C_9-C_{10}+2C_{13}\right)[X,[X,Y]]+\cdots$ \\
\end{tabular} \\ \hline

    $2$ & 
\begin{tabular}{c}
$\left(C_1+C_2\right)X+C_6 Y$
\end{tabular}
 & 
\begin{tabular}{c}
$\left(C_1+C_2\right)X+C_6 Y+C_{10}[X,Y]$
\end{tabular}  \\ \hline

    $3$ & 
\begin{tabular}{c}
\tiny$\left(C_1-C_2\right)X+\left(-C_2+C_3\right)Y$\\
\tiny$+\frac{1}{2}C_2[X,Y]+\cdots$ 
\end{tabular}
 & 
\begin{tabular}{c}
\tiny$\left(C_1-C_2\right)X+\left(-C_2+C_3\right)Y+\left(\frac{1}{2}C_2-C_8\right)[X,Y]$\\
\tiny$-\frac{1}{12}C_2[X,[X,Y]]+\cdots$ 
\end{tabular} \\ \hline

    $4$ & 
\begin{tabular}{c}
\tiny$\left(C_1-C_2-C_4+C_5\right)X$\\
\tiny$+\left(-C_2+C_3-C_4\right)Y$\\
\tiny$+\left(\frac{1}{2}C_2-\frac{1}{2}C_4-C_7\right)[X,Y]+\cdots$
\end{tabular}
 & 
\begin{tabular}{c}
\tiny$\left(C_1-C_2-C_4+C_5\right)X+\left(-C_2+C_3-C_4\right)Y$\\
\tiny$+\left(\frac{1}{2}C_2-\frac{1}{2}C_4-C_7-C_8+C_9-C_{11}\right)[X,Y]$\\
\tiny$+\left(\frac{1}{12}C_2-\frac{1}{12}C_4-\frac{1}{2}C_7-\frac{1}{2}C_9\right)[X,[X,Y]]+\cdots$
\end{tabular}  \\ \hline

    $5$ & 
\begin{tabular}{c}
\tiny$\left(C_1-C_4+C_5\right)X+\left(-C_4+C_6\right)Y$  \\
\tiny$+\left(-\frac{1}{2}C_4+C_6\right)[X,Y]+\cdots$ 
\end{tabular}
 & 
\begin{tabular}{c}
\tiny$\left(C_1-C_4+C_5\right)X+\left(-C_4+C_6\right)Y+\left(-\frac{1}{2}C_4+C_6-C_7+C_{13}\right)[X,Y]$  \\
\tiny$+\left(-\frac{1}{12}C_4+\frac{1}{2}C_6-\frac{1}{2}C_7+C_{13}\right)[X,[X,Y]]+\cdots$ 
\end{tabular} \\ \hline

    $6$ & 
\begin{tabular}{c}
\tiny$\left(-C_4+C_5\right)X+\left(C_3-C_4\right)Y$  \\
\tiny$-\frac{1}{2}C_4[X,Y]+\cdots$
\end{tabular}
& 
\begin{tabular}{c}
\tiny$\left(-C_4+C_5\right)X+\left(C_3-C_4\right)Y+\left(-\frac{1}{2}C_4-C_{11}\right)[X,Y]$  \\
\tiny$-\frac{1}{12}C_4[X,[X,Y]]+\cdots$
\end{tabular}  \\ \hline
   
    $7$ & 
\begin{tabular}{c}
\tiny$\left(-C_2-C_4+C_5\right)X$ \\
\tiny$+\left(-C_2-C_4+C_6\right)Y$\\
\tiny$+\left(-\frac{1}{2}C_2-\frac{1}{2}C_4+C_6\right)[X,Y]+\cdots$
  \end{tabular}
 & 
\begin{tabular}{c}
\tiny$\left(-C_2-C_4+C_5\right)X+\left(-C_2-C_4+C_6\right)Y$ \\
\tiny$+\left(-\frac{1}{2}C_2-\frac{1}{2}C_4+C_6+C_9-C_{10}+C_{13}\right)[X,Y]$\\
\tiny$+\left(-\frac{1}{12}C_2-\frac{1}{12}C_4+\frac{1}{2}C_6+\frac{1}{2}C_9-C_{10}+C_{13}\right)[X,[X,Y]]+\cdots$
  \end{tabular}
 \\ \hline

    $8$ & $C_1 X+C_4 Y$ & $C_1 X+C_4 Y+C_7 [X,Y]$  \\ \hline

    $9$ & $C_2 X+C_5 Y$ & $C_2 X+C_5 Y+C_9 [X,Y]$  \\ \hline
\end{tabular}
\end{table}
\vspace{\stretch{2}}
\pagebreak
\newpage

By combining these formulas obtained above,
we prove the functional equations of $\ell$-adic Galois polylogarithms.

\begin{thm}[Functional equations for $\ell$-adic Galois dilogarithms]\label{l-dilog}
Given a $K$-rational point
$(x,y) \in V_{\rm non\text{-}Fano}(K)$ and a path $\gamma_0 \in \pi_1^{\rm top}\left(V_{\rm non\text{-}Fano}^\an;{\vec{v}},(x,y)\right)$,
define the path system $\{\gamma_i\}_{i=1,\ldots,9}$ associated with $\gamma_0$ as in (\ref{gammai}).
For any $\sigma \in G_K$,
the following hold.
\begin{description}
   \item[(a-$\ell$)~$\ell$-adic Schaeffer equation]
\[
Li_{2}^{\ell}\left({\frac{x(1-y)}{y(1-x)};\gamma_4,\sigma}\right)
-Li_{2}^{\ell}\left(y;\gamma_9,\sigma\right)
+Li_{2}^{\ell}\left(x;\gamma_8,\sigma\right)
-Li_{2}^{\ell}\left(\frac{x}{y};\gamma_3,\sigma\right)
\]
\[
-Li_{2}^{\ell}\left({\frac{1-y}{1-x};\gamma_6,\sigma}\right)
=\rho_{y,\gamma_9}(\sigma)\rho_{\frac{1-y}{1-x},\gamma_6}(\sigma)-{\boldsymbol \zeta}_2^{\ell}(\sigma).
\]
   \item[(b-$\ell$)~$\ell$-adic Kummer equation]
\[
Li_{2}^{\ell}\left({\frac{x(1-y)^2}{y(1-x)^2};\gamma_1,\sigma}\right)
-Li_{2}^{\ell}\left({\frac{x(1-y)}{x-1};\gamma_5,\sigma}\right)
-Li_{2}^{\ell}\left({\frac{1-y}{y(x-1)};\gamma_7,\sigma}\right)
\]
\[
-Li_{2}^{\ell}\left({\frac{x(1-y)}{y(1-x)};\gamma_4,\sigma}\right)
-Li_{2}^{\ell}\left({\frac{1-y}{1-x};\gamma_6,\sigma}\right)
\]
\[
=\frac{1}{2}\left(\rho_{y,\gamma_9}(\sigma)\right)^2+\frac{1}{2}\rho_{y,\gamma_9}(\sigma)+\rho_{1-x,\gamma'_8}(\sigma)-\rho_{1-xy,\gamma'_2}(\sigma).
\]
   \item[(c-$\ell$)~$\ell$-adic Hill equation]
\[
Li_{2}^{\ell}\left({\frac{1-y}{y(x-1)};\gamma_7,\sigma}\right)
+Li_{2}^{\ell}\left({xy;\gamma_2,\sigma}\right)
-Li_{2}^{\ell}\left({x;\gamma_8,\sigma}\right)
\]
\[
-Li_{2}^{\ell}\left({y;\gamma_9,\sigma}\right)
-Li_{2}^{\ell}\left({\frac{x(1-y)}{x-1};\gamma_5,\sigma}\right)
\]
\[
=
-{\boldsymbol \zeta}_2^{\ell}(\sigma)
+\rho_{y,\gamma_9}(\sigma)\rho_{\frac{1-y}{1-x},\gamma_6}(\sigma)
-\frac{1}{2}\left(\rho_{y,\gamma_9}(\sigma)\right)^2-\frac{1}{2}\rho_{y,\gamma_9}(\sigma).
\]
\end{description}
\end{thm}

\begin{proof}
Let $\sigma \in G_K$ and
$k_i \in \{a_i,b_i,c_i\}$, where $a_i,b_i,c_i~(i=1,\ldots,9)$ are shown in TABLE \ref{table-abcd}.
As in the complex case,
the homotopy criterion
\cite[Theorem~5.7~$({\rm i})_{\ell}$]{NW12}
\[
\sum_{i=1}^{9} k_i \cdot \varphi_2\left({\rm gr}_{\Gamma}^{2}\left( \iota_{\delta_i} \circ f_{i\ast} \right) \right) =0~\quad~{\rm in}~\quad~{\rm Hom}_{\bZ_\ell}\left({\rm gr}_{\Gamma}^{2}\left(\pi_1^{\ell \text{\rm -\'et}}\left(V_{\rm non\text{-}Fano}\right),\vec{v}\right),\bZ_{\ell} \right),
\]
and the tensor criterion
\cite[Theorem~5.7~$({\rm ii})_{\ell}$]{NW12}
\[
\sum_{i=1}^{9} k_i \cdot \left( f_i \wedge \left(f_i-1\right) \right) =0~\quad~{\rm in}~\quad~
\left({\CalO}^{\times}/\overline{K}^{\times}\right) \wedge \left({\CalO}^{\times}/\overline{K}^{\times}\right)
\]
hold.
Therefore,
we have the functional equation
\cite[Theorem~5.7~$({\rm iii})_{\ell}$,~Corollary 5.8]{NW12}
\begin{align}\label{l-func-dilog}
\sum_{i=1}^{9} k_i \cdot \mathcal{L}^{\varphi_{2}(f_i)_{\vec{l}}}_{\rm nv}\left(f^{}_{i}(x,y),f_i\left(\vec{v}\right);f_i\left(\gamma_{0}\right),\sigma\right)=\sum_{i=1}^{9} k_i \cdot \varphi_{2,\vec{l}}\left(\left(\iota_{\delta_{i}} \circ f_{i\ast}\right)\left(\left[{\boldsymbol \log}\left(\left({{\mathfrak f}^{(x,y),\gamma_0}_{\vec{v},\sigma}}\right)^{-1}\right)\right]_{<2}\right)\right).
\end{align}
The right-hand side of this equation is the ``$\ell$-adic error term'' referred to in \cite[Subsection 4.3]{NW12} which corresponds to lower weight terms in the functional equation.

By TABLE \ref{table-assoc} and (\ref{rhobranch2}),
the right-hand side of (\ref{l-func-dilog}) equals
\begin{equation} 
\label{compute1}
\begin{dcases}
&0 \qquad (\text{if}~k_i=a_i), \\
&\frac{1}{2}\rho_{y,\gamma_9}(\sigma)+\rho_{1-x,\gamma'_8}(\sigma)-\rho_{1-xy,\gamma'_2}(\sigma) \qquad (\text{if}~k_i=b_i), \\
&-\frac{1}{2}\rho_{y,\gamma_9}(\sigma) \qquad (\text{if}~k_i=c_i).
\end{dcases}
\end{equation}
By (\ref{li-GaltoLi}), (\ref{LHSofGalBasic}), (\ref{LHSofGalBasic3}) and (\ref{rhobranch}),
the left-hand side of (\ref{l-func-dilog}) equals 
\begin{equation} 
\label{compute2}
\begin{dcases}
&Li_{2}^{\ell}\left({\frac{x(1-y)}{y(1-x)};\gamma_4,\sigma}\right)
-Li_{2}^{\ell}\left(y;\gamma_9,\sigma\right)
+Li_{2}^{\ell}\left(x;\gamma_8,\sigma\right)
-Li_{2}^{\ell}\left(\frac{x}{y};\gamma_3,\sigma\right)\\
& \quad -Li_{2}^{\ell}\left({\frac{1-y}{1-x};\gamma_6,\sigma}\right)
-\rho_{y,\gamma_9}(\sigma)\rho_{\frac{1-y}{1-x},\gamma_6}(\sigma)+{\boldsymbol \zeta}_2^{\ell}(\sigma) \qquad (\text{if}~k_i=a_i), \\
& \\
&Li_{2}^{\ell}\left({\frac{x(1-y)^2}{y(1-x)^2};\gamma_1,\sigma}\right)
-Li_{2}^{\ell}\left({\frac{x(1-y)}{x-1};\gamma_5,\sigma}\right)
-Li_{2}^{\ell}\left({\frac{1-y}{y(x-1)};\gamma_7,\sigma}\right)\\
&-Li_{2}^{\ell}\left({\frac{x(1-y)}{y(1-x)};\gamma_4,\sigma}\right)-Li_{2}^{\ell}\left({\frac{1-y}{1-x};\gamma_6,\sigma}\right)-\frac{1}{2}\left(\rho_{y,\gamma_9}(\sigma)\right)^2 \qquad (\text{if}~k_i=b_i), \\
& \\
&Li_{2}^{\ell}\left({\frac{1-y}{y(x-1)};\gamma_7,\sigma}\right)
+Li_{2}^{\ell}\left({xy;\gamma_2,\sigma}\right)
-Li_{2}^{\ell}\left({x;\gamma_8,\sigma}\right)
-Li_{2}^{\ell}\left({y;\gamma_9,\sigma}\right)\\
&-Li_{2}^{\ell}\left({\frac{x(1-y)}{x-1};\gamma_5,\sigma}\right)
+{\boldsymbol \zeta}_2^{\ell}(\sigma)
-\rho_{y,\gamma_9}(\sigma)\rho_{\frac{1-y}{1-x},\gamma_6}(\sigma)+\frac{1}{2}\left(\rho_{y,\gamma_9}(\sigma)\right)^2 \qquad (\text{if}~k_i=c_i).
\end{dcases}
\end{equation}
By substituting (\ref{compute1}) and (\ref{compute2}) into (\ref{l-func-dilog}),
we obtain the desired equations
(a-$\ell$), (b-$\ell$) and (c-$\ell$).
This completes the proof of Theorem \ref{l-dilog}.
\end{proof}

\begin{proof}[Proof of Theorem \ref{main2}]
Here we prove (d-$\ell$).
Let $\sigma \in G_K$.
As in the complex case,
the tensor homotopy criteria \cite[Theorem~5.7, $({\rm i})_{\ell}$,~$({\rm ii})_{\ell}$]{NW12} hold:
\begin{align}\label{l-hom}
\sum_{i=1}^{9} d_i \cdot \varphi_3\left({\rm gr}_{\Gamma}^{3}\left( \iota_{\delta_i} \circ f_{i\ast} \right) \right) =0~\quad~{\rm in}~\quad~{\rm Hom}_{\bZ_\ell}\left({\rm gr}_{\Gamma}^{3}\left(\pi_1^{\ell \text{\rm -\'et}}\left(V_{\rm non\text{-}Fano}\right),\vec{v}\right) ,\bZ_{\ell} \right),
\end{align}
\begin{align}\label{l-ten}
\sum_{i=1}^{9} d_i \left(f_i \otimes \left( f_i \wedge \left(f_i-1\right) \right) \right) =0~\quad~{\rm in}~\quad~
\left({\CalO}^{\times}/\overline{K}^{\times}\right) \otimes \left(\left({\CalO}^{\times}/\overline{K}^{\times}\right) \wedge \left({\CalO}^{\times}/\overline{K}^{\times}\right) \right),
\end{align}
where $d_i~(i=1,\ldots,9)$ are shown in TABLE \ref{table-abcd}.
Therefore,
we have the functional equation
\cite[Theorem~5.7~$({\rm iii})_{\ell}$,~Corollary 5.8]{NW12}:
\begin{align}\label{l_ii3}
\sum_{i=1}^{9} d_i \cdot \mathcal{L}^{\varphi_{3}(f_i)_{\vec{l}}}_{\rm nv}\left(f^{}_{i}(x,y),f_i\left(\vec{v}\right);f_i \left(\gamma_{0}\right),\sigma\right)=
\sum_{i=1}^{9} d_i \cdot \varphi_{3,\vec{l}}\left(\left(\iota_{\delta_{i}} \circ f_{i\ast}\right)\left(\left[{\boldsymbol \log}\left(\left({{\mathfrak f}^{(x,y),\gamma_0}_{\vec{v},\sigma}}\right)^{-1}\right)\right]_{<3}\right)\right).
\end{align}
The right-hand side of this equation gives the ``$\ell$-adic error term'' as defined in \cite[Subsection 4.3]{NW12}.
By TABLE \ref{table-assoc} and (\ref{rhobranch2}),
the right-hand side of (\ref{l_ii3}) equals
\begin{align}\label{compute3}
&-Li_2^\ell\left(\dfrac{x(1-y)}{x-1};\gamma_5,\sigma\right)-Li_2^\ell\left(\dfrac{1-y}{y(x-1)};\gamma_7,\sigma\right)-\dfrac{1}{2}\rho_{\frac{1-xy}{1-x},\gamma'_5}(\sigma)\left(\rho_{\frac{x(1-y)^2}{y(1-x)^2},\gamma_1}(\sigma)-1\right)\\
&\quad -\dfrac{1}{6}\rho_{y,\gamma_9}(\sigma)\left(2-3\rho_{\frac{1-y}{1-x},\gamma_6}(\sigma)+3\rho_{y,\gamma_9}(\sigma)\right)-{\boldsymbol \zeta}^{\ell}_2(\sigma). \notag
\end{align}
Substituting (\ref{LHSofGalBasic2}) and (\ref{compute3}) into (\ref{l_ii3}) and applying (\ref{li-GaltoLi}), (\ref{LHSofGalBasic3}) and (\ref{rhobranch}),
we obtain the desired equation
(d-$\ell$).
In this process,
the nine $Li_2^{\ell}$ terms appearing on the left-hand side of (\ref{l_ii3}) are cancelled by using (a-$\ell$), (b-$\ell$) and (c-$\ell$).
This completes the proof of (d-$\ell$).
We will prove 
(d'-$\ell$) in Corollary \ref{l-char} later.
\end{proof}

\begin{rem}\label{imprem3}
Note that the right-hand side of (\ref{l_ii3}) is non-zero,
as computed in (\ref{compute3}),
whereas the corresponding quantity in the complex case (\ref{c-func-wii3}) vanishes identically.
This discrepancy reflects precisely the structural difference described in Remark \ref{Zl-integral}:
the $\ell$-adic error terms in (\ref{l_ii3}) are the cumulative effect of the higher-order BCH terms propagated through the non-trivial weight-2 contributions in Table \ref{table-assoc}.
 In particular,
it is this non-vanishing right-hand side that gives rise to the lower-weight terms appearing explicitly in
Theorem \ref{main2} (d-$\ell$),
which have no analogue in the complex Spence--Kummer equation (d-$\mathbb{C}$).
Furthermore, as we will demonstrate in the subsequent $\mathbb{Z}_\ell$-integrality tests (see Remark \ref{Zl-test}),
this non-vanishing cumulative error term is not merely a geometric byproduct, but a necessary arithmetic correction indispensable for preserving the strict $\mathbb{Z}_\ell$-integrality specific to the $\ell$-adic Galois setting.
\end{rem}

\begin{rem}\label{landen} \normalfont
Let us consider a special case of the $\ell$-adic Spence--Kummer equation ($d$-$\ell$) where
$x \to 0$~(i.e.,~taking $x$ as a tangential base point at 0).
We write
\begin{align} \notag
&\hat{\gamma}_{i}:=\left. \gamma_i \right|_{x \to 0} \in \pi_1^{\rm top}\left(\mathbb{P}^1(\mathbb{C}) \backslash \{0,1,\infty\}; \overrightarrow{01}, \left. f^{\an}_{i}(x,y) \right|_{x \to 0} \right),\\
&{\hat{\gamma}''}_{8}:= \delta_{\overrightarrow{0\infty}} \cdot \phi_{\overrightarrow{0\infty}}\left(\hat{\gamma}_{8}\right)
\in \pi_1^{\rm top}\left(\mathbb{P}^1(\mathbb{C}) \backslash \{0,1,\infty\}; \overrightarrow{01}, \frac{y}{y-1} \right) \notag
\end{align}
where $\phi_{\overrightarrow{0\infty}} \in {\rm Aut} \left( \mathbb{P}^1(\mathbb{C}) \backslash \{0,1,\infty\} \right)~\text{is given by}~\phi_{\overrightarrow{0\infty}}(t)=\frac{t}{t-1}.$
Then,
we obtain Landen's $\ell$-adic trilogarithm functional equation, which is equivalent to \cite[Theorem 1.1]{NS25}
\begin{align} \notag
&{Li}_{3}^{\ell}\left(y;\hat{\gamma}_8,\sigma\right)
+{Li}_{3}^{\ell}\left(1-y;{\hat{\gamma}'}_{8},\sigma\right)
+{Li}_{3}^{\ell}\left(\frac{y}{y-1};{\hat{\gamma}''}_{8},\sigma\right)\\ 
&\quad ={\boldsymbol \zeta}_{3}^{\ell}(\sigma)-{\boldsymbol \zeta}_{{2}}^\ell(\sigma)\rho_{1-y,{\hat{\gamma}'}_{8}}(\sigma)+\frac{1}{2}\rho_{y,\hat{\gamma}_{8}}(\sigma)\left({\rho_{1-y,{\hat{\gamma}'}_{8}}(\sigma)}\right)^2-\frac{1}{6}\left({\rho_{1-y,{\hat{\gamma}'}_{8}}(\sigma)}\right)^3 \\ \notag
&\qquad
-\frac{1}{2}{Li}^\ell_{2}\left(y;\hat{\gamma}_{8},\sigma\right)
-\frac{1}{12}\rho_{1-y,{\hat{\gamma}'}_{8}}(\sigma)-\frac{1}{4}\left({\rho_{1-y,{\hat{\gamma}'}_{8}}(\sigma)}\right)^2
\end{align}
by combining the $\ell$-adic Spence--Kummer equation ($d$-$\ell$) with $x \to 0$ and
the inversion formula \cite[(6.31)]{NW12}.
By a similar computation,
we can obtain Landen's complex trilogarithm functional equation \cite{L1780}, \cite[(1.3)]{NS25} from a special case of the Spence--Kummer equation ($d$-$\bC$) with
$x \to 0$ and the inversion formula \cite[(6.26)]{NW12}.
\end{rem}

\begin{cor}[Functional equations for the generalized $\ell$-adic Soul\'e character]\label{l-char}
Given a $K$-rational point
$(x,y) \in V_{\rm non\text{-}Fano}(K)$ and a path $\gamma_0 \in \pi_1^{\rm top}\left(V_{\rm non\text{-}Fano}^\an;{\vec{v}},(x,y)\right)$,
define the path system $\{\gamma_i\}_{i=1,\ldots,9}$ associated with $\gamma_0$ as in (\ref{gammai}).
For any $\sigma \in G_K$,
the following hold.
\begin{description}
   \item[(a'-$\ell$)~Integral $\ell$-adic Schaeffer equation]
\begin{align} \notag
&\tilde{\chi}_{2}^{\frac{x(1-y)}{y(1-x)},\gamma_4}(\sigma)
-\tilde{\chi}_{2}^{{y},\gamma_9}(\sigma)
+\tilde{\chi}_{2}^{{x},\gamma_8}(\sigma)
-\tilde{\chi}_{2}^{{\frac{x}{y}},\gamma_3}(\sigma)
-\tilde{\chi}_{2}^{\frac{1-y}{1-x},\gamma_6}(\sigma)\\ \notag
&\quad =
\rho_{\frac{1-y}{1-x},\gamma_6}(\sigma)\rho_{y,\gamma_9}(\sigma)
-\tilde{\chi}_{2}^{{\overrightarrow{10}},\delta_{\overrightarrow{10}}}(\sigma).
\notag
\end{align}
   \item[(b'-$\ell$)~Integral $\ell$-adic Kummer equation]
\begin{align} \notag
&\tilde{\chi}_{2}^{\frac{x(1-y)^2}{y(1-x)^2},\gamma_1}(\sigma)
-\tilde{\chi}_{2}^{\frac{x(1-y)}{x-1},\gamma_5}(\sigma)
-\tilde{\chi}_{2}^{\frac{1-y}{y(x-1)},\gamma_7}(\sigma)
-\tilde{\chi}_{2}^{\frac{x(1-y)}{y(1-x)},\gamma_4}(\sigma)
-\tilde{\chi}_{2}^{\frac{1-y}{1-x},\gamma_6}(\sigma)\\ \notag
&\quad =
\frac{1}{2}\left(\rho_{y,\gamma_9}(\sigma)\right)^2
-\frac{1}{2}\rho_{y,\gamma_9}(\sigma)
-\rho_{1-x,\gamma'_8}(\sigma)
+\rho_{1-xy,\gamma'_2}(\sigma)\\
& \quad \quad
+\left(\frac{\chi(\sigma)-1}{2}\right)\left(2\rho_{1-xy,\gamma'_2}(\sigma)-2\rho_{1-x,\gamma'_8}(\sigma)-\rho_{y,\gamma_9}(\sigma)\right).
\notag
\end{align}
   \item[(c'-$\ell$)~Integral $\ell$-adic Hill equation]
\begin{align} \notag
&
\tilde{\chi}_{2}^{\frac{1-y}{y(x-1)},\gamma_7}(\sigma)
+\tilde{\chi}_{2}^{{xy},\gamma_2}(\sigma)
-\tilde{\chi}_{2}^{{x},\gamma_8}(\sigma)
-\tilde{\chi}_{2}^{{y},\gamma_9}(\sigma)
-\tilde{\chi}_{2}^{\frac{x(1-y)}{x-1},\gamma_5}(\sigma) \notag \\
&\quad =
-\tilde{\chi}_{2}^{{\overrightarrow{10}},\delta_{\overrightarrow{10}}}(\sigma)
+\rho_{\frac{1-y}{1-x},\gamma_6}(\sigma)\rho_{y,\gamma_9}(\sigma)
-\frac{1}{2}\left(\rho_{y,\gamma_9}(\sigma)\right)^2
+\frac{1}{2}\rho_{y,\gamma_9}(\sigma) \notag \\
&\quad \quad+\left(\frac{\chi(\sigma)-1}{2}\right)\rho_{y,\gamma_9}(\sigma).
\notag
\end{align}
  \item[(d'-$\ell$)~Integral $\ell$-adic Spence-Kummer equation]
\begin{align} \notag
&\tilde{\chi}_{3}^{{\frac{x(1-y)^2}{y(1-x)^2}},\gamma_1}(\sigma)
+\tilde{\chi}_{3}^{{xy},\gamma_2}(\sigma)
+\tilde{\chi}_{3}^{{\frac{x}{y}},\gamma_3}(\sigma)
-2\tilde{\chi}_{3}^{\frac{x(1-y)}{y(1-x)},\gamma_4}(\sigma)-2\tilde{\chi}_{3}^{\frac{x(1-y)}{x-1},\gamma_5}(\sigma) \notag  \\
& \quad
-2\tilde{\chi}_{3}^{\frac{1-y}{1-x},\gamma_6}(\sigma)-2\tilde{\chi}_{3}^{\frac{1-y}{y(x-1)},\gamma_7}(\sigma)
-2\tilde{\chi}_{3}^{{x},\gamma_8}(\sigma)-2\tilde{\chi}_{3}^{{y},\gamma_9}(\sigma)+2\tilde{\chi}_{3}^{{\overrightarrow{10}},\delta_{\overrightarrow{10}}}(\sigma) \notag \\
&\quad \quad=-2\rho_{y,\gamma_9}(\sigma)^2 \rho_{\frac{1-y}{1-x},\gamma_6}(\sigma)+\left(\frac{1-\chi(\sigma)^2}{2}\right)\rho_{y,\gamma_9}(\sigma)+\frac{2}{3}\rho_{y,\gamma_9}(\sigma)^3 \notag \\
& \quad \quad \quad
+2\chi(\sigma)\left(\tilde{\chi}_{2}^{\frac{x(1-y)}{x-1},\gamma_5}(\sigma)+\tilde{\chi}_{2}^{\frac{1-y}{y(x-1)},\gamma_7}(\sigma)\right)+
\chi(\sigma)^2\rho_{\frac{1-xy}{1-x},\gamma'_5}(\sigma)-\frac{2}{3}\rho_{y,\gamma_9}(\sigma). \notag
\end{align}
\end{description}
\end{cor}

\begin{proof}
By substituting the explicit formula (\ref{explicit3}) into (a-$\ell$), (b-$\ell$), (c-$\ell$) and (d-$\ell$),
we obtain the desired equations
(a'-$\ell$), (b'-$\ell$), (c'-$\ell$) and (d'-$\ell$).
\end{proof}


\noindent
{\bf $\Z_\ell$-integrality tests:}~
The above functional equations ($a',b',c',d'$-$\ell$) enable us to check the $\bZ_\ell$-integrality of both sides of each equation.
By the definitions (\ref{gsoule}), (\ref{kummer1}) and
\[
\chi(\sigma) \equiv 1~{\rm mod}~2,
\] 
the right-hand sides of ($a',b',c'$-$\ell$) have no denominator.
The right-hand side of ($d'$-$\ell$)
is equal to
\[
-2\rho_{y,\gamma_9}(\sigma)^2 \rho_{\frac{1-y}{1-x},\gamma_6}(\sigma)-12\tilde{\chi}_{2}^{{\overrightarrow{10}},\delta_{\overrightarrow{10}}}(\sigma)\rho_{y,\gamma_9}(\sigma)
+2\chi(\sigma)\left(\tilde{\chi}_{2}^{\frac{x(1-y)}{x-1},\gamma_5}(\sigma)+\tilde{\chi}_{2}^{\frac{1-y}{y(x-1)},\gamma_7}(\sigma)\right)
\]
\[
+
\chi(\sigma)^2\rho_{\frac{1-xy}{1-x},\gamma'_5}(\sigma)
-\frac{2}{3}\rho_{y,\gamma_9}(\sigma)\left(1-\rho_{y,\gamma_9}(\sigma)\right)\left(1+\rho_{y,\gamma_9}(\sigma)\right),
\]
so it also has no denominator.

\begin{rem}\label{Zl-test}
In the formula ($d'$-$\ell$),
the term $-\frac{2}{3}\rho_{y,\gamma_9}(\sigma)$ is an $\ell$-adic error term (a lower-weight term), but the term $\frac{2}{3}\rho_{y,\gamma_9}(\sigma)^3$ is not.
As indicated by the $\Z_\ell$-integrality tests mentioned above,
these two fractional terms of different weights work together,
making 
the sum
\[
-\frac{2}{3}\rho_{y,\gamma_9}(\sigma)+\frac{2}{3}\rho_{y,\gamma_9}(\sigma)^3
\]
an $\ell$-adic integer.
This highlights the crucial role of the $\ell$-adic lower-weight error term in ensuring that both sides of the integral $\ell$-adic functional equation ($d'$-$\ell$) agree exactly,
and that the $\mathbb{Z}_\ell$-integrality of both sides is preserved.
\end{rem}



\end{document}